\documentclass[pdflatex,sn-mathphys-num]{sn-jnl}

\usepackage{graphicx}%
\usepackage{multirow}%
\usepackage{amsmath,amssymb,amsfonts}%
\usepackage{amsthm}%
\usepackage{mathrsfs}%
\usepackage[title]{appendix}%
\usepackage{xcolor}%
\usepackage{textcomp}%
\usepackage{manyfoot}%
\usepackage{booktabs}%
\usepackage{listings}%
\usepackage{subcaption} 
\usepackage{cleveref}
\usepackage[normalem]{ulem}

\usepackage{algorithm}
\usepackage{algorithmic} 
\usepackage{changes}

\theoremstyle{thmstyleone}%
\newtheorem{theorem}{Theorem}
\newtheorem{proposition}[theorem]{Proposition}%

\theoremstyle{thmstyletwo}%
\newtheorem{remark}{Remark}%

\theoremstyle{thmstylethree}%
\newtheorem{definition}{Definition}%
\newtheorem{assumption}{Assumption}

\begin{document}

\title[Parameterized Methods for Game Dynamics]{Parameterized Methods for Game Dynamics}


\author*[1]{\fnm{} \sur{Yijie Jin}}\email{yijiejin@gatech.edu}

\author[1]{\fnm{} \sur{Haomin Zhou}}\email{hmzhou@math.gatech.edu}
\equalcont{These authors contributed equally to this work.}


\affil*[1]{\orgdiv{School of Mathematics}, \orgname{Georgia Institute of Technology}, \orgaddress{\street{686 Cherry St NW}, \city{Atlanta}, \postcode{30342}, \state{Georgia}, \country{United States}}}


\abstract{
We introduce a parameterized computational framework for the evolution of strategic behavior in continuous games. We consider the collective dynamics of players through a time-dependent probability density over the strategy space, representing the likelihood of each strategy being chosen at any given time. Instead of directly solving the high-dimensional Fokker-Planck equation that governs this evolution, we represent the probability density as the pushforward of a reference distribution with parameterized pushforward maps and consider the evolution of the parameterized equations. 
The motivation for this work comes from a limitation of the parameterized Wasserstein gradient flow (PWGF) framework \cite{jin2025parameterized} when it is applied to game dynamics. PWGF provides a parameterized approach for evolution equations of probability density that can be formulated as Wasserstein gradient flows. However, not all game dynamics admit such a gradient flow formulation. 
We generalize the parameterized pushforward map framework to the non-gradient flows and apply it to the game dynamics with provable error bound in Wasserstein metric. Numerical experiments with several non-gradient systems in economics are provided to demonstrate the effectiveness of this new framework. 
}


\maketitle

\section{Introduction}\label{sec1}
In a non-cooperative game, a Nash equilibrium is a state in which no player can unilaterally improve their payoff by changing their chosen strategy. The concept of Nash equilibrium was formally proposed by John Nash \cite{nash1950equilibrium} in 1950, while the analysis of game equilibrium dates back to the time of Antonie A. Cournot \cite{cournot1838recherches}, who applied it to his competition of oligopolies in 1838. The model considers a market in which a finite number of firms produce a homogeneous good and simultaneously choose the quantities they supply.

In general, finding a Nash equilibrium in a game with multiple players or a continuous state space can be challenging. This difficulty arises from several factors. First, equilibrium analysis typically requires detailed knowledge of the game, including its structure, the players’ payoff functions, and assumptions about rationality and mutual beliefs. Second, Nash equilibria need not be unique; some games admit multiple, or even uncountably many, equilibria, making equilibrium selection a nontrivial issue. Third, even when equilibria can be identified, their strategic implications may be subtle and sometimes difficult to interpret in terms of intuitive predictions of player behavior \cite{hargreaves2004game}.

These considerations motivate the study of game dynamics, which describe how players' strategies evolve over time and help understand how equilibrium may emerge. Classical examples include replicator dynamics \cite{sandholm2010population}, best response dynamics \cite{matsui1992best, hofbauer2002global, hofbauer2007evolution}, logit dynamics \cite{hofbauer2002global, fudenberg1998theory}, and Smith dynamics \cite{smith1984stability}. Such game models provide a dynamic perspective on equilibrium computation; instead of solving directly for a Nash equilibrium, one studies the evolution of strategies under a prescribed adjustment rule and analyzes the long-time behavior of the resulting dynamical system. For example, in best response dynamics, players ``myopically'' choose their strategies that are currently optimal for them given other participants' actions. This viewpoint is particularly useful for understanding transient behaviors, stability, and equilibrium selection.

Despite their broad applicability, classical game dynamics also face limitations when applied to continuous or high-dimensional strategy spaces. Direct simulation of individual strategy trajectories captures the evolution of particular strategies, but it doesn't provide information about the likelihood of each strategy being selected. To obtain such information, it is natural to consider a distributional formulation, where strategic behavior is represented by a probability density over the continuous strategy space.

This leads to a density-level description of game dynamics. In a potential game, the incentive of all players to update their action can be described by a single global function, named the potential function \cite{rosenthal1973class, monderer1996potential}. In this case, the Fokker-Planck equation governing the evolution of the strategy density induced by the best response dynamics can be formulated as a Wasserstein gradient flow over probability measures. This connection allows methods developed for Wasserstein gradient flows, such as the parameterized Wasserstein gradient flow (PWGF) framework \cite{jin2025parameterized}, to be applied to simulate the evolution of strategy distributions and compute Nash equilibria in potential games.

However, not all games are potential games. In non-potential games, the strategic interaction generally cannot be represented as the gradient of a single potential function, and the resulting dynamics may not possess a Wasserstein gradient flow structure. Consequently, methods that rely on this structure, such as a direct application of PWGF, are no longer applicable in a straightforward way. This motivates the development of a more general parameterized framework for game dynamics that doesn't require the existence of a potential function. In this work, we consider the evolution of probability distributions over the strategy space and approximate this evolution through a parameterized pushforward map. This provides a computational approach for simulating strategic dynamics beyond the class of potential games. Besides, our framework is beneficial for sampling purpose. Once the transport map is computed, one can directly push any new initial strategy set to obtain a simulation with significantly lower cost.

The main contributions of this paper are three-folded. 
\begin{itemize}
    \item We reformulate game dynamics as McKean-Vlasov equations and propose a new method with parameterized pushforward map $T_{\theta}$ to simulate the evolution of the strategic behavior. 
    \item We derive the evolution equation for the parameter $\theta$ such that the parameterized trajectory $\pmb{y}(t)$ induced by the parameterized pushforward maps $T_{\theta}$ best approximates the exact trajectory $\pmb{x}(t)$ governed by the McKean-Vlasov equation in $L^2(\lambda)$ sense.
    \item We conduct numerical experiments on both potential and non-potential multi-player games to demonstrate the effectiveness of the proposed algorithm. 
\end{itemize}
This paper is organized as follows. In \cref{sec: literature review}, we review some related works in game theory and the computation of game. In \cref{sec: background}, we briefly review WGF \& PWGF, and rewrite the potential games in a PWGF formulation. In \cref{sec: parameterized framework}, we introduce an example of non-potential game and derive the dynamics for computing the parameter $\theta$. We also provide error analysis of our framework in Wasserstein-2 metric. In \cref{sec: Numerical Experiments}, we demonstrate our algorithm with several potential and non-potential games in high-dimensional setup. In \cref{sec: discussion}, we discuss the advantages and limitations of our framework, as well as possible directions for future work.

\section{Literature Review}
\label{sec: literature review}
Game theory is a mathematical framework for analyzing decision-making among multiple autonomous players.  Originally developed in the early 20th century, game theory has since become a foundational tool in economics, social science, computer science, evolutionary biology, and many other disciplines \cite{smith1972game, smith1973logic, smith1982evolution}. Von Neumann and Morgensten \cite{von2007theory} gave the first general mathematical formulation of game theory, focusing on two-player zero-sum games. John Nash introduced the concept of Nash equilibrium for two-player games and proved that every finite game has at least one equilibrium in mixed strategies \cite{nash1950equilibrium}.

Game dynamics are studied to understand how strategic behavior evolves over time, particularly in contexts where players or populations do not immediately reach equilibrium. Unlike static models that assume players are perfectly rational and instantly select equilibrium strategies, dynamic models offer a more realistic and algorithmic view, grounded in behavioral or biological principles \cite{hofbauer1998evolutionary, sandholm2010population, fudenberg1998theory, brown1951iterative, monderer1996fictitious, hernandez2017survey, schuster1983replicator}. 
Various game dynamics have been developed across different fields, such as replicator \cite{sandholm2010population}, best response \cite{matsui1992best, hofbauer2002global, hofbauer2007evolution}, logit \cite{hofbauer2002global, fudenberg1998theory} and Smith dynamics \cite{smith1984stability}. As the number of players increases, game dynamics are often studied through the evolution of population distributions rather than individual strategies. A similar distributional perspective also arises in mean-field games.

Mean field games are closely related to distributional formulations of game dynamics. Both frameworks characterize the aggregate behavior of a large population of interacting agents through the evolution of a time-dependent probability distribution.
In mean-field games, the distributional evolution is coupled with individual optimality conditions, leading to a system that typically consists of a Hamilton-Jacobi-Bellman equation and a Fokker-Planck equation \cite{lasry2007mean, achdou2020mean, benamou2017variational}.

\section{Preliminaries}
\label{sec: background}
In this section, we first review WGF and PWGF \cite{jin2025parameterized}, then we introduce the definition of potential games and discuss the connection between the best response dynamics and PWGF. 
\subsection{Wasserstein Gradient Flow}
Let $\mathcal{M}$ be a smooth manifold without boundary. For simplicity, we assume $\mathcal{M} = \mathbb{R}^{d}$ throughout the present work, while extension to general manifolds is straightforward. We also omit subscript $\mathcal{M}$ for all integrals unless otherwise noted. 
Denote the density function space defined on $\mathcal{M}$ as
\begin{equation}
\label{eq:PM}
    \mathcal{P}(\mathcal{M})=\Big\{\rho: \mathcal{M} \to \mathbb{R} \colon\int \rho(x)dx=1,~\rho(x)\geqslant 0,~\int |x|^2\rho(x)~ dx<\infty\Big\}.
\end{equation}
 The Wasserstein-2 distance on $\mathcal{P} (\mathcal{M})$ \cite{Lafferty,otto-PME} is defined by
\begin{equation}
\label{eq:W2}
    W(\rho_1, \rho_2) = \Big(\inf_{\pi\in \Pi (\rho_1, \rho_2)}\iint |x-y|^2d\pi(x, y) \Big)^{1/2}
\end{equation}
for any $\rho_1,\rho_2 \in \mathcal{P}(\mathcal{M})$, where $\Pi (\rho_1, \rho_2)$ is the joint density space with $\rho_1$ and $\rho_2$ as marginals.
Then $\mathcal{P}(\mathcal{M})$ becomes an infinite-dimensional Riemannian manifold with $W$ inducing its Riemannian metric, and a Wasserstein gradient flow (WGF) can be written in the following general form:
\begin{align}
\label{equ: WGF}
    \frac{\partial \rho}{\partial t}= - \textrm{grad}_{W}\mathcal{F}(\rho) = \nabla\cdot \Big(\rho(x) \nabla \frac{\delta \mathcal{F}}{\delta\rho}(x)\Big), \quad \rho(0, x)=\rho_0(x),
\end{align}
where $x\in \mathcal{M}$, $\rho_0$ is a given initial probability density, $\mathcal{F}: \mathcal{P}(\mathcal{M}) \rightarrow \mathbb{R}$ is some energy functional defined on $\mathcal{P}(\mathcal{M})$, $\textrm{grad}_{W}$ stands for the gradient of functionals on $\mathcal{P}(\mathcal{M})$ with respect to the Wasserstein metric, and $\frac{\delta \mathcal{F}}{\delta\rho}$ is the first variation of $\mathcal{F}$ in the $L^2$ sense.
We denote the tangent space of $\mathcal{P}(\mathcal{M})$ at $\rho$ by
\begin{align*}
    \mathcal{T}_\rho\mathcal{P}(\mathcal{M})=\Big\{\sigma\in C^{\infty}(\mathcal{M}) \colon \int \sigma(x)\, dx=0\Big\}, 
\end{align*}
which is the same for different $\rho$ \cite{otto-PME}.
\subsection{Pushforward Maps and PWGF}
Fix any reference probability distribution $\lambda$ that is absolutely continuous with respect to the standard Lebesgue measure $\mu$ on $\mathcal{M}$, we use $\varrho=d\lambda /d \mu$, the Radon-Nikodym derivative of $\lambda$ with respect to $\mu$, to denote the reference density determined by $\lambda$. Then for any pushforward map $T: \mathbb{R}^d\rightarrow\mathbb{R}^d$, it induces a new probability measure which is absolutely continuous to the Lebesgue measure $\mu$. Hence, there exists a probability density $T_{\sharp}\varrho$ on $\mathbb{R}^{d}$ defined by
\begin{align*}
    \int_{E} T_{\sharp}\varrho (x) \,d\mu(x) =\int_{T^{-1}(E)} \varrho(z) \,d\mu(z) & = \int_{T^{-1}(E)} d\lambda(z) = \lambda (T^{-1}(E))\\
    & \quad \textrm{ for all measurable } E\subset \mathbb{R}^d,
\end{align*}
where $T^{-1}(E)$ is the pre-image of $E$. Hereafter we use $dx$ instead of $d\mu(x)$ to reduce notation complexity.
Let us take $T_{\theta}$ as a parameterized map, namely for any $\theta\in \Theta$, $T_{\theta}: \mathbb{R}^d\rightarrow\mathbb{R}^d$ is a parametric function with parameter $\theta$. 
The parameterized Wasserstein gradient flow (PWGF) can be expressed as
\begin{align}
\label{eq:PFPE}
    \dot\theta = -\widehat{G}(\theta)^{\dagger}\nabla_{\theta} F(\theta),
\end{align}
where $F(\theta):=\mathcal{F}(\rho_{\theta})$ and $\widehat{G}(\theta)^{\dagger}$ is the minimal norm pseudo inverse of $\widehat{G}(\theta)$ with
\begin{align}
        \widehat{G}(\theta)=\int  \partial_{\theta} T_\theta(z)^{\top} \partial_{\theta} T_\theta(z)~d\lambda(z).\label{relaxed metric tensor}
\end{align}
\subsection{Potential Game}
\begin{definition}[Exact Potential Game \cite{monderer1996potential}]
    Consider a game $\mathcal{G}$ with $d$ players, denote the strategy space by $\mathbb{S} = \mathbf{S}_1 \times \mathbf{S}_2 \times \cdots \times \mathbf{S}_d$ and payoff functions by $\Pi_i, \forall i \in [d] = \{1, \cdots, d\}$. The game $\mathcal{G}$ is an exact potential game if and only if a potential function $F: \mathbb{S} \rightarrow \mathbb{R}$ exists such that, $\forall i \in [d]$, 
    \begin{align}
        \Pi_i(x_i, y_{-i}) - \Pi_i(y_i, y_{-i}) = F(x_i, y_{-i}) - F(y_i, y_{-i}), \quad \forall x_i, y_i \in \mathbf{S}_i; \forall y_{-i} \in \mathbf{S}_{-i}
    \end{align}
    where $y_{-i}$ is the notation for all players except for $y_i$ and $\mathbf{S}_{-i}$ is joint strategy space of all players other than $y_i$.
\end{definition}
\begin{remark}
    Suppose each strategy set $\mathbf{S}_i \subset \mathbb{R}$ is a continuous interval and each payoff function $\Pi_i \in C^1(\mathbb{R})$, then $\mathcal{G}$ is a continuous game. For such a game to be an exact potential game, an equivalent definition is, $\forall i \in [d]$, 
    \begin{align}
        \frac{\partial \Pi_i(x_i, x_{-i})}{\partial x_i} = \frac{\partial F(x_i, x_{-i})}{\partial x_i}, \quad \forall x_i \in \mathbf{S}_i, x_{-i} \in \mathbf{S}_{-i}.
    \end{align}
\end{remark}
\subsection{Connection between the best response dynamic and PWGF}
Consider a game with $d$ players and a continuous strategy space $\mathbb{S} = \mathbf{S}_1 \times \mathbf{S}_2 \times \cdots \times \mathbf{S}_d = \mathbb{R}^d$. Denote the strategy of the players by $\widetilde{\pmb{x}} = (\widetilde{x}_1, \widetilde{x}_2, ..., \widetilde{x}_d)^T$. For any strategy $\widetilde{\pmb{x}} \in \mathbb{S}$, denote the probability density of $\widetilde{\pmb{x}}$ being the strategy at time $t$ by $\rho(\widetilde{\pmb{x}},t)$. Suppose the payoff function of each player is $\Pi_i(\widetilde{x}_i, \widetilde{x}_{-i})$ for $i \in [d]$. 
The deterministic best response dynamics \cite{blume1993statistical, benaim1999mixed} is 
\begin{align}
    \label{eq: deterministic_BRD}
    \frac{d \widetilde{x}_i(t)}{dt} = \frac{\partial \Pi_i(\widetilde{x}_i, \widetilde{x}_{-i})}{\partial \widetilde{x}_i}, \quad \text{for all } i \in [d]
\end{align}
and the perturbed (stochastic) version \cite{hofbauer2002global, hofbauer2007evolution} is 
\begin{align}
    \label{eq: stochastic_BRD}
    d \widetilde{x}_i(t) = \frac{\partial \Pi_i(\widetilde{x}_i, \widetilde{x}_{-i})}{\partial \widetilde{x}_i} dt + \sigma_i \, dW_t^{(i)}, \quad \text{for all } i \in [d]
\end{align}
where $\sigma_i \in \mathbb{R}$ and $\{dW_t^{(i)}\}_{i=1}^d$ is a set of independent standard Wiener processes. The stochastic best response dynamics is considered as players don't always make rational decisions, and the diffusion term models the behavior of random decisions. Meanwhile, diffusion helps the dynamics to jump out of unstable Nash equilibria, if there are any. Throughout this work, unless otherwise specified, we use best response dynamics to refer to its stochastic version.

The evolution of probability density $\rho(\widetilde{\pmb{x}},t)$ can be described by the following Fokker-Planck equation
\begin{align}
    \label{eq: game_FPE}
    \frac{\partial \rho(\widetilde{\pmb{x}}, t)}{\partial t} & = - \nabla \cdot \left[ \pmb{\mu}(\widetilde{\pmb{x}},t)
    \rho(\widetilde{\pmb{x}},t) \right] + \frac{1}{2}  D \Delta_{\widetilde{\pmb{x}}} \left( \rho(\widetilde{\pmb{x}},t) \right) \nonumber \\
    & = - \sum_{i=1}^d \frac{\partial}{\partial \widetilde{x}_i} \left( \rho(\widetilde{\pmb{x}}, t) \frac{\partial \Pi_i(\widetilde{x}_i, \widetilde{x}_{-i})}{\partial \widetilde{x}_i }\right) + \frac{1}{2} \sum_{i,j = 1}^d \frac{\partial^2}{\partial \widetilde{x}_i \partial \widetilde{x}_j} \left( D_{ij} \rho(\widetilde{\pmb{x}},t) \right).
\end{align}
where $\pmb{\mu} = \left( \frac{\partial \Pi_1(\widetilde{x}_1, \widetilde{x}_{-1})}{\partial \widetilde{x}_1}, \frac{\partial \Pi_2(\widetilde{x}_2, \widetilde{x}_{-2})}{\partial \widetilde{x}_2}, ..., 
\frac{\partial \Pi_d(\widetilde{x}_d, \widetilde{x}_{-d})}{\partial \widetilde{x}_d} \right)^T$ is the drift function and $D = [D_{i,j}]$ is the $d \times d$ diffusion matrix with $D = \pmb{\sigma}\pmb{\sigma}^T, \pmb{\sigma} = \left( \sigma_1, \sigma_2, ..., \sigma_d \right)^T$. If $\pmb{\mu}$ is a gradient field, then \eqref{eq: game_FPE} can be written as a WGF. Therefore, PWGF can be directly applied to compute the approximated dynamics for $\widetilde{\pmb{x}}(t)$.

If the payoff function is not the same for each player, this game is not necessarily a potential game which imposes challenges in finding the Nash equilibrium. Meanwhile, as $\pmb{\mu}(\widetilde{\pmb{x}},t)$ is not a gradient field, the Fokker-Planck equation defined above loses the gradient flow structure and cannot be fit into the computational framework introduced in PWGF \cite{jin2025parameterized} directly. 
\begin{theorem}
    In general, for distinct payoff functions $\Pi_i(\widetilde{x}_i, \widetilde{x}_{-i}), i \in [d]$, the Fokker-Planck equation \eqref{eq: game_FPE} is not a Wasserstein gradient flow. 
\end{theorem}
\begin{proof}
    We prove by contradiction. Suppose there exists a functional $\mathcal{F}(\rho)$ s.t. 
    \begin{align}
        \frac{\partial \rho}{\partial t} = - \nabla \cdot \left( \rho \nabla \frac{\partial \mathcal{F}}{\partial \rho}\right),
    \end{align}
    then $\left( \frac{\partial \Pi_1(x_1, x_{-1})}{\partial x_1}, \frac{\partial \Pi_2(x_2, x_{-2})}{\partial x_2}, ..., \frac{\partial \Pi_d(x_d, x_{-d})}{\partial x_d} \right)^T$ needs to be a gradient field of a scalar function. However, it doesn't hold for distinct $\Pi_i(x_i, x_{-i})$ in general. Hence, \eqref{eq: game_FPE} is not a Wasserstein gradient flow. 
\end{proof}
It motivates us to explore alternative methods for addressing non-potential games in high-dimensional contexts. In this work, we derive a parameterized equation for computing the game dynamics with the aid of its McKean-Vlasov stochastic differential equation. Since the derivation doesn't require the energy functional $\mathcal{F}(\cdot)$ as it does in PWGF, it's applicable to evolutional equation for probability densities as long as a McKean-Vlasov type SDE in particle level is known. Meanwhile, this new framework is compatible with PWGF in the following sense, if $\pmb{\mu} = \begin{pmatrix}
        \frac{\partial \Pi_1(x_1, x_{-1})}{\partial x_1}, ..., 
        \frac{\partial \Pi_d(x_d, x_{-d})}{\partial x_d}
    \end{pmatrix}^T$ is a gradient field, as in potential games, then \eqref{eq: game_FPE} reduces to a Wasserstein gradient flow.

\section{Parameterized Framework}
\label{sec: parameterized framework}

In this section, we first present a non-potential game example that cannot be formulated as a WGF. We then derive the parameterized evolution equation for computing the dynamics of non-potential games based on the McKean-Vlasov stochastic differential equation. Finally, we establish an error estimate in the Wasserstein-2 metric. 

\subsection{Non-potential Cournot Duopoly Game}
\label{cournot_duopoly}
The Cournot duopoly game is an economy model describing the industry structure of two companies competing on the quantity of the product they produce \cite{kopel1996simple}. Consider two companies that produce $x_1(t)$ and $x_2(t)$ amounts of product at time $t$, respectively. Both would like to maximize their payoff functions $\Pi_1(x_1, x_2)$ and $\Pi_2(x_1, x_2)$. In the Cournot game, their payoffs are simply modeled by the profit which is the income minus the cost. The income is the unit price $p$ multiplied by the amount of the product $x_1$ or $x_2$, and the cost is $c_1$ and $c_2$ respectively. Together, the payoff for both companies are
\begin{align}
    \Pi_1(x_1, x_2) = p(x_1, x_2) x_1 - c_1(x_1, x_2), \\
    \Pi_2(x_1, x_2) = p(x_1, x_2) x_2 - c_2(x_1, x_2).
\end{align}
The unit price can be modeled as the inverse linear function indicating the product price is decreasing as more product is available, e.g., 
\begin{align}
    p(x_1, x_2) = a - b (x_1 + x_2), \quad a, b > 0.
\end{align}
The cost functions $c_1(x_1, x_2)$ and $c_2(x_1, x_2)$ are given by
\begin{align}
    c_1(x_1, x_2) = d + ax_1 - bx_1x_2(1+2\mu) + 2b\mu x_1 x_2^2, \\
    c_2(x_1, x_2) = d + ax_2 - bx_1x_2(1+2\mu) + 2b\mu x_2 x_1^2.
\end{align}
where constant term $d$ represents the fixed cost of production, the linear term $a x_i$ captures the direct cost associated with firm $i$'s own output, the interaction term $-b x_1 x_2(1+2\mu)$ reflects a ``book-buying habit'' effect, i.e., moderate competitor
output may create market awareness, buying habits, or distribution channels, thereby reducing firm $i$'s effective cost. The nonlinear term $2b\mu x_1 x_2^2$ captures the offsetting crowding effect when the competitor's output becomes too large. Thus, the competitor's output has a non-monotone effect on firm $i$'s marginal cost, with $\mu$ controlling the strength of this externality.
The strategies they can adjust are the quantities $x_1$ and $x_2$ they produce. To determine the quantity at time $t$, the firms form expectations on the quantity of the other firm, which might depend on their own quantity and the quantity of the other firm, both produced in the previous period.

The number of Nash equilibria varies with different parameter sets for $\mu$. If $\mu \in (0, 4]$, there are four Nash equilibria 
\begin{align*}
    (0,0), \left( \frac{\mu - 1}{\mu}, \frac{\mu - 1}{\mu} \right), \left(x_1^*, x_2^* \right), (x_2^*, x_1^*)
\end{align*}
where $x_1^* = \frac{\mu + 1 - \sqrt{(\mu + 1)(\mu - 3)}}{2\mu}$ and $x_2^* = \frac{\mu + 1 + \sqrt{(\mu + 1)(\mu - 3)}}{2\mu}$. Note that the last two Nash equilbria can be complex for some $\mu$ and we only consider the real Nash equilibria in numerical experiments. 
More examples are presented in \cref{sec: Numerical Experiments}.

\subsection{Derivation of the Parameterized Equation}
Consider the continuous game introduced in the previous section. The stochastic best response dynamics are described in \eqref{eq: stochastic_BRD} and the Fokker-Planck equation describing the evolution of the probability density $\rho(\pmb{x},t)$ is \eqref{eq: game_FPE}. The McKean-Vlasov type SDEs corresponding to \eqref{eq: game_FPE} is
\begin{align}
    \label{eq: Vlasov-game}
    \frac{\mathrm{d} x_i(t)}{\mathrm{d} t} = \frac{\partial \Pi_i(x_i, x_{-i})}{\partial x_i} - \frac{1}{2} \sigma_i^2 \frac{\partial \log \rho(\pmb{x},t)}{\partial x_i}, \quad \forall i \in [d] 
\end{align}
Note that \eqref{eq: stochastic_BRD} and \eqref{eq: Vlasov-game} matches in density level, in the sense that the dynamics of the probability distribution $\rho(\pmb{x}, t)$ and $\rho(\widetilde{\pmb{x}}, t)$ can both be described by \eqref{eq: game_FPE}. 
Hence, we can take advantage of the equivalence of the two SDEs in density level and derive the parameterized equation, i.e., the dynamics of $\theta(t)$ accordingly.

\begin{definition}[Pushforward Map $T$]
    For any $\pmb{x}_0 \sim \lambda$, consider $\pmb{x}(0) = \pmb{x}_0$ and $\pmb{x}(t) \in \mathbb{R}^d, t \in [0, \infty)$ follows the McKean-Vlasov equation \eqref{eq: Vlasov-game}; $\widetilde{\pmb{x}}(0) = \pmb{x}_0$ and $\widetilde{\pmb{x}}(t) \in \mathbb{R}^d, t \in [0, \infty)$ follows the stochastic best response dynamics \eqref{eq: stochastic_BRD}. $\pmb{x}(t)$ and $\widetilde{\pmb{x}}(t)$ match in density level, i.e., $\rho(\widetilde{\pmb{x}}, t)$ and $\rho(\pmb{x}, t)$ follow the same Fokker-Planck equation as defined in \eqref{eq: game_FPE}. The pushforward map $T(\cdot): \mathbb{R}^d \rightarrow \mathbb{R}^d$ is defined such that $\pmb{x}(t) = T(\pmb{x}(0)), \forall t \in [0, \infty)$.
\end{definition}
In the following proposition, we show the derivation of the parameterized equation \eqref{eq: game_parameter_equation}. The idea of this proposition is as follows. For the pushforward map $T$ such that $\pmb{x}(t) = T( \pmb{x}(0))$ denotes the dynamics of the strategies, consider an approximated trajectory $\pmb{y}(t) = T_{\theta(t)}(\pmb{x}(0))$ with the parameterized pushforward map $T_{\theta(t)}$. We derive the dynamics of the parameter $\theta$ by calculating the best approximation $\pmb{y}(t)$ for the actual dynamics $\pmb{x}(t)$ in $L^2(\lambda)$ sense. 

\begin{proposition}
    Consider a continuous game with $d$ players, suppose the strategies evolve according to the stochastic best response dynamics \eqref{eq: Vlasov-game} and the strategies are approximated with the pushforward map $T_{\theta(t)}$ in $L^2(\lambda)$ sense, then the parameter equation for the game is 
    \begin{align}
        \label{eq: game_parameter_equation}
        \dot \theta(t) = \widehat{G}(\theta)^{\dagger} \int \left[ \frac{\partial T_{\theta} (z)}{\partial \theta} \right]^T 
        \left[
        \begin{pmatrix}
            \frac{\partial \Pi_1(T_{\theta}(\mathbf{z)}_1, T_{\theta}(\mathbf{z)}_{-1})}{\partial x_1} \\
            \vdots \\
            \frac{\partial \Pi_d(T_{\theta}(\mathbf{z)}_d, T_{\theta}(\mathbf{z)}_{-d})}{\partial x_d}
        \end{pmatrix} - 
        \frac{1}{2} D \nabla_{\pmb{x}} \log \rho\left(T_{\theta}(\mathbf{z}),t\right) 
        \right] d\lambda(\mathbf{z})
    \end{align}
    where $\widehat{G}(\theta)$ is the pullback Wasserstein metric on the parameter space and $\widehat{G}(\theta)^{\dagger}$ is its Moore-Penrose pseudo inverse introduced in \cite{jin2025parameterized, wu2025parameterized}. $D = \pmb{\sigma} \pmb{\sigma}^T$ is the diffusion tensor. $\lambda$ is the probability measure of the initial distribution such that $d\lambda / d\mu = p(\cdot, 0)$ and $z$ is randomly drawn from $\lambda$.   
\end{proposition}

\begin{proof}
    Recall the kernel operator $\mathcal{K}_{\theta}$ defined in parameterized Wasserstein Hamiltonian flow and gradient flow \cite{wu2025parameterized, jin2025parameterized}, for any $f \in L^2(\mathcal{M}; \mathcal{M}, \lambda)$, $\mathcal{K}_{\theta}[f] \in L^2(\mathcal{M}; \mathcal{M}, \lambda)$ is defined as 
    \begin{align}
        \mathcal{K}_{\theta}[f] (\cdot) =  \frac{\partial T_{\theta}(\cdot)}{\partial \theta} \widehat{G}(\theta)^{\dagger} \int \left[ \frac{\partial T_{\theta}(z)}{\partial \theta} \right]^T f(z) \, d \lambda(z) = \int K_{\theta} (\cdot, z) f(z) \, d\lambda(z) 
    \end{align}
    where the kernel matrix $K_{\theta}(z', z) \in \mathbb{R}^{d \times d}$ is defined by 
    \begin{align}
        K_{\theta}(z', z) := \left[\frac{\partial T_{\theta}(z')}{\partial \theta} \right] \widehat{G}(\theta)^{\dagger} \left[ \frac{\partial T_{\theta}(z)}{\partial \theta} \right]^T 
    \end{align}
    Recall the McKean-Vlasov equation \eqref{eq: Vlasov-game}. Define a pushforward map $T$ such that $\pmb{x}(t) = T( \pmb{x}(0))$. Plug it into \eqref{eq: Vlasov-game} and write in the vector form, we obtain
    \begin{align}   
        \label{eq: game_Vlasov_x_pushforward}
        \frac{d}{dt} \pmb{x}(t) = 
        \begin{pmatrix}
            \frac{\partial \Pi_1(x_1, x_{-1})}{\partial x_1} \\
            \vdots \\
            \frac{\partial \Pi_d(x_d, x_{-d})}{\partial x_d}
        \end{pmatrix} 
        - \frac{1}{2} D \nabla_{\pmb{x}} \log \rho(T(\pmb{x}(0)), t)
    \end{align}
    Consider a parameterized pushforward map $T_{\theta(t)}$ which approximates the pushforward map $T$. Denote $\pmb{y}(t) = T_{\theta(t)}(\pmb{x}(0))$ the pushforward point by the parameterized map $T_{\theta(t)}$. Then $\pmb{y}(t)$ forms another curve in $\mathbb{R}^d$ satisfying 
    \begin{align}
        \label{eq: game_y_pushforward}
        \frac{d}{dt} \pmb{y}(t) = \frac{d}{dt} \left[ T_{\theta(t)}(\pmb{x}(0)) \right] =  \left[ \frac{\partial T_{\theta(t)} (\pmb{x}(0))}{\partial \theta} \right] \dot\theta(t), \quad \pmb{y}(0) = T_{\theta(0)} (\pmb{x}(0))
    \end{align}
    The idea of this proof is to derive $\theta(t)$ such that $\pmb{y}(t)$ is the least square approximation of the dynamics $\pmb{x}(t)$ on the tangent space spanned by $\frac{\partial T_{\theta}}{\partial \theta}$. 
    Assume that $\pmb{y}(t)$ follows 
    \begin{align}
        \label{eq: game_y_Vlasov}
        \frac{d}{dt} \pmb{y}(t) =  
        \begin{pmatrix}
            \frac{\partial \Pi_1(T_{\theta(t)}(\pmb{x}(0))_1, T_{\theta(t)}(\pmb{x}(0))_{-1})}{\partial x_1} \\
            \vdots \\
            \frac{\partial \Pi_d(T_{\theta(t)}(\pmb{x}(0))_d, T_{\theta(t)}(\pmb{x}(0))_{-d})}{\partial x_d}
        \end{pmatrix} 
        - \frac{1}{2} D \nabla_{\pmb{x}} \log \rho(T_{\theta(t)} (\pmb{x}(0)), t)
    \end{align}
    Combining \eqref{eq: game_y_pushforward} and \eqref{eq: game_y_Vlasov}, we obtain 
    \begin{align}
        \label{eq: game_combined}
        \frac{d}{dt} \pmb{y}(t) & = \left[ \frac{\partial T_{\theta(t)} (\pmb{x}(0))}{\partial \theta} \right] \dot\theta(t) \nonumber\\
        & = \begin{pmatrix}
            \frac{\partial \Pi_1(T_{\theta(t)}(\pmb{x}(0))_1, T_{\theta(t)}(\pmb{x}(0))_{-1})}{\partial x_1} \\
            \vdots \\
            \frac{\partial \Pi_d(T_{\theta(t)}(\pmb{x}(0))_d, T_{\theta(t)}(\pmb{x}(0))_{-d})}{\partial x_d}
        \end{pmatrix} - \frac{1}{2} D \nabla_{\pmb{x}} \log \rho(T_{\theta(t)} (\pmb{x}(0)), t)
    \end{align}
    Now we multiply $\left[ \frac{\partial T_{\theta(t)}}{\partial \theta} \right]^T$ to the left on both sides of \eqref{eq: game_combined} and take integration w.r.t. $\mathrm{d} \lambda(\mathbf{z})$, we obtain the following equation in distribution sense, 
    \begin{align}
        \nonumber
        & \quad \underbrace{\left( \int \left[ \frac{\partial T_{\theta(t)} (\mathbf{z})}{\partial \theta} \right]^T \left[ \frac{\partial T_{\theta(t)} (\mathbf{z})}{\partial \theta} \right] \mathrm{d} \lambda(\mathbf{z}) \right)}_{\widehat{G}(\theta)} \, \dot\theta(t) \\ \nonumber
        & = \int \left[ \frac{\partial T_{\theta(t)} (\cdot)}{\partial \theta} \right]^T \left[ 
            \begin{pmatrix}
                \frac{\partial \Pi_1 \left(T_{\theta(t)}(\cdot)_1, T_{\theta(t)}(\cdot)_{-1} \right)}{\partial x_1} \\
                \vdots \\
                \frac{\partial \Pi_d \left(T_{\theta(t)}(\cdot)_d, T_{\theta(t)}(\cdot)_{-d} \right)}{\partial x_d}
            \end{pmatrix}  
             - \frac{1}{2} D \nabla_{\pmb{x}} \log \rho(T_{\theta(t)} (\cdot), t) \right] \circ \mathbf{z} \, \mathrm{d} \lambda(\mathbf{z})
    \end{align}
    Multiply the Moore-Penrose pseudo inverse of $\widehat{G}(\theta)$ to the left on both side, we obtain the new parameterized equation
    \begin{align}
        \dot\theta(t) & = \widehat{G}(\theta)^{\dagger} \int \left[ \frac{\partial T_{\theta(t)} (\cdot)}{\partial \theta} \right]^T \left[ 
            \begin{pmatrix}
                \frac{\partial \Pi_1 \left(T_{\theta(t)}(\cdot)_1, T_{\theta(t)}(\cdot)_{-1} \right)}{\partial x_1} \\
                \vdots \\
                \frac{\partial \Pi_d \left(T_{\theta(t)}(\cdot)_d, T_{\theta(t)}(\cdot)_{-d} \right)}{\partial x_d}
            \end{pmatrix} \right. \nonumber\\
            & \quad\quad\quad\quad\quad \left.
             - \frac{1}{2} D \nabla_{\pmb{x}} \log \rho(T_{\theta(t)} (\cdot), t) \right] \circ \mathbf{z} \, \mathrm{d} \lambda(\mathbf{z})
    \end{align}
    The dynamics of $\pmb{y}(t)$ is 
    \begin{align}
        \frac{d}{dt} \pmb{y}(t) & = \left[ \frac{\partial T_{\theta(t)} (\pmb{x}(0))}{\partial \theta} \right] \dot\theta(t) \nonumber \\
        & = \left[ \frac{\partial T_{\theta(t)} (\pmb{x}(0))}{\partial \theta} \right] \widehat{G}(\theta)^{\dagger} \int \left[ \frac{\partial T_{\theta(t)} (\cdot)}{\partial \theta} \right]^T \nonumber \\
        & \quad\quad\quad \left[ 
            \begin{pmatrix}
                \frac{\partial \Pi_1 \left(T_{\theta(t)}(\cdot)_1, T_{\theta(t)}(\cdot)_{-1} \right)}{\partial x_1} \\
                \vdots \\
                \frac{\partial \Pi_d \left(T_{\theta(t)}(\cdot)_d, T_{\theta(t)}(\cdot)_{-d} \right)}{\partial x_d}
            \end{pmatrix} 
             - \frac{1}{2} D \nabla_{\pmb{x}} \log \rho(T_{\theta(t)} (\cdot), t) \right] \circ \mathbf{z} \, \mathrm{d} \lambda(\mathbf{z}) \nonumber \\
        & = \mathcal{K}_{\theta}[f(T_{\theta}; D, \Pi_1, ..., \Pi_d)(\cdot)](\pmb{x}(0))
    \end{align}
    where
    \begin{align}
        \label{eq: game_rhs_vector_in_parameter_eq}
        f(T_{\theta}; D, \Pi_1, ..., \Pi_d)(\cdot) = \begin{pmatrix}
            \frac{\partial \Pi_1 \left(T_{\theta(t)}(\cdot)_1, T_{\theta(t)}(\cdot)_{-1} \right)}{\partial x_1} \\
            \vdots \\
            \frac{\partial \Pi_d \left(T_{\theta(t)}(\cdot)_d, T_{\theta(t)}(\cdot)_{-d} \right)}{\partial x_d}
        \end{pmatrix} 
         - \frac{1}{2} D \nabla_{\pmb{x}} \log \rho(T_{\theta(t)} (\cdot), t)
    \end{align}
\end{proof}

\begin{remark}
    If \eqref{eq: game_FPE} is indeed a WGF with an energy functional $\mathcal{F}(\rho)$, then \eqref{eq: game_parameter_equation} will degenerate to the PWGF described in \cite{jin2025parameterized}
    \begin{align}
        \dot\theta(t) = -\widehat{G}(\theta)^{\dagger} \nabla_{\theta} F(\theta)
    \end{align}
    where $F(\theta) = \mathcal{F}(\rho_{\theta}) = \mathcal{F}(T_{\theta\sharp} \varrho)$.
\end{remark}

\subsection{Error Analysis in Wasserstein-2 Metric}
In this section, we develop an error estimate in Wasserstein-2 distance between the true probability density $\rho(\cdot, t)$ and the approximated probability density $T_{\theta(t) \sharp} \rho(\cdot, 0)$. Consider $f(T_{\theta}; D, \Pi_1, ..., \Pi_d)(\cdot)$ defined in \eqref{eq: game_rhs_vector_in_parameter_eq}. 
We assume that the pushforward map satisfies the following Lipschitz condition. 
\begin{assumption}
    \label{assumption: game_projection_error}
    There exists a constant $C > 0$ such that for any two pushforward maps $T$ and $\widetilde{T}$, there is 
    \begin{align}
        \int \left| f(T; D, \Pi_1, ..., \Pi_d)(z) - f(\widetilde{T}; D, \Pi_1, ..., \Pi_d)(z)  \right|^2 d \lambda(z) \leqslant C \int |T(z) - \widetilde{T}(z)|^2 d \lambda(z)
    \end{align}
    where $f(T; D, \Pi_1, ..., \Pi_d)$ is defined in \eqref{eq: game_rhs_vector_in_parameter_eq}. 
\end{assumption}

\begin{definition}
    Define the projection error as 
    \begin{align}
        \delta_0 & = \sup_{\theta \in \Theta} \min_{\xi \in \mathcal{T}_{\theta}^* \Theta} \left\{ \int \left| f(T_{\theta}; D, \Pi_1, ..., \Pi_d)(z) - \frac{\partial T_{\theta}(z)}{\partial \theta} \xi \right|^2 d \lambda(z) \right\} \\
        & = \sup_{\theta \in \Theta} \left\{ \int \left| f(T_{\theta}; D, \Pi_1, ..., \Pi_d)(z) - \mathcal{K}_{\theta(t)} \left[ f(T_{\theta}; D, \Pi_1, ..., \Pi_d)(z) \right] \right|^2 d \lambda(z) \right\}
    \end{align}
\end{definition}
The error is essentially the difference between $f(T_{\theta}; D, \Pi_1, ..., \Pi_d) [\cdot]$ and its orthogonal projection onto the tangent space spanned by $\mathcal{T}_{\rho_{\theta}} \mathcal{P}(\mathbb{R}^d)$ where $\rho_{\theta}= T_{\theta \sharp} \rho(\cdot, 0)$. 

\begin{theorem}
    Suppose \cref{assumption: game_projection_error} holds, $\theta$ evolves following \eqref{eq: game_parameter_equation} with initial value $\theta(0)$, and the probability density $\rho(\cdot, t)$ follows \eqref{eq: game_FPE}. Then the Wasserstein-2 distance between the pushforward probability density $\rho_{\theta(t)}(\cdot)$ and the true density $\rho(\cdot,t)$ satisfies
    \begin{align}
        W_2^2(\rho_{\theta(t)}(\cdot), \rho(\cdot,t)) \leqslant e^{(1+2 C)t} \epsilon_{\rho(\cdot, 0)} + \frac{2\delta_0}{1+2C} \left( e^{(1+2C)t} -1 \right)
    \end{align}
    where $\epsilon_{\rho(\cdot,0)} = W_2^2 (\rho_{\theta(0)}, \rho(\cdot, 0))$ is the initial approximation error. 
\end{theorem}

\begin{proof}
    The proof essentially follows the idea of the error estimate for PWGF \cite{jin2025parameterized}.
    Assume the strategy profile $\pmb{x}(t)$ follows  \eqref{eq: Vlasov-game} and the dynamics of the approximation $\pmb{y}(t)$ is described by \eqref{eq: game_y_Vlasov}. Suppose the Monge map from $\rho_{\theta(0)}$ to $\rho(\cdot, 0)$ is given by $\omega$ and we assume $\pmb{x}(t), \pmb{y}(t)$ are coupled via $\pmb{x}(0) = \omega(\pmb{y}(0))$. Reformulate $\frac{d}{dt} \pmb{y}(t)$ as follows
    \begin{align}
        \frac{d}{dt} \pmb{y}(t) & = f(T_{\theta}; D, \Pi_1, ..., \Pi_d)(z) \nonumber \\
        & \quad\quad + \left( \mathcal{K}_{\theta}[f(T_{\theta}; D, \Pi_1, ..., \Pi_d)](z) - f(T_{\theta}; D, \Pi_1, ..., \Pi_d)(z) \right)
    \end{align}
    Denote $E(t) = \mathbb{E} \|\pmb{x}(t) - \pmb{y}(t) \|^2$, we compute its derivative,
    \begin{align}
        \label{eq: game_error_estimate_error_derivative}
        \frac{d}{dt} E(t) & = 2 \mathbb{E} \left[ (\pmb{x}(t) - \pmb{y}(t)) \cdot \left( \frac{d}{dt} \pmb{x}(t) - \frac{d}{dt} \pmb{y}(t)\right)\right] \\
        & \leqslant 2 \sqrt{\mathbb{E} \|\pmb{x}(t) - \pmb{y}(t) \|^2} \sqrt{\mathbb{E} \left\| \frac{d}{dt} \pmb{x}(t) - \frac{d}{dt} \pmb{y}(t) \right\|^2} \\
        & \leqslant \mathbb{E} \|\pmb{x}(t) - \pmb{y}(t) \|^2 + \mathbb{E} \left\| \frac{d}{dt} \pmb{x}(t) - \frac{d}{dt} \pmb{y}(t) \right\|^2
    \end{align}
    Note that $\pmb{x}(t)$ follows the McKean-Vlasov equation \eqref{eq: game_Vlasov_x_pushforward}, we have an estimate of the second term in the previous inequality,
    \begin{align}
        \label{eq: game_error_estimate_second_term}
        & \quad \mathbb{E} \left\| \frac{d}{dt} \pmb{x}(t) - \frac{d}{dt} \pmb{y}(t) \right\|^2 \nonumber \\
        & = \int \left| \left(f(T; D, \Pi_1, ..., \Pi_d)(z) - f(T_{\theta}; D, \Pi_1, ..., \Pi_d)(z)\right. \right) \nonumber \\
        & \quad \left. + \left(f(T_{\theta}; D, \Pi_1, ..., \Pi_d)(z) - \mathcal{K}_{\theta(t)} \left[ f(T_{\theta}; D, \Pi_1, ..., \Pi_d) \right](z)\right) \right|^2 d \lambda(z) \nonumber \\
        & \leqslant 2 \delta_0 + 2 C \mathbb{E} \|\pmb{x}(t) - \pmb{y}(t) \|^2
    \end{align}
    Plug \eqref{eq: game_error_estimate_second_term} into \eqref{eq: game_error_estimate_error_derivative} and apply Gr\"onwall's inequality, we obtain 
    \begin{align}
        E(t) \leqslant e^{(1+2 C) t} E(0) + \frac{2 \delta_0}{1+2 C} \left( e^{(1 + 2 C)t} - 1 \right)
    \end{align}
    Hence, we have 
    \begin{align}
        W_2^2(\rho_{\theta(t)}(\cdot), \rho(\cdot,t))  & = W_2^2 (\mathrm{Law}(\pmb{x}_t), \mathrm{Law}(\pmb{y}_t)) \nonumber \\
        & \leqslant \mathbb{E} \| \pmb{x}_t - \pmb{y}_t \|^2 = E(t) \nonumber \\
        & \leqslant e^{(1+2 C)t} \epsilon_{\rho(\cdot, 0)} + \frac{2 \delta_0}{1+ 2 C} \left( e^{(1+2 C)t} - 1 \right)
    \end{align}
\end{proof}

\section{Numerical Experiments}
\label{sec: Numerical Experiments}
In this section, we illustrate the power of the parameterized equation with several examples. In all examples for game dynamics, the pushforward map is represented by a Neural ODE with the structure 
\begin{align}
    T_{\theta} = \mathrm{NODE} (R_{\theta})
\end{align}
where $R_{\theta}$ is a multilayer perceptron with $3$ hidden layers. Each layer has $64$ nodes, except for the input and output layers. The experiments were conducted on a NVIDIA T4 GPU with $16$ GB memory. The algorithm to compute $\rho_{\theta(t)}$ with the dynamics of $\theta(t)$ is shown in \cref{alg: game}. 
In our algorithm, the parameter dynamics are computed by solving the normal equations associated with $\widehat{G}(\theta)$ using the minimal residual method. The parameterized pushforward map $T_{\theta}$ used in our experiments contains approximately $10,000$ parameters, with slight variations depending on the dimension $d$. We didn't observe numerical issues caused by large condition numbers in our experiments, and therefore solving the normal equations was sufficient. If ill-conditioning becomes significant, one may consider alternative approaches such as those discussed in \cite{chen2024teng} and \cite{wu2025deep}. In particular, \cite{wu2025deep} solves the associated linear system using the singular value decomposition of the Jacobian matrix to improve numerical stability.

\begin{algorithm}[H]
\caption{Computation of Parameter Equation for Game Dynamics}
\label{alg: game}
\begin{algorithmic}[1]
\STATE{Initialize the neural network $T_{\theta(0)}$ and set $\theta^{(0)} = \theta(0)$. Set the maximum iteration number to be $M$ and the time step to be $\Delta t$. }
\STATE{Sample $\{z^i\}_{i=1}^N$ from $\rho(\cdot,0)$.}
\FOR{$k=1, \cdots, M$}
\STATE{Set $t = k \cdot \Delta t$. Update $x_t^i = T_{\theta(t)}(z^i)$ for all $i \in [N]$. }
\STATE{Approximate $\nabla_x \log \rho \left( x_t^i, t \right)$ with reverse mode of Neural ODE and autodifferentiation by setting $\tilde{z}_t^i = T^{-1}_{\theta^{(k)}} \left( x_t^i \right)$.}
\STATE{Solve the linear system
\begin{align}
    \widehat{G}(\theta) \eta = \int \left[ \frac{\partial T_{\theta}(z)}{\partial \theta} \right]^T f(T_{\theta}; D, \Pi_1, ..., \Pi_d) (z) d\lambda(z)
\end{align}
to obtain $\eta$ with minimal residual method.}
\STATE{Update $\theta^{(k+1)}=\theta^{(k)} + \Delta t \cdot \eta$}
\ENDFOR
\STATE{return $\theta^{(M)}, T_{\theta^{(M)}}$}
\end{algorithmic}
\end{algorithm}

\subsection{2D Non-potential Game}
Consider the game described in \cref{cournot_duopoly}. Suppose the parameters are $a = 3, b = 1, d = 1, \mu = 2$, then the payoff functions are 
\begin{align}
    \Pi_1(x_1, x_2) = \left[ 3 - (x_1 + x_2) \right] x_1 - 1 - 3 x_1 + 5 x_1 x_2 - 4 x_1 x_2^2 \\
    \Pi_2(x_1, x_2) = \left[ 3 - (x_1 + x_2) \right] x_2 - 1 - 3 x_2 + 5 x_1 x_2 - 4 x_2 x_1^2
\end{align}
Suppose the coefficients $\sigma_1$ and $\sigma_2$ are set to be $\sigma_1 = \sigma_2 = 0.1$. In this example, we have two known Nash equilibria, the origin $(0,0)$ and $\left( \frac{1}{2}, \frac{1}{2} \right)$. We consider the sample trajectories computed with the best response dynamics as the ground truth. We initiate the strategy pairs by sampling $10000$ points from a uniform distribution in $[0,1]^2$. Euler–Maruyama scheme is used to compute the dynamics with $h = 0.001$. 


In \cref{fig: 2d_nonpotential_trajectories}, we plot the trajectories of three samples points computed by the best response dynamics and pushforward maps at $t \in [0, 0.2, 0.4, 0.6, 0.8, 1.0]$, respectively. In \cref{fig: 2d_nonpotential_brd_traj}, the curves with orange, purple, and green colors show the trajectories computed by the best response dynamics. In \cref{fig: 2d_nonpotential_pushforward_traj}, the curves with orange, purple, and green colors show the trajectories computed by pushforward maps and the blue shaded circles show the probability density at the points. The red points mark the known Nash equilibria $(0,0)$ and $(0.5, 0.5)$.

\begin{figure}[H]
    \centering
    \begin{subfigure}[t]{0.43\textwidth}
        \centering
        \includegraphics[width=\textwidth]{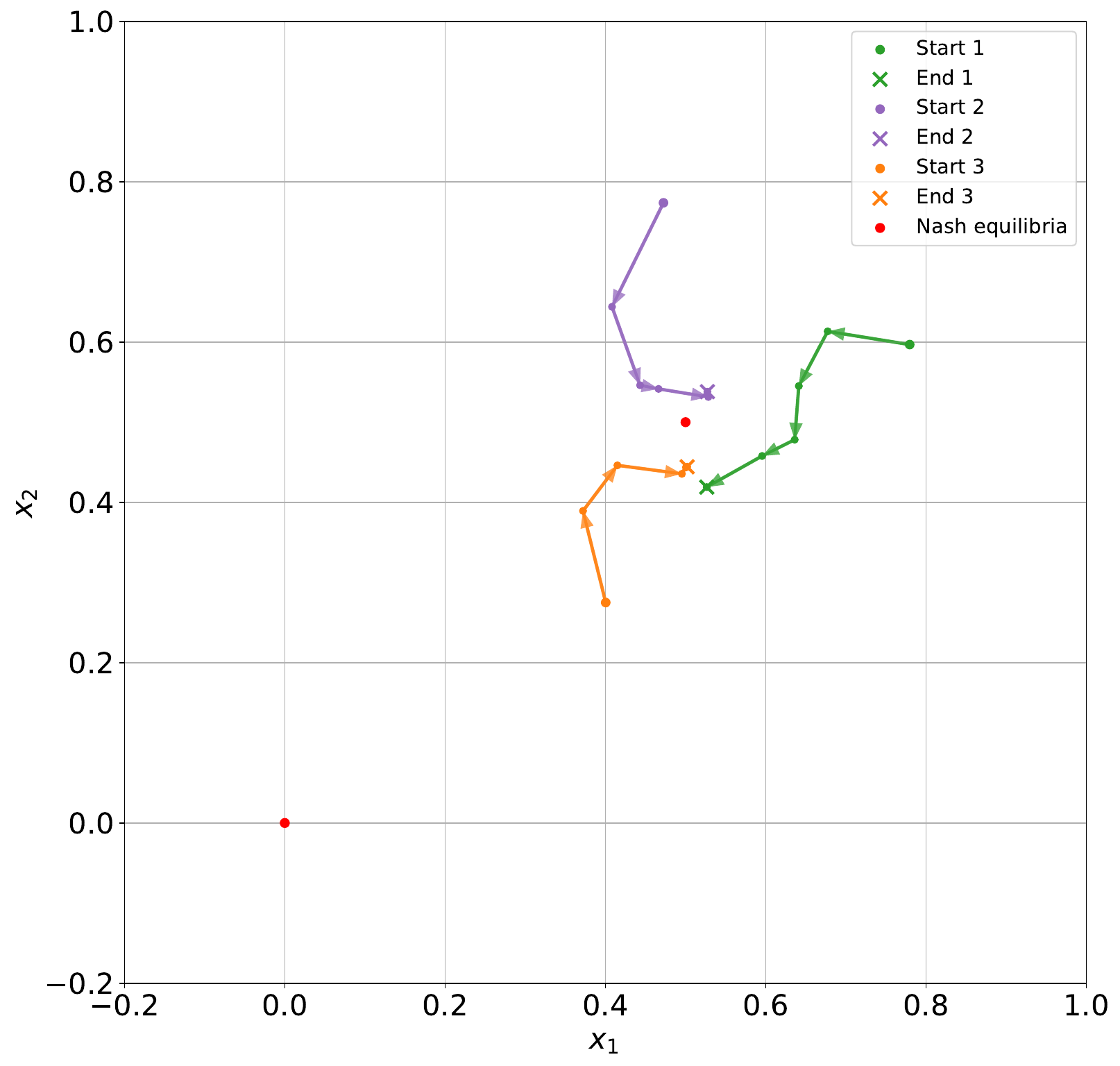}
        \caption{Best Response Dynamics}
        \label{fig: 2d_nonpotential_brd_traj}
    \end{subfigure}%
    ~ 
    \begin{subfigure}[t]{0.5\textwidth}
        \centering
        \includegraphics[width=\textwidth]{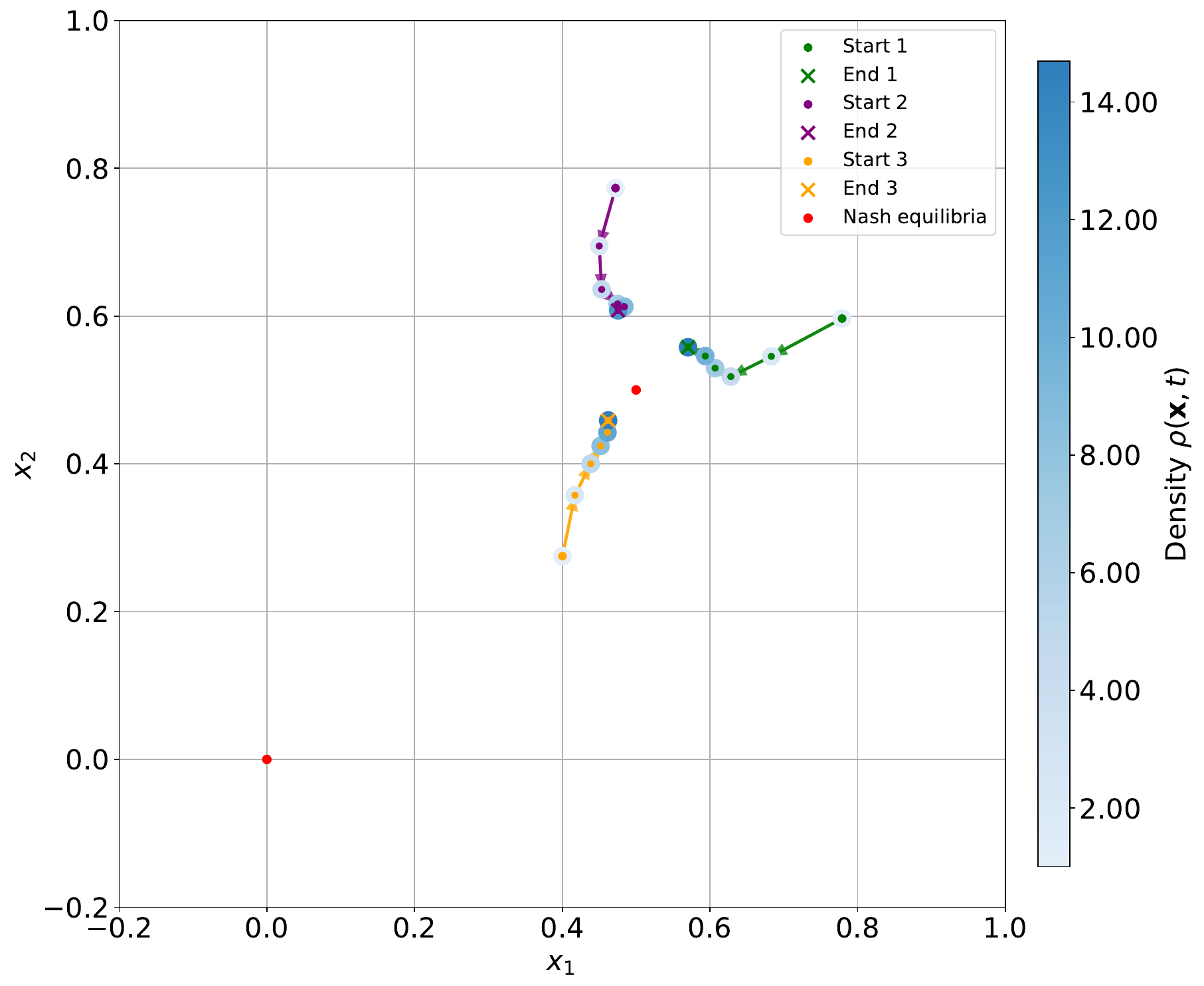}
        \caption{Pushforward Map}
        \label{fig: 2d_nonpotential_pushforward_traj}
    \end{subfigure}
    \caption{Plots of the trajectories of $3$ randomly sampled points of best response dynamics and pushforward maps. The initial samples are drawn from a uniform distribution in $[0,1]^2$. The coefficient for diffusion term is $\sigma_1 = \sigma_2 = 0.1$. }
    \label{fig: 2d_nonpotential_trajectories}
\end{figure}

We present snapshots of $1000$ sample points computed by the best response dynamics and pushforward maps at different times in \cref{fig: 2d_nonpotential_combined}. In the plots, the samples points computed by the best response dynamics and pushforward maps are marked in blue and orange, respectively. The Nash equilibria are marked in red.

\begin{figure}[H]
    \centering
    \begin{subfigure}[t]{0.3\textwidth}
        \centering
        \includegraphics[width=\textwidth]{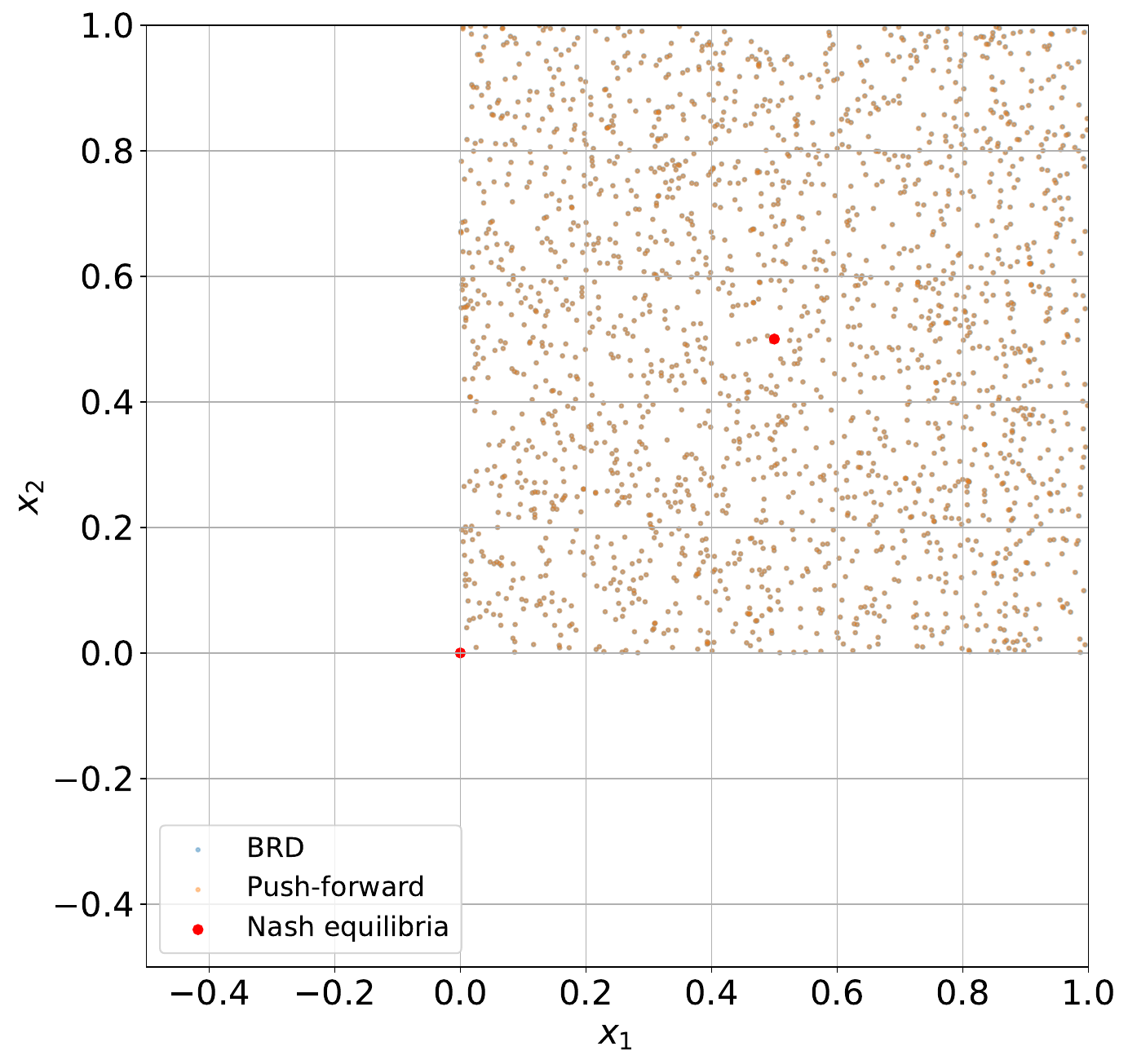}
        \caption{$t=0$}
    \end{subfigure}%
    ~ 
    \begin{subfigure}[t]{0.3\textwidth}
        \centering
        \includegraphics[width=\textwidth]{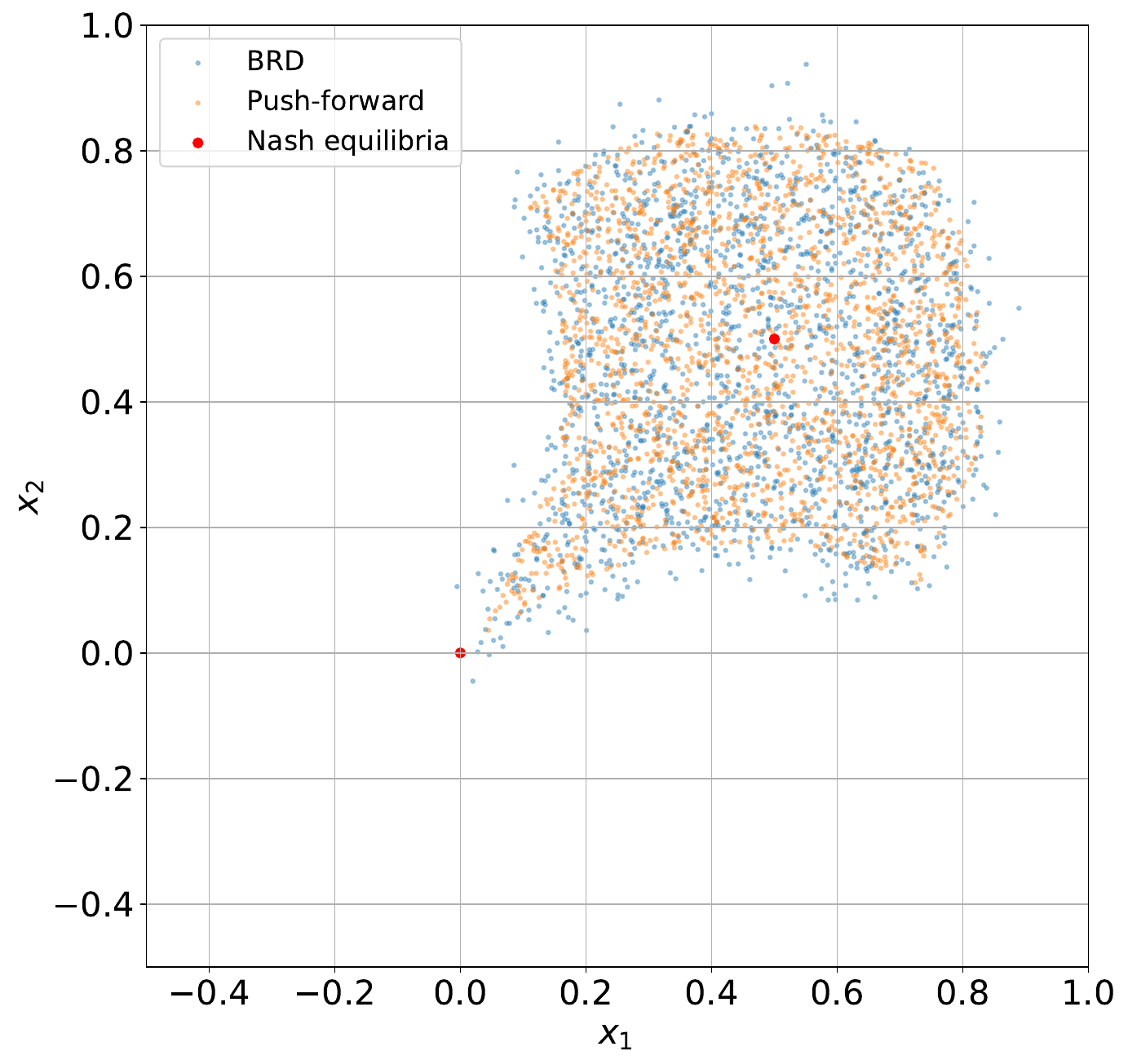}
        \caption{$t=0.2$}
    \end{subfigure}
    ~ 
    \begin{subfigure}[t]{0.3\textwidth}
        \centering
        \includegraphics[width=\textwidth]{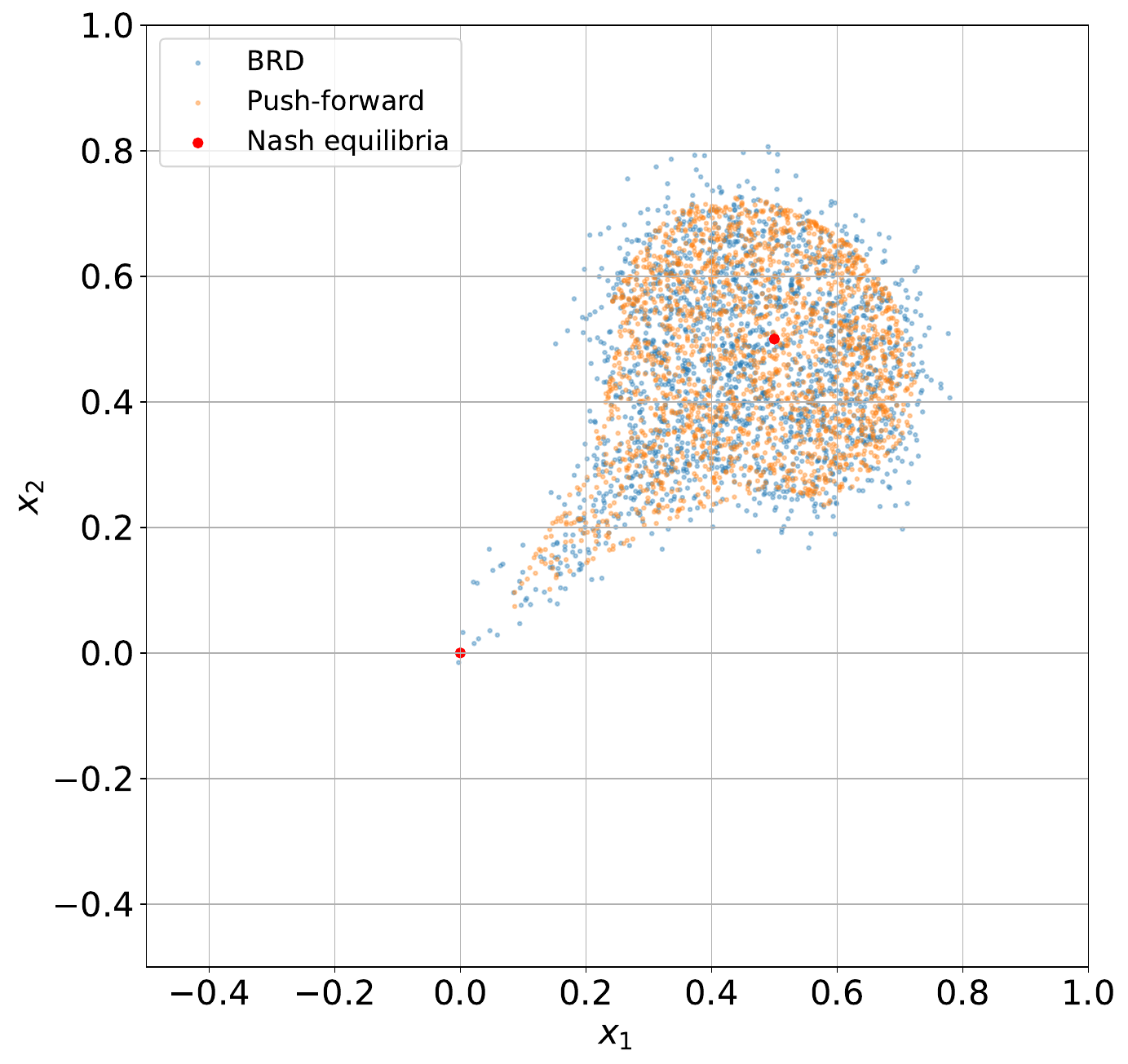}
        \caption{$t=0.4$}
    \end{subfigure}
    ~ 
    \begin{subfigure}[t]{0.3\textwidth}
        \centering
        \includegraphics[width=\textwidth]{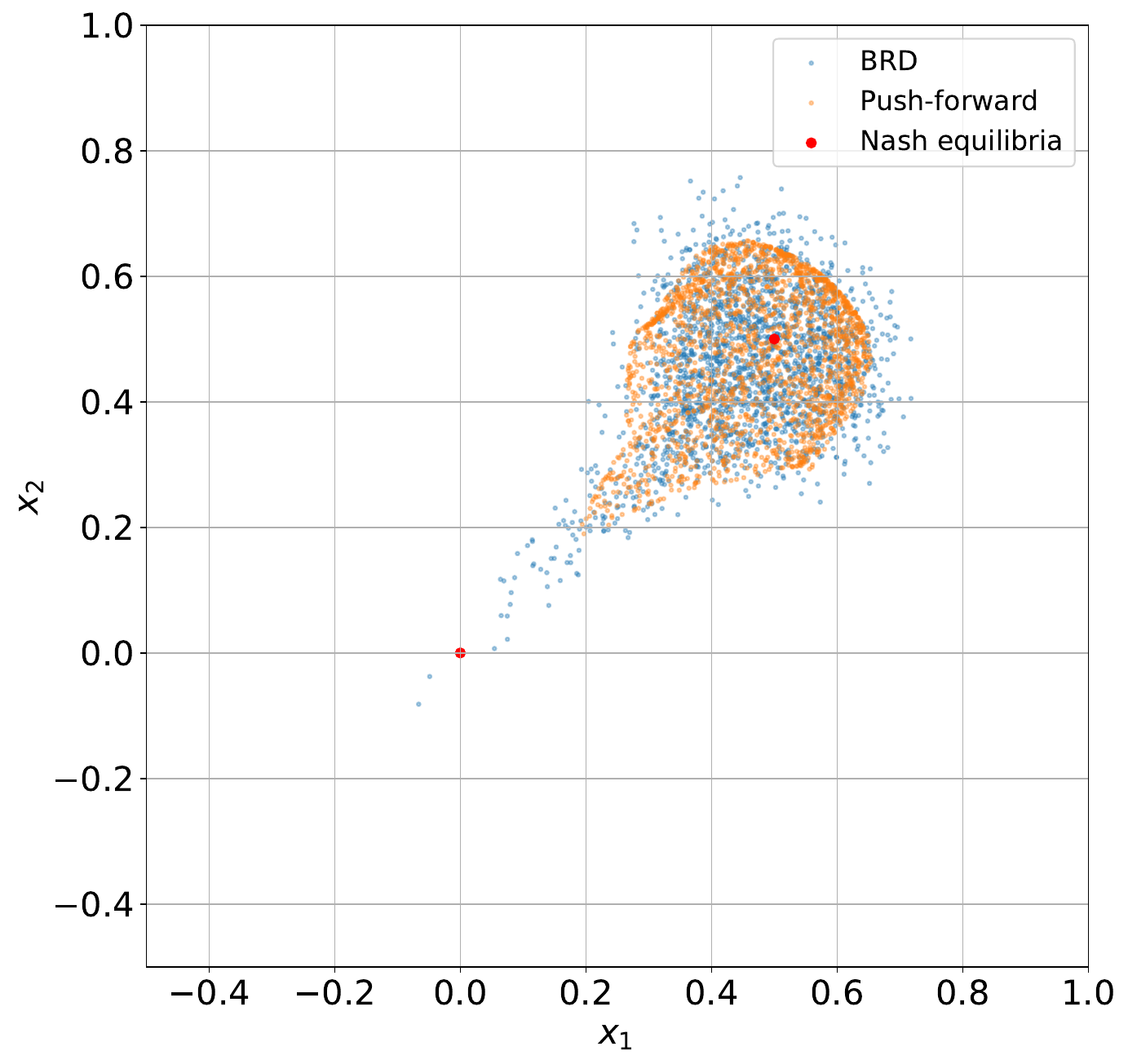}
        \caption{$t=0.6$}
    \end{subfigure}
    ~ 
    \begin{subfigure}[t]{0.3\textwidth}
        \centering
        \includegraphics[width=\textwidth]{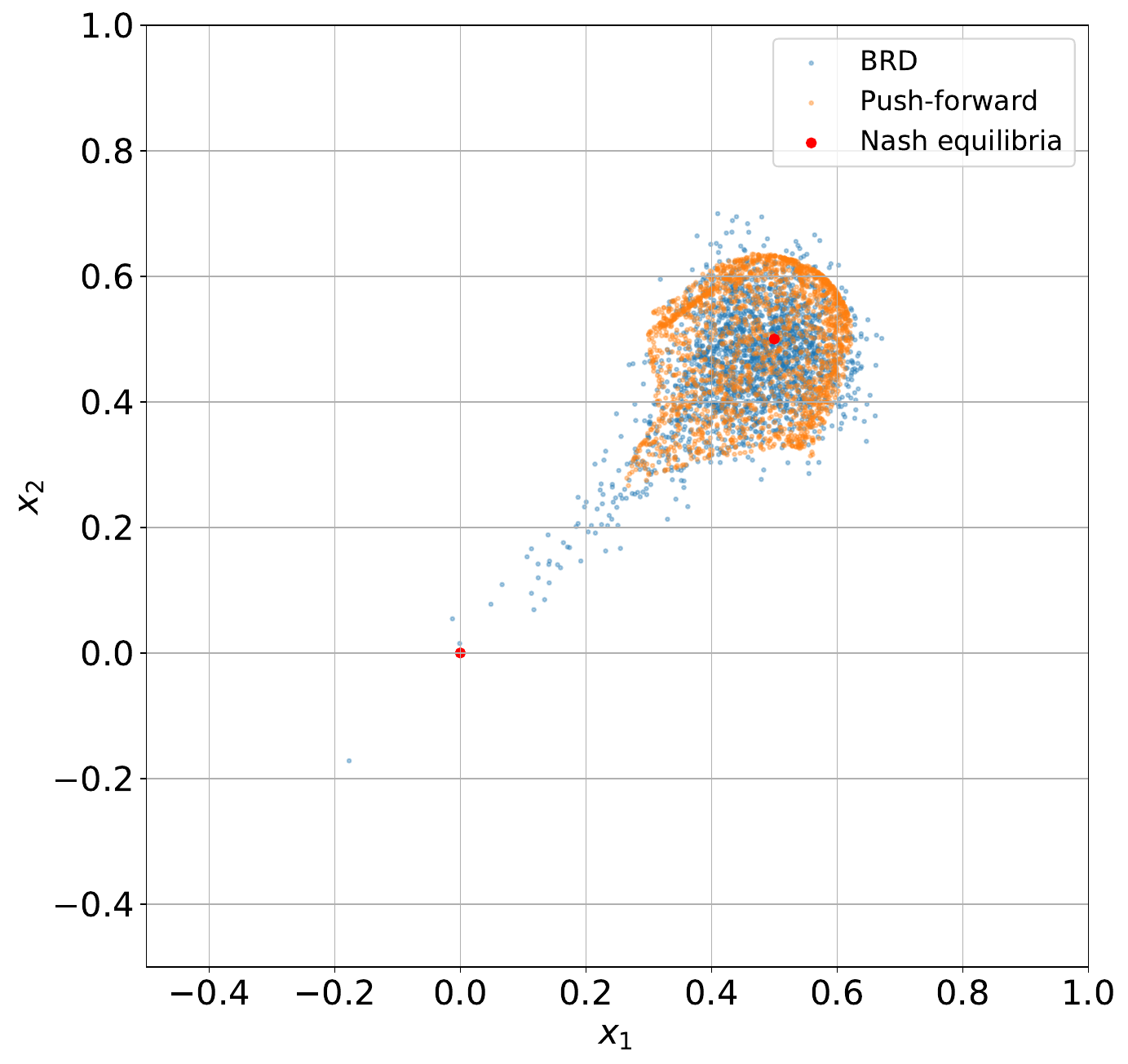}
        \caption{$t=0.8$}
    \end{subfigure}
    ~ 
    \begin{subfigure}[t]{0.3\textwidth}
        \centering
        \includegraphics[width=\textwidth]{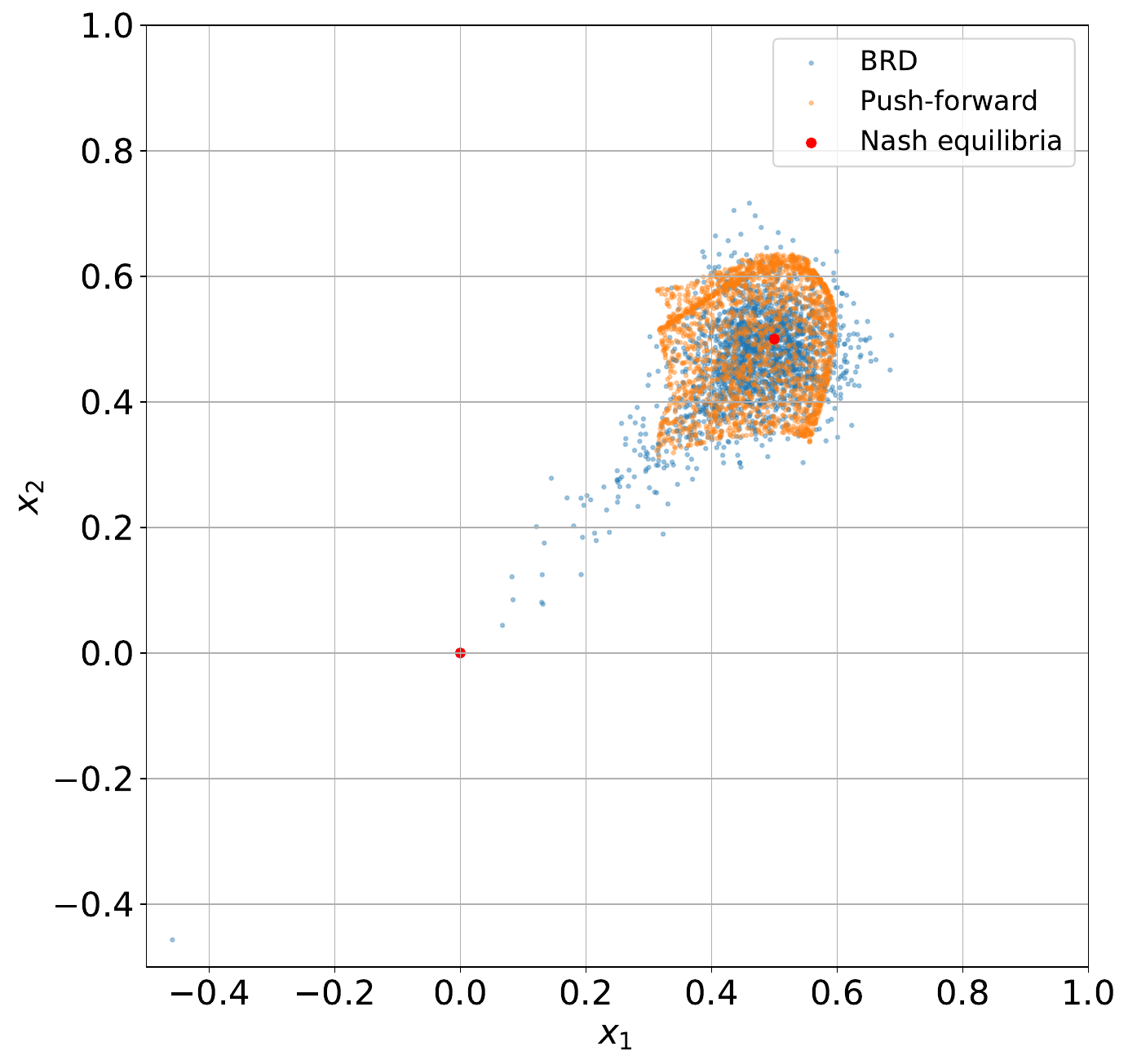}
        \caption{$t=1$}
    \end{subfigure}
    \caption{Plots of $1000$ sample points of 2D non-potential game computed with best response dynamics and pushforward map from $t=0$ to $t=1$. The initial samples are drawn from a uniform distribution in $[0,1]^2$. The coefficient for diffusion term is $\sigma_1 = \sigma_2 = 0.1$. The blue points are computed with the best response dynamics, the orange points are computed with pushforward maps, and the red points mark the known Nash equilibria $(0,0)$ and $(0.5, 0.5)$.}
    \label{fig: 2d_nonpotential_combined}
\end{figure}

As shown in the plots, the behavior of the strategies are well captured by the pushforward maps. With time going by, the strategies converge to the known stable Nash equilibrium $(0.5, 0.5)$. 

Meanwhile, we also test the performance of the pushforward maps with a different initial distribution. Suppose the initial samples are drawn from a Gaussian distribution $\mathcal{N}(\mu, \sigma I)$ where $\mu = (0.5, 0.5)^T$ and $\sigma = 0.03$. The parameters of the pushforward map are obtained by training with uniform initial samples. The following figures show the samples brought by the pushforward maps at different times. The results are shown in \cref{fig: 2d_nonpotential_pushforward_Gaussian}.


\begin{figure}[H]
    \centering
    \begin{subfigure}[t]{0.3\textwidth}
        \centering
        \includegraphics[width=\textwidth]{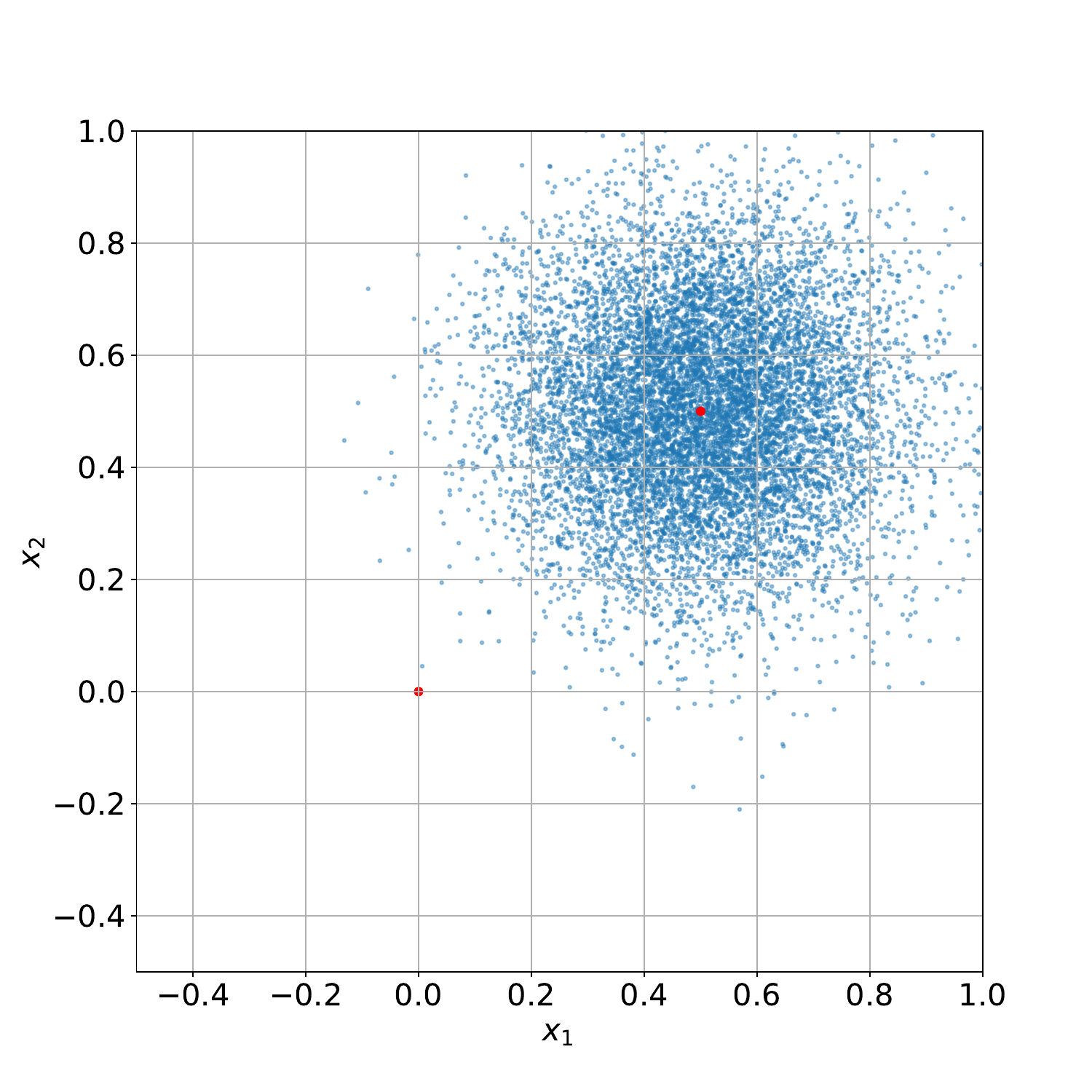}
        \caption{$t=0$}
    \end{subfigure}
    ~ 
    \begin{subfigure}[t]{0.3\textwidth}
        \centering
        \includegraphics[width=\textwidth]{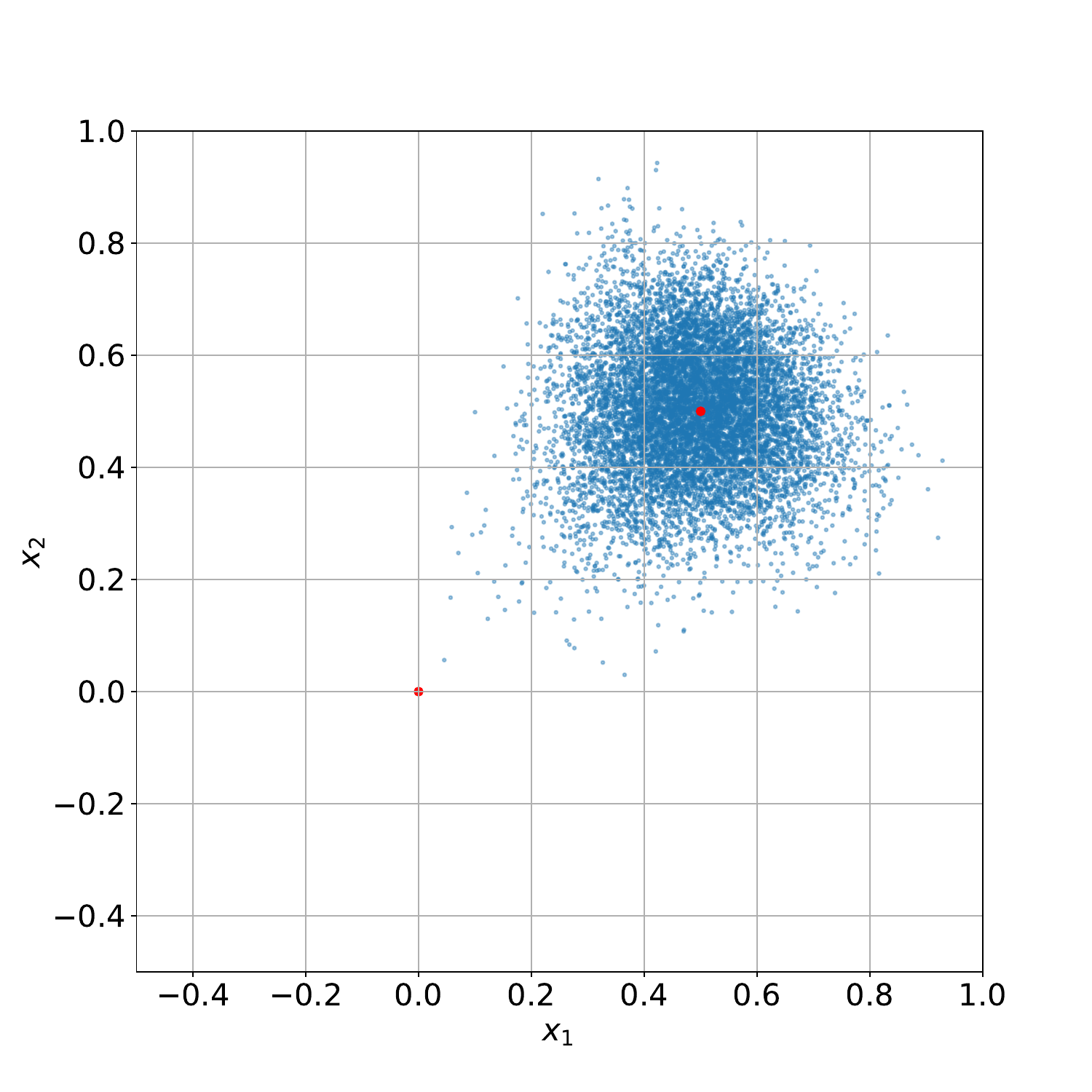}
        \caption{$t=0.2$}
    \end{subfigure}
    ~ 
    \begin{subfigure}[t]{0.3\textwidth}
        \centering
        \includegraphics[width=\textwidth]{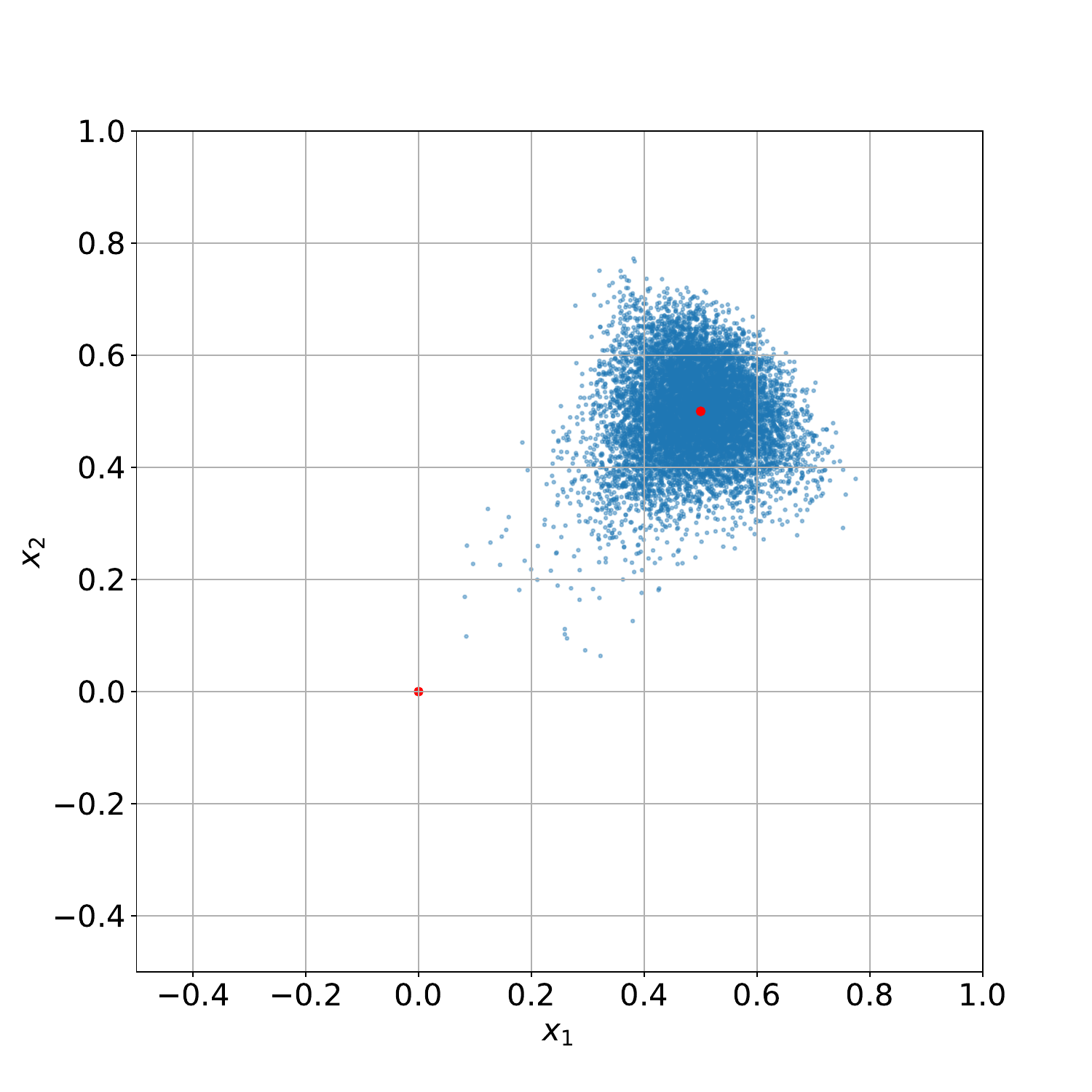}
        \caption{$t=0.4$}
    \end{subfigure}
    ~ 
    \begin{subfigure}[t]{0.3\textwidth}
        \centering
        \includegraphics[width=\textwidth]{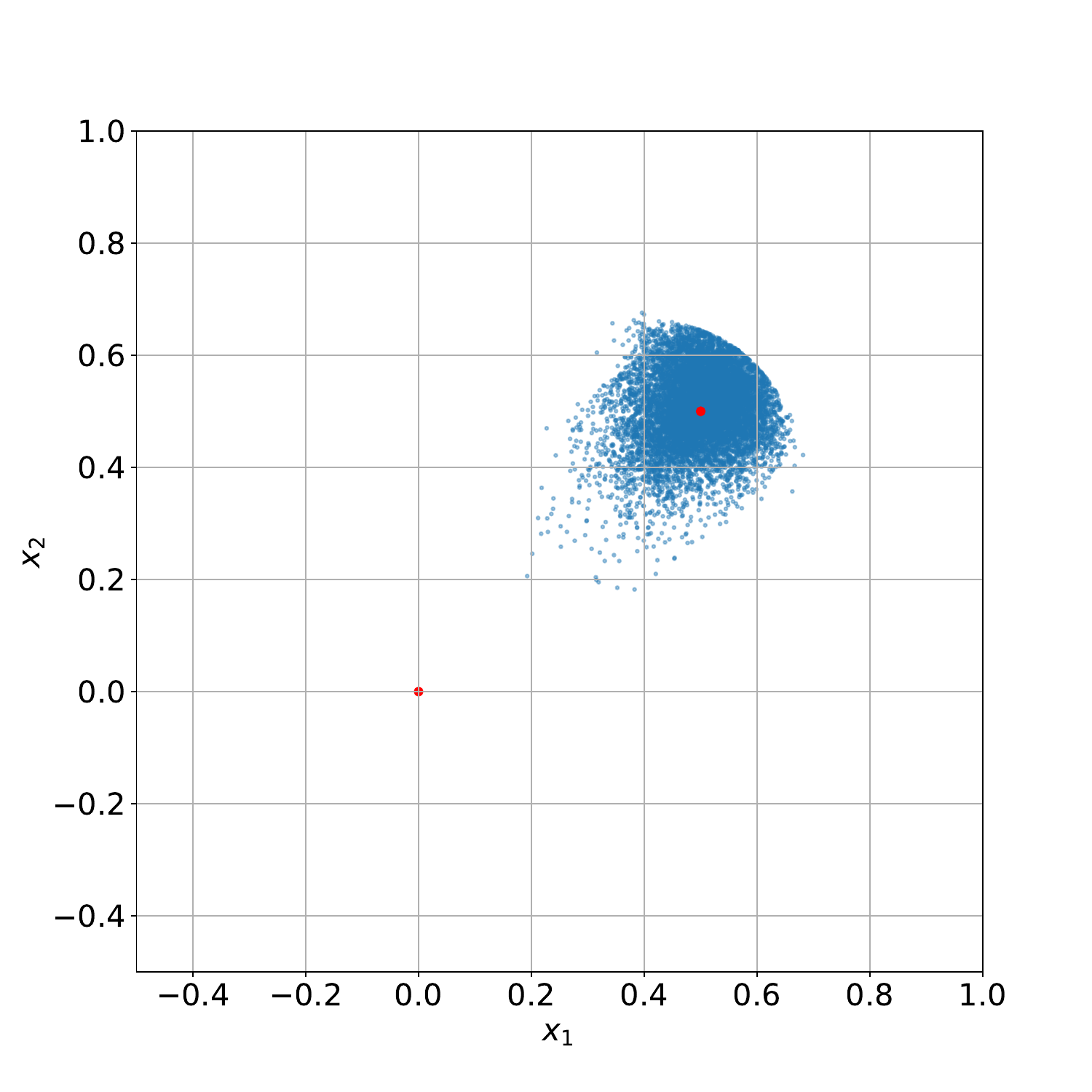}
        \caption{$t=0.6$}
    \end{subfigure}
    ~ 
    \begin{subfigure}[t]{0.3\textwidth}
        \centering
        \includegraphics[width=\textwidth]{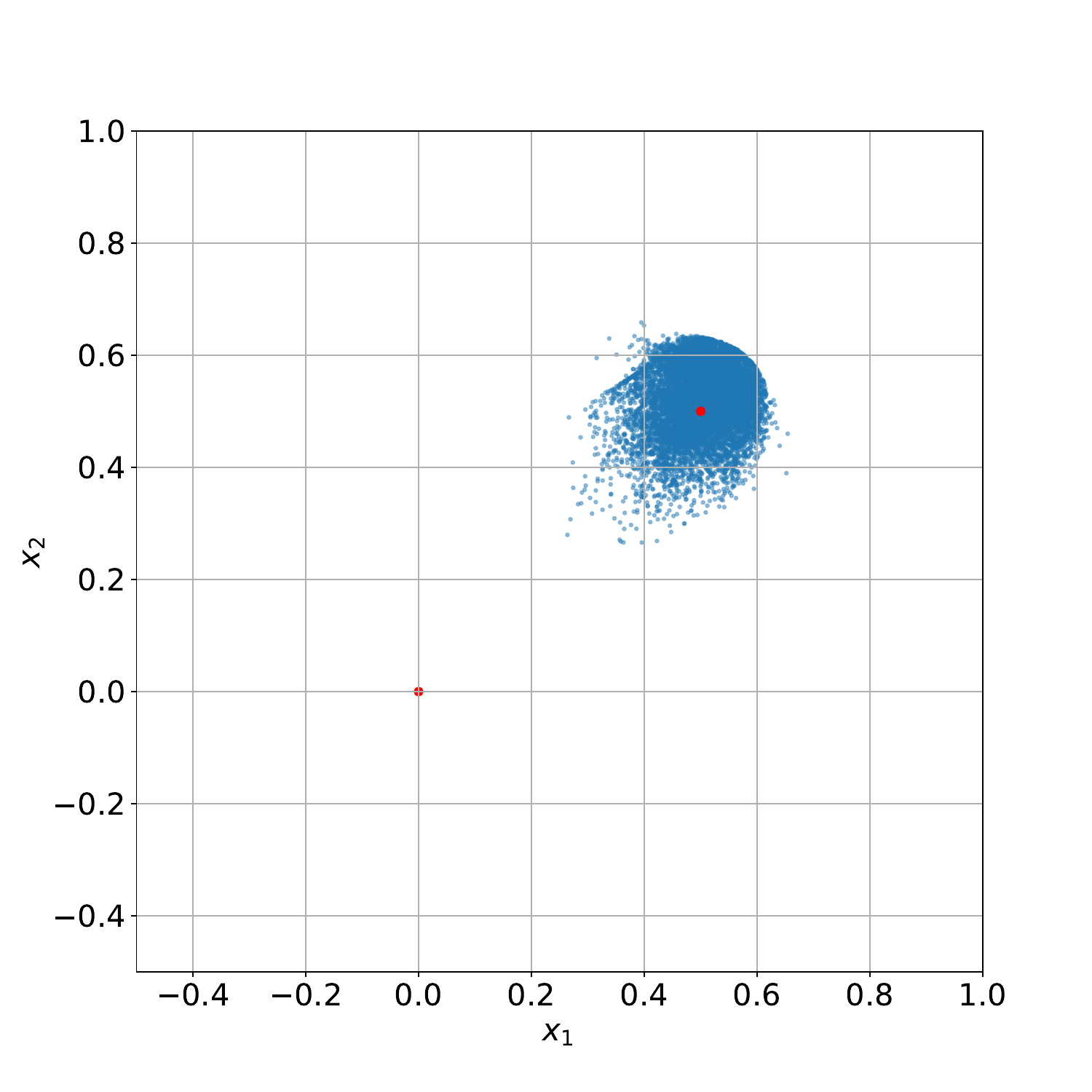}
        \caption{$t=0.8$}
    \end{subfigure}
    ~ 
    \begin{subfigure}[t]{0.3\textwidth}
        \centering
        \includegraphics[width=\textwidth]{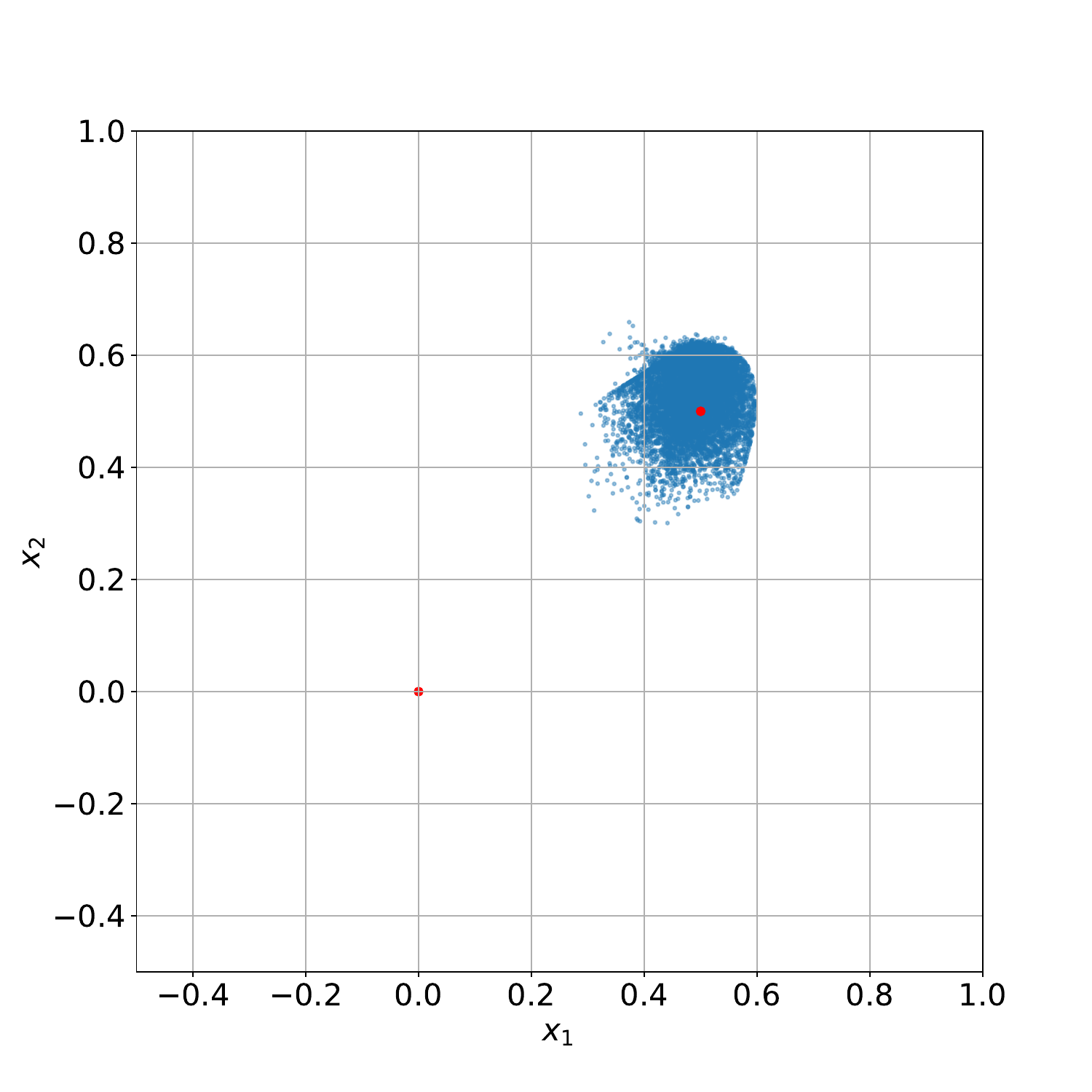}
        \caption{$t=1$}
    \end{subfigure}
    \caption{Plots of sample points of 2D non-potential game computed with pushforward map $T_{\theta(t)}$ from $t=0$ to $t=1$. The initial samples are drawn from a Gaussian distribution. The coefficient for diffusion term is $\sigma_1 = \sigma_2 = 0.1$. The red points mark the known Nash equilibria $(0,0)$ and $(0.5, 0.5)$.}
    \label{fig: 2d_nonpotential_pushforward_Gaussian}
\end{figure}

We compare the result by directly computing the best response dynamics with Gaussian initial samples. The plots are shown below in \cref{fig: 2d_nonpotential_brd_Gaussian}.

\begin{figure}[H]
    \centering
    \begin{subfigure}[t]{0.3\textwidth}
        \centering
        \includegraphics[width=\textwidth]{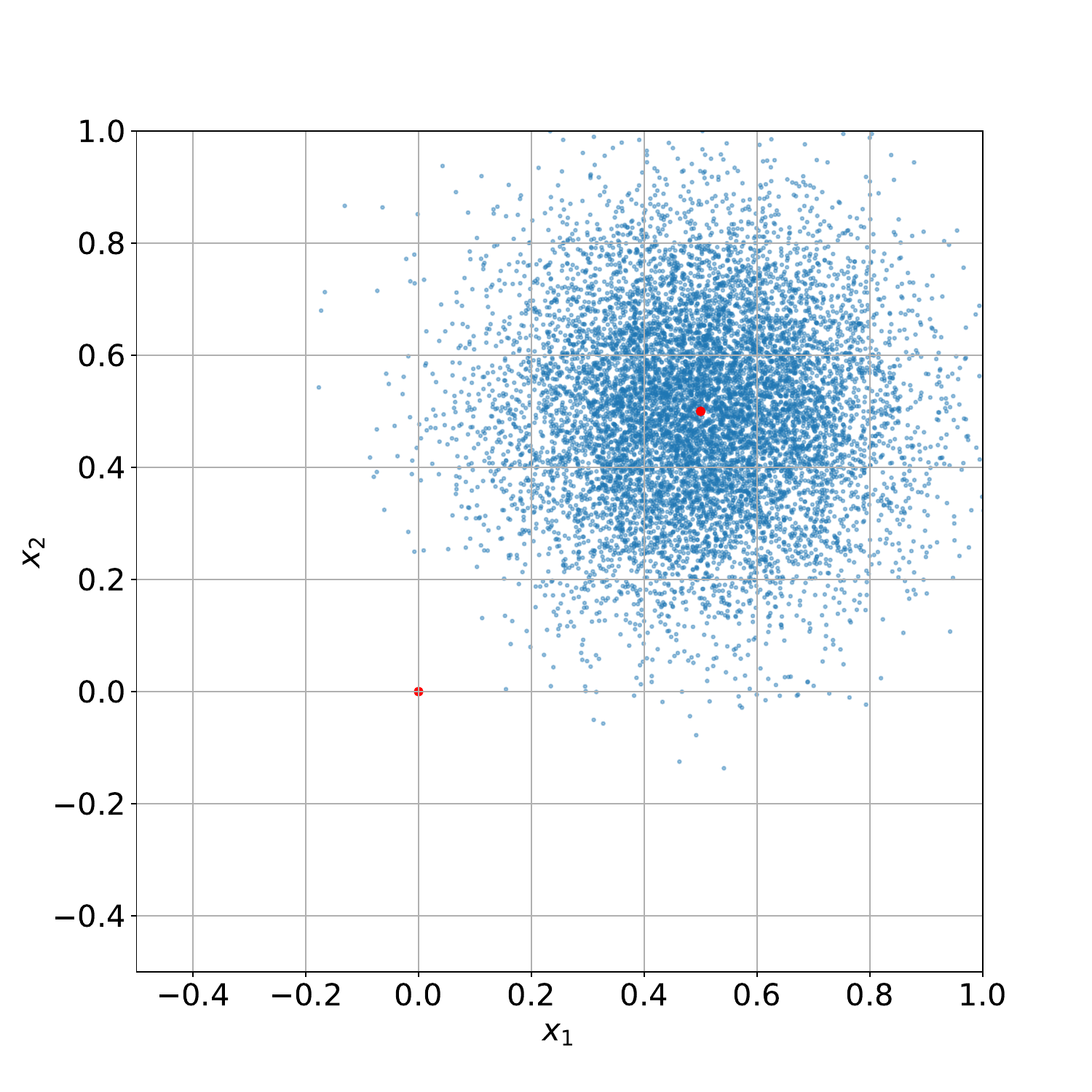}
        \caption{$t=0$}
    \end{subfigure}%
    ~ 
    \begin{subfigure}[t]{0.3\textwidth}
        \centering
        \includegraphics[width=\textwidth]{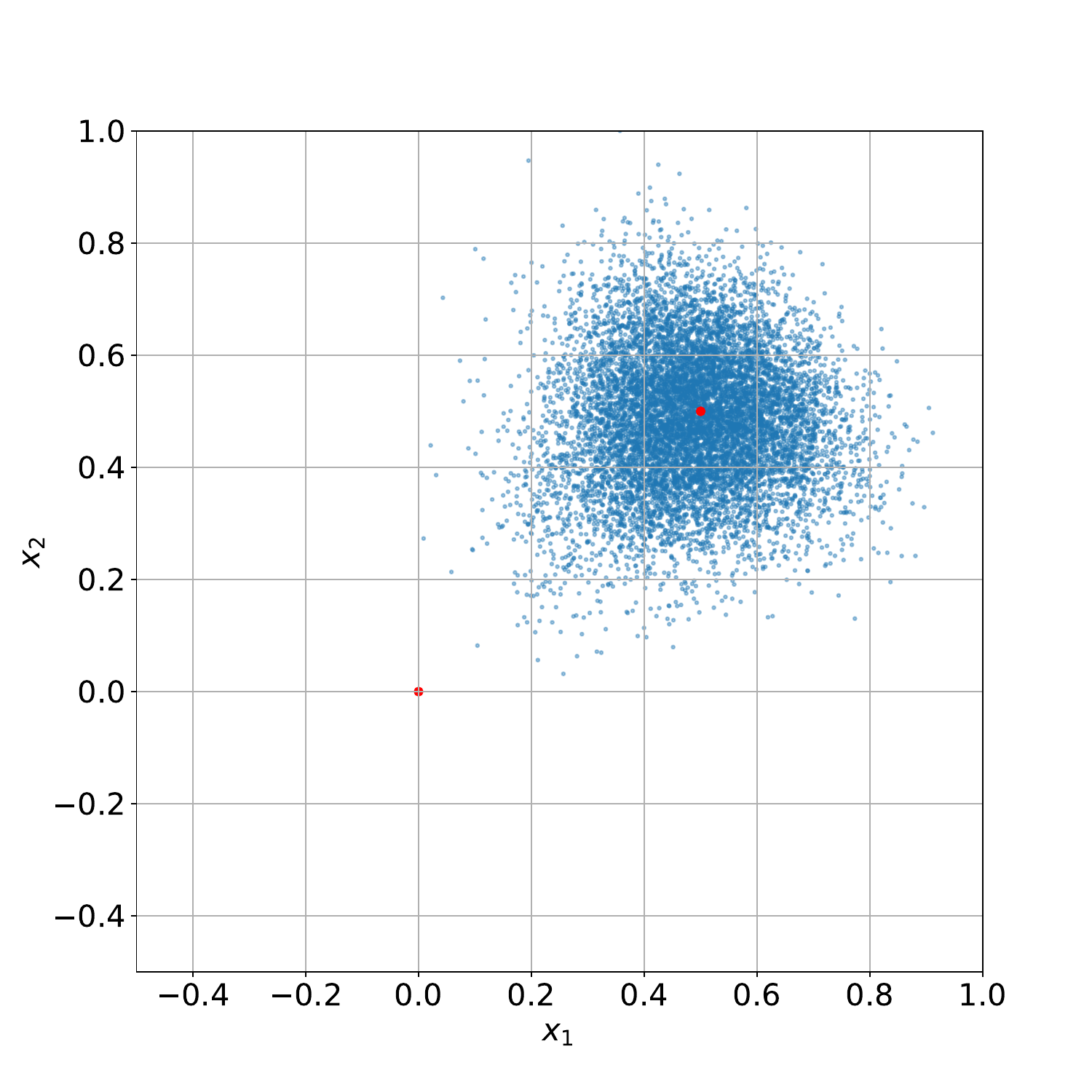}
        \caption{$t=0.2$}
    \end{subfigure}
    ~ 
    \begin{subfigure}[t]{0.3\textwidth}
        \centering
        \includegraphics[width=\textwidth]{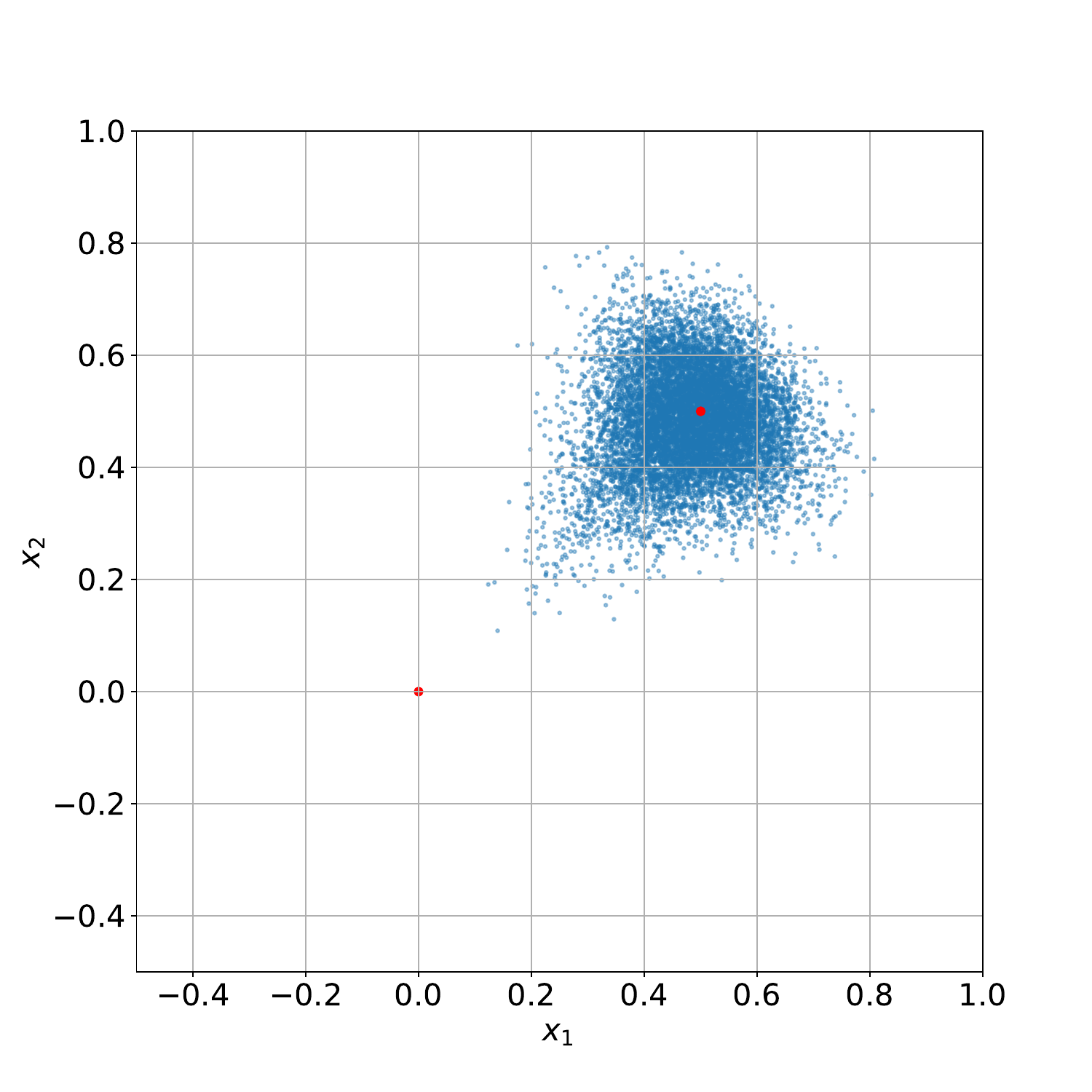}
        \caption{$t=0.4$}
    \end{subfigure}
    ~ 
    \begin{subfigure}[t]{0.3\textwidth}
        \centering
        \includegraphics[width=\textwidth]{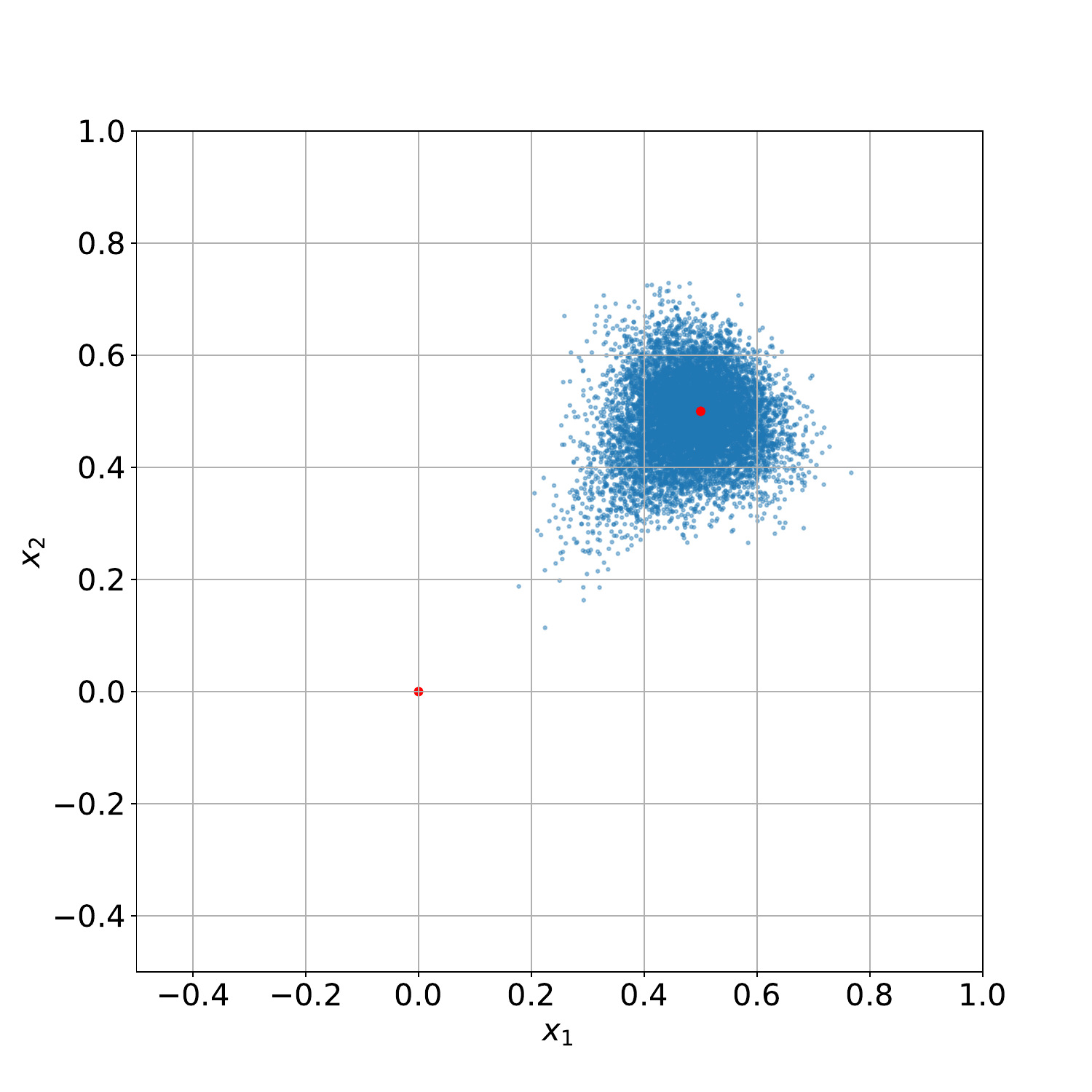}
        \caption{$t=0.6$}
    \end{subfigure}
    ~ 
    \begin{subfigure}[t]{0.3\textwidth}
        \centering
        \includegraphics[width=\textwidth]{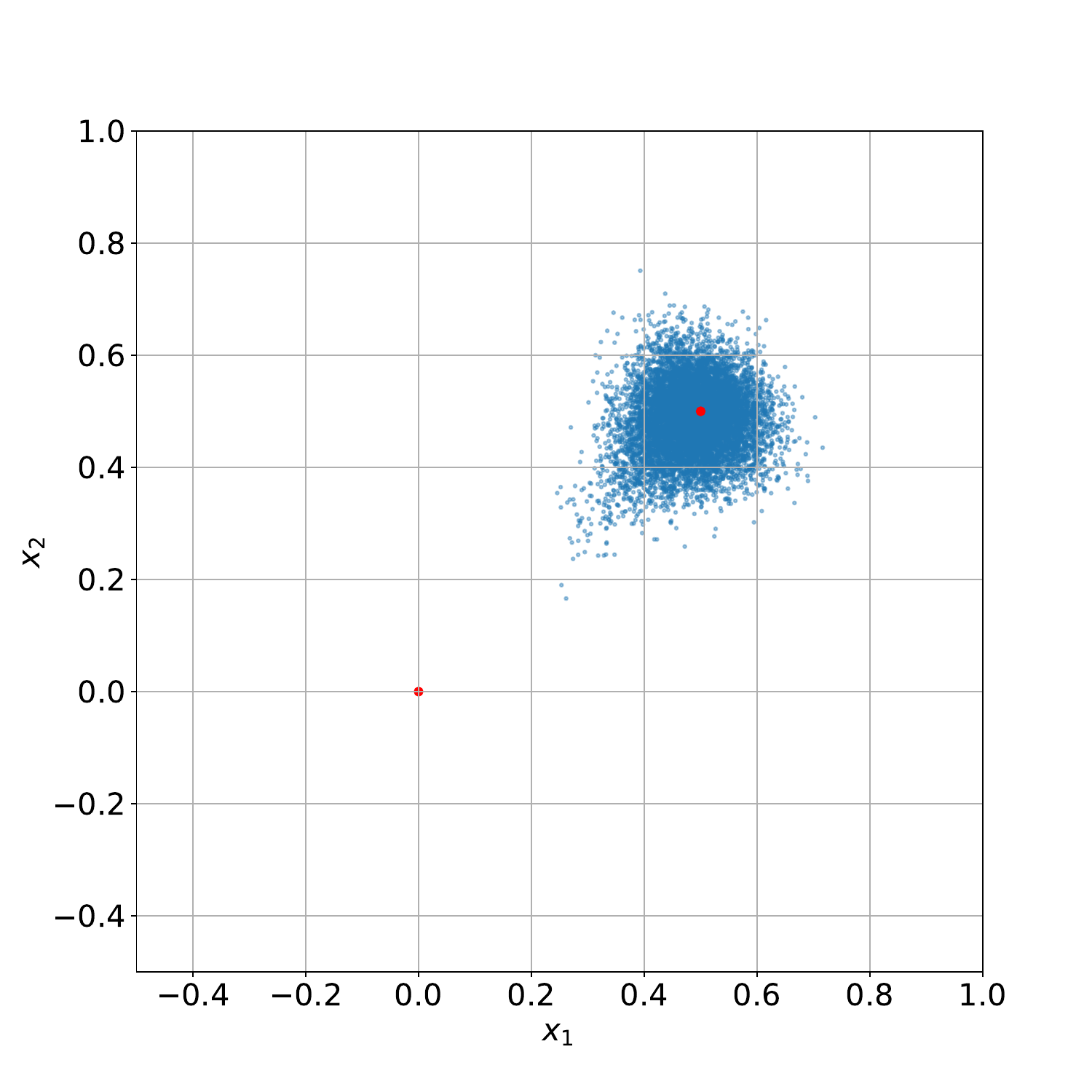}
        \caption{$t=0.8$}
    \end{subfigure}
    ~ 
    \begin{subfigure}[t]{0.3\textwidth}
        \centering
        \includegraphics[width=\textwidth]{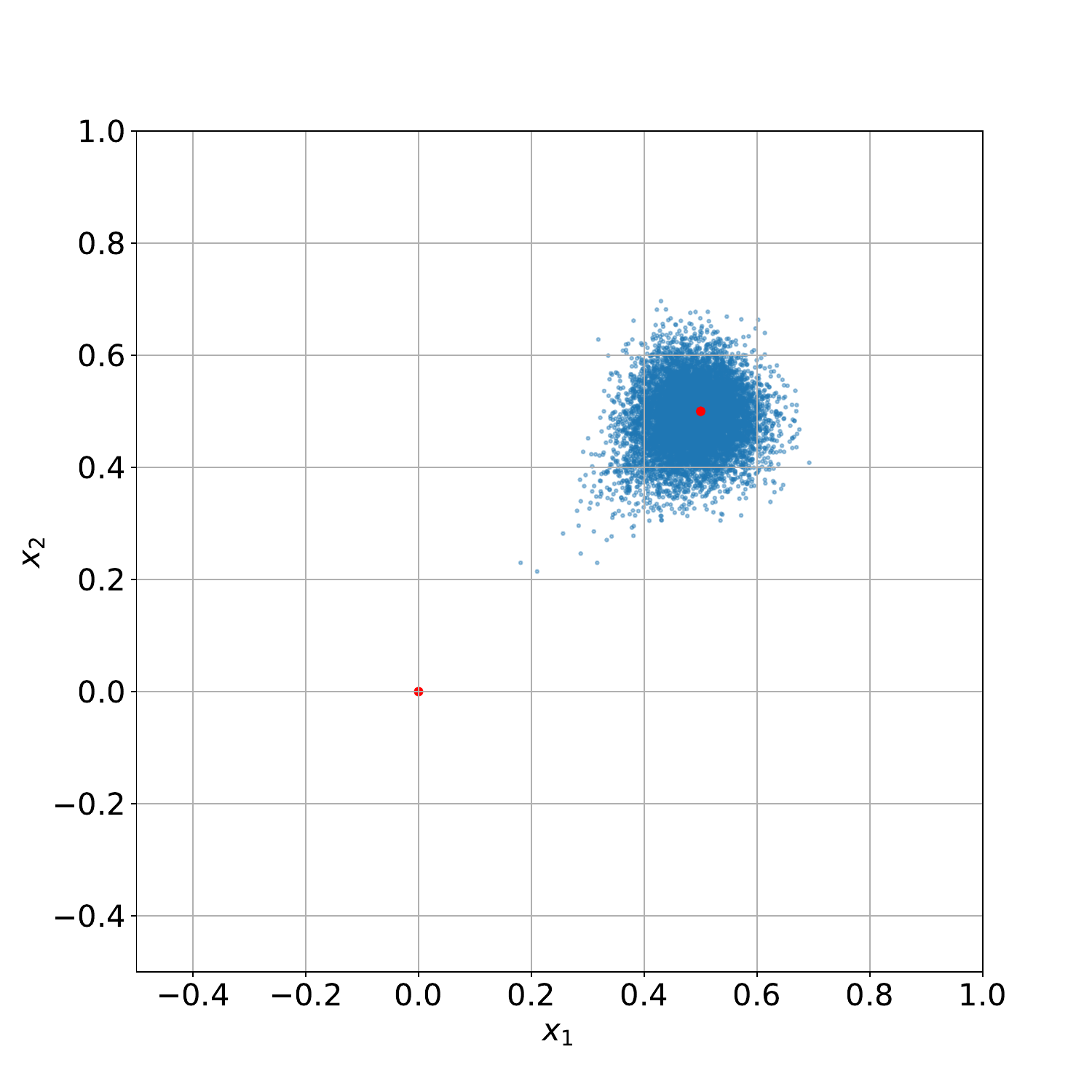}
        \caption{$t=1$}
    \end{subfigure}
    \caption{Plots of sample points of 2D non-potential game computed with best response dynamics from $t=0$ to $t=1$. The initial samples are drawn from a Gaussian distribution. The coefficient for diffusion term is $\sigma_1 = \sigma_2 = 0.1$. The red points mark the known Nash equilibria $(0,0)$ and $(0.5, 0.5)$.}
    \label{fig: 2d_nonpotential_brd_Gaussian}
\end{figure}

\subsection{High-dimensional Non-potential Game}

Cournot game can be generalized to model a competition among $d$ players. Denote the quantities produced by each firm by $x_i, i \in [d]$. The total quantity available in the market is 
\begin{align}
    Q = \sum_{i=1}^d x_i.
\end{align}
Suppose the cost functions can be denoted by 
\begin{align}
    c_i(x_i, x_{-i}) = 2b \mu x_i \left( \sum_{j=1, j \neq i}^d x_j \right)^2,
\end{align}
the payoff functions $\Pi_i(x_i, x_{-i})$ for each firm $i$ is 
\begin{align}
    \Pi_i(x_i, x_{-i}) & = p(x_i, x_{-i}) x_i - c_i(x_i, x_{-i}) \\
    & = a - b \left( \sum_{j=1}^d x_j \right) - 2b \mu x_i \left( \sum_{j=1, j \neq i}^d x_j \right)^2
\end{align}
In this example, finding Nash equilibria for a general $d$ is challenging. Some equilibria can be identified by exploiting the symmetry of the payoff functions, but it remains uncertain whether additional equilibria exist.

\subsubsection{Cournot Game with Three Players}
Consider a Cournot game with $3$ players. The quantities produced by the firms are denoted by $x_1, x_2$ and $x_3$. The price of the homogenenous product is determined by the inverse demand function 
\begin{align}
    p(x_1, x_2, x_3) = a - b (x_1 + x_2 + x_3).
\end{align}
The payoff functions for each player $x_i$ is 
\begin{align}
    \Pi_i (x_i, x_{-i}) = \left[ a - b \left( x_1 + x_2 + x_3 \right)\right] x_i - c_i(x_i, x_{-i})
\end{align}
where the cost function
\begin{align}
    c_i(x_i, x_{-i}) = d + a x_i - b x_i (1+2\mu) \left( x_1 + x_2 + x_3 \right) + 2b \mu x_i (x_1 + x_2 + x_3)^2, \quad i \in [3]
\end{align}
In this example, the parameters are set to be $a = 3, b = 1, d = 1$ and $\mu = 2$.
In general, computing Nash equilibria in multiplayer games is a difficult task. In this example, we leveraged the symmetry of the payoff functions to identify four equilibria. It is possible that additional equilibria exist, but our stability analysis was limited to those currently known. By symmetry of the payoff functions, we obtain $5$ Nash equilibria, $(0, 0, 0)$, $\left( \frac{3}{8}, \frac{3}{8}, \frac{3}{8} \right)$, $\left( \frac{1}{2}, \frac{1}{2}, 0 \right)$, $\left( \frac{1}{2}, 0, \frac{1}{2} \right)$, $\left( 0, \frac{1}{2}, \frac{1}{2} \right)$. The origin is unstable and the remaining known Nash equilibria are asymptotically stable.

Below in \cref{fig: 3d_nonpotential_brd_uniform} and in \cref{fig: 3d_nonpotential_pushforward_uniform}, we show the sample plots computed with the best response dynamics and pushforward map, respectively. The computational results with pushforward map are consistent with those obtained from best response dynamics.

\begin{figure}[!th]
    \centering
    \begin{subfigure}[t]{0.4\textwidth}
        \centering
        \includegraphics[width=\textwidth]{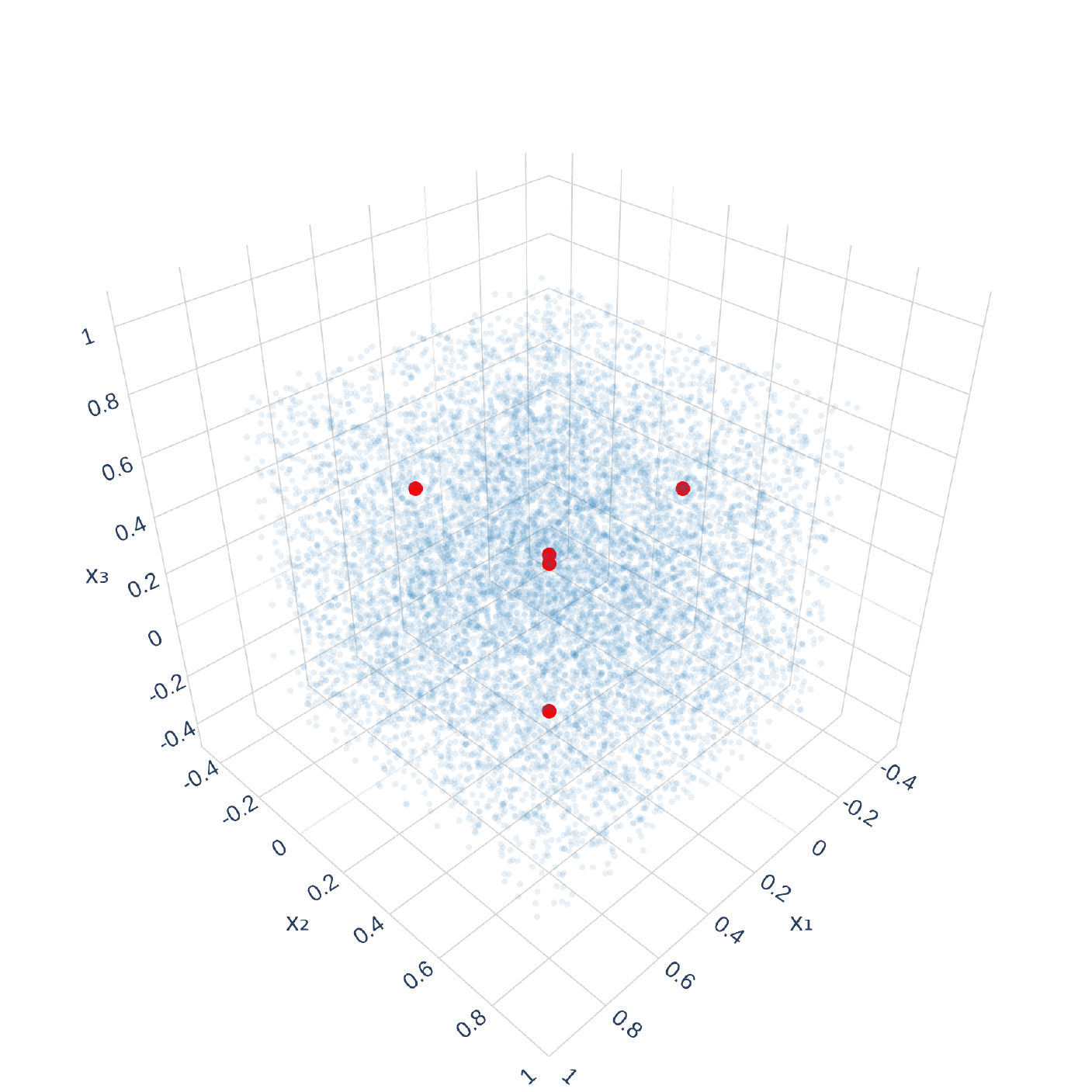}
        \caption{$t=0$}
    \end{subfigure}%
    ~ 
    \begin{subfigure}[t]{0.4\textwidth}
        \centering
        \includegraphics[width=\textwidth]{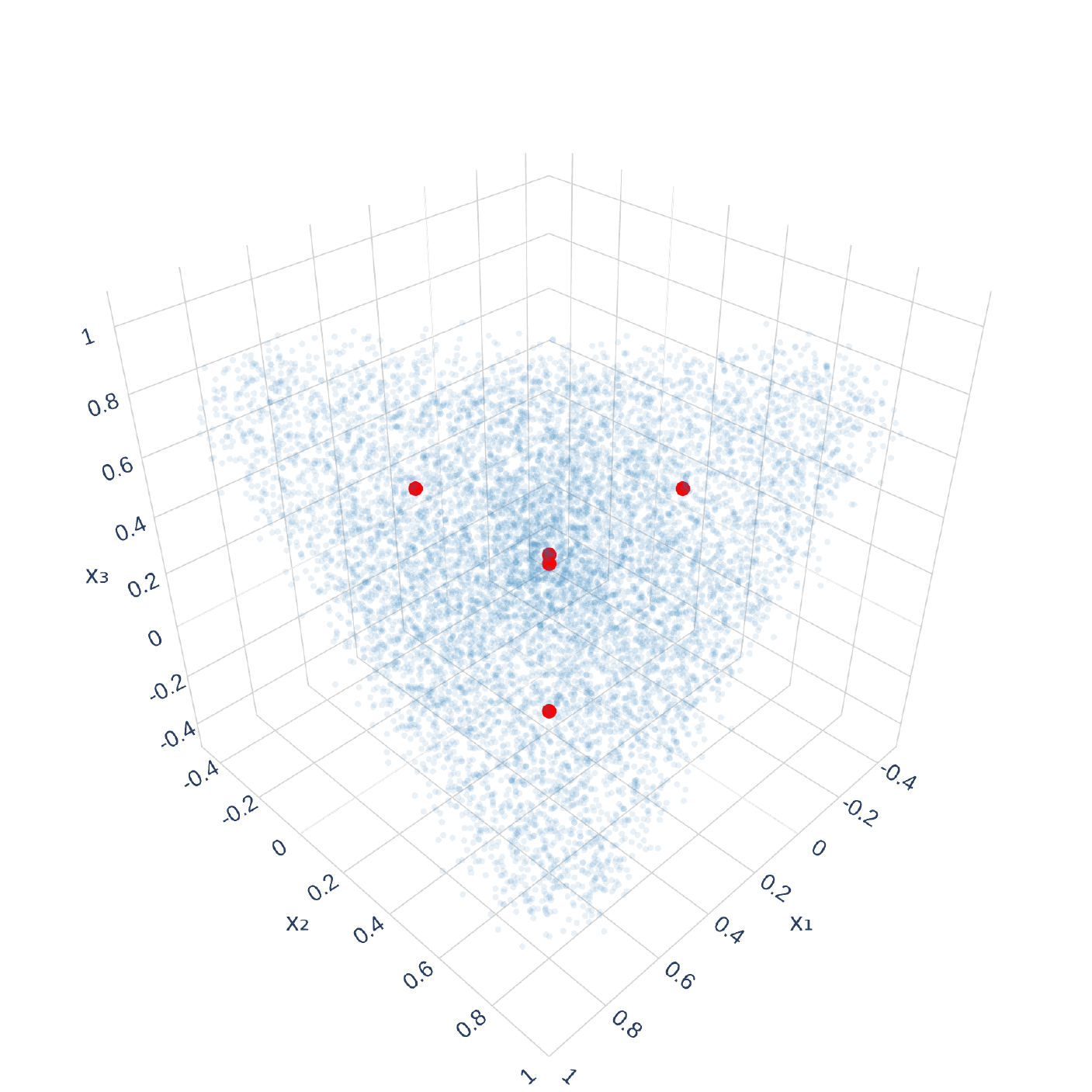}
        \caption{$t=0.2$}
    \end{subfigure}
    ~ 
    \begin{subfigure}[t]{0.4\textwidth}
        \centering
        \includegraphics[width=\textwidth]{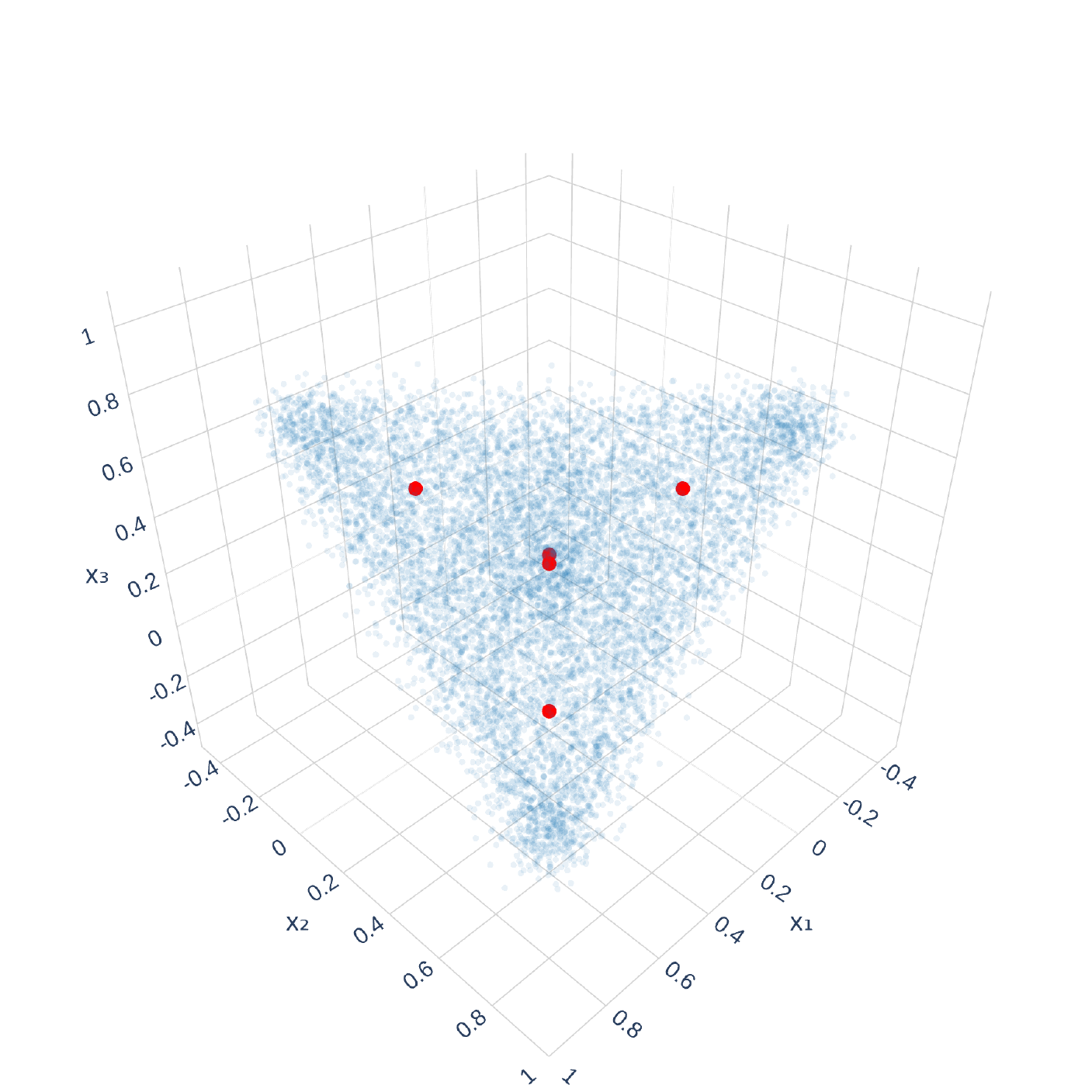}
        \caption{$t=0.4$}
    \end{subfigure}%
    ~ 
    \begin{subfigure}[t]{0.4\textwidth}
        \centering
        \includegraphics[width=\textwidth]{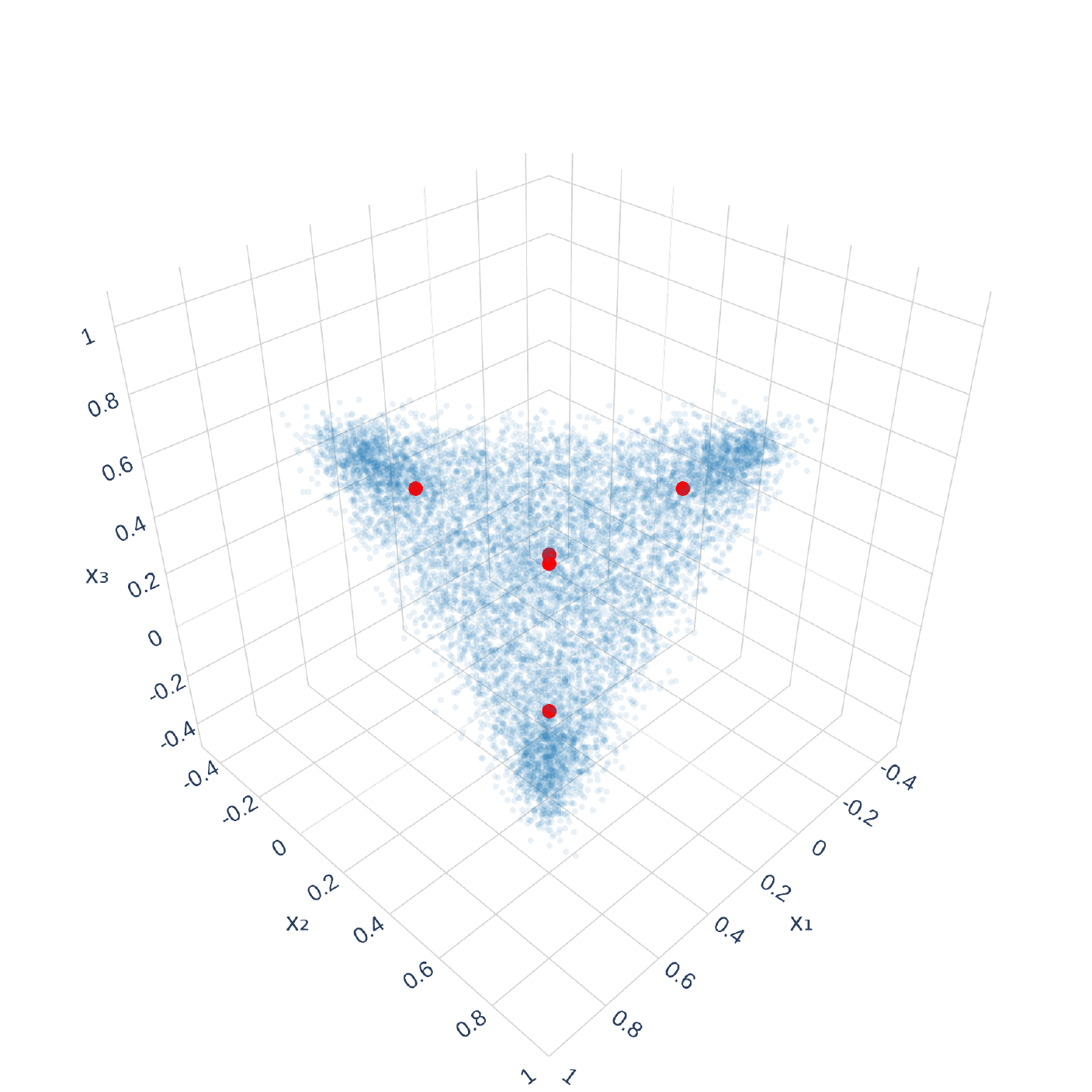}
        \caption{$t=0.6$}
    \end{subfigure}
    ~ 
    \begin{subfigure}[t]{0.4\textwidth}
        \centering
        \includegraphics[width=\textwidth]{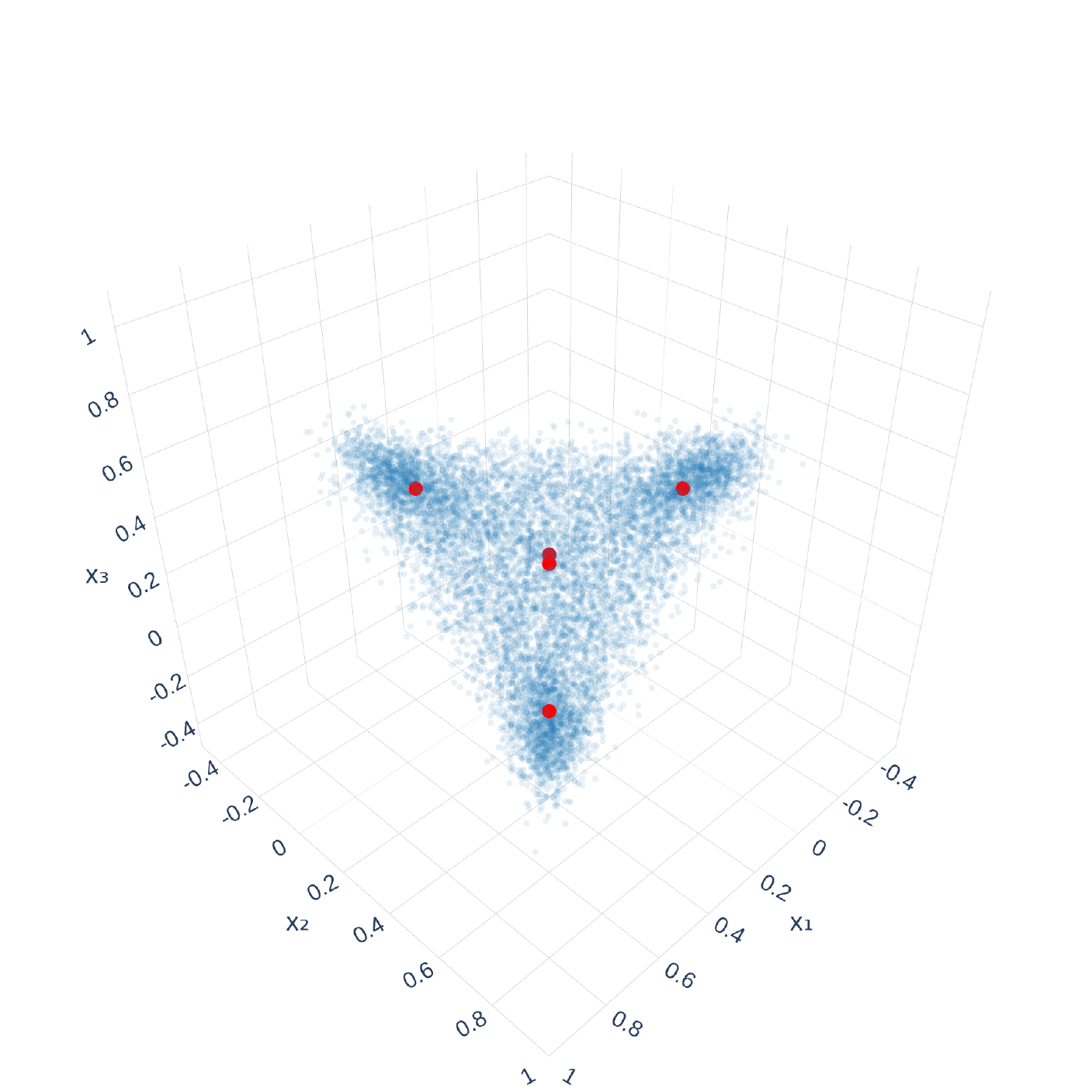}
        \caption{$t=0.8$}
    \end{subfigure}%
    ~ 
    \begin{subfigure}[t]{0.4\textwidth}
        \centering
        \includegraphics[width=\textwidth]{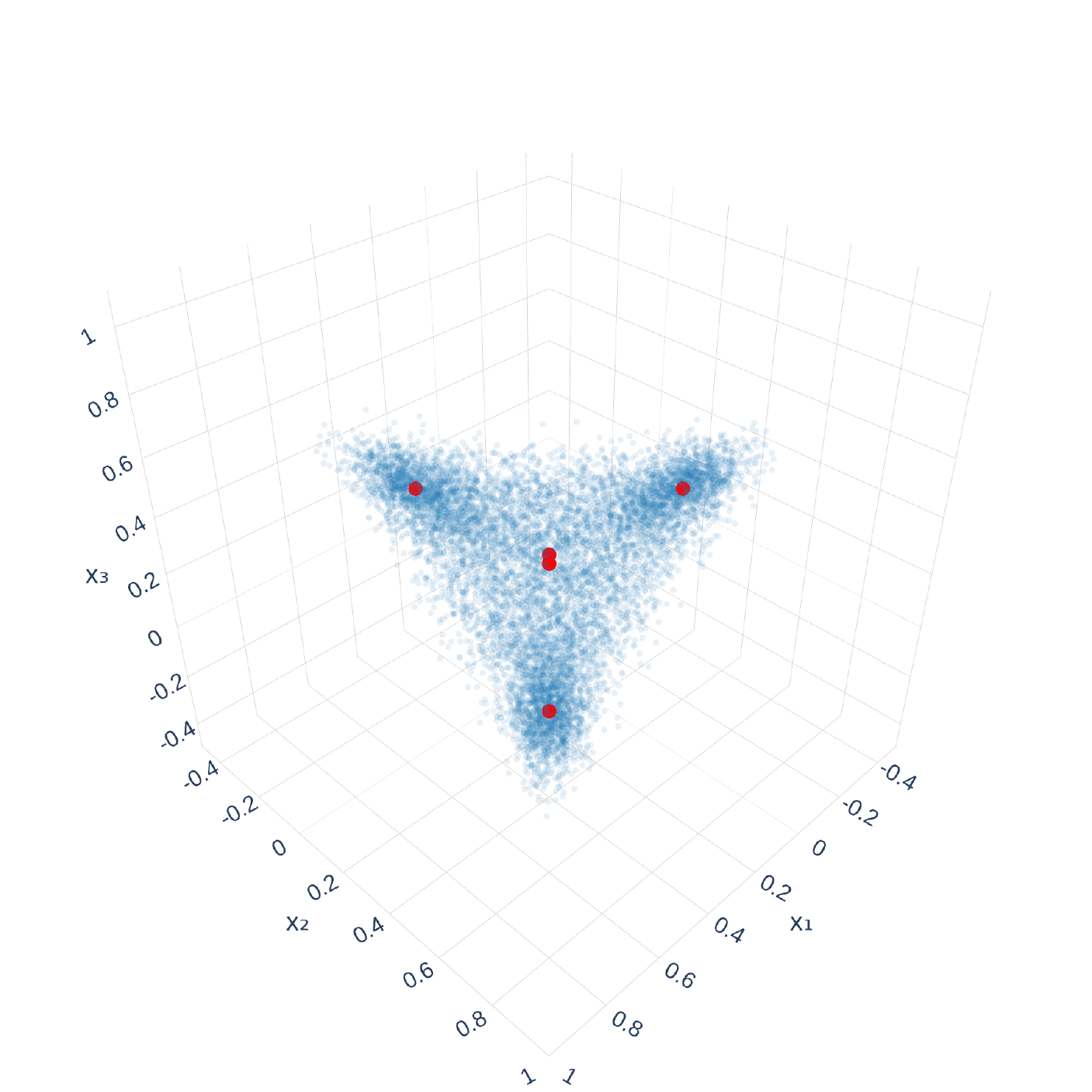}
        \caption{$t=1$}
    \end{subfigure}
    \caption{Plots of sample points of 3D non-potential game computed with best response dynamics from $t=0$ to $t=1$. The initial samples are drawn from a uniform distribution in $[0,1]^3$. The coefficient for diffusion term is $\sigma_1 = \sigma_2 = \sigma_3 = 0.1$. The red points mark the known Nash equilibria $(0,0,0)$, $\left( \frac{3}{8}, \frac{3}{8}, \frac{3}{8} \right)$, $(0.5, 0.5, 0)$, $(0.5, 0, 0.5)$ and $(0, 0.5, 0.5)$.}
    \label{fig: 3d_nonpotential_brd_uniform}
\end{figure}

\begin{figure}[!th]
    \centering
    \vspace{-1em}
    \begin{subfigure}[t]{0.4\textwidth}
        \centering
        \includegraphics[width=\textwidth]{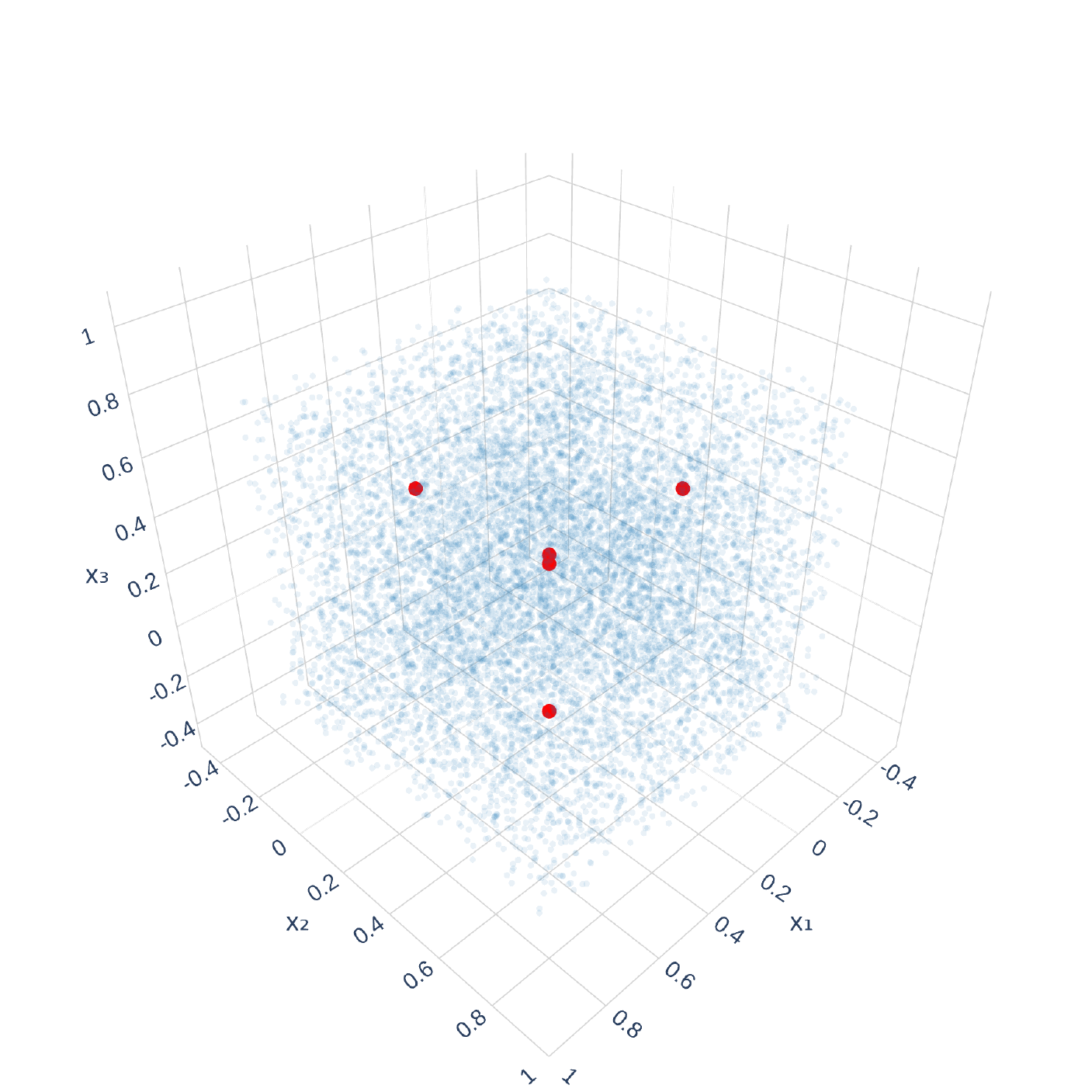}
        \caption{$t=0$}
    \end{subfigure}%
    \vspace{-0.3em}
    ~ 
    \begin{subfigure}[t]{0.4\textwidth}
        \centering
        \includegraphics[width=\textwidth]{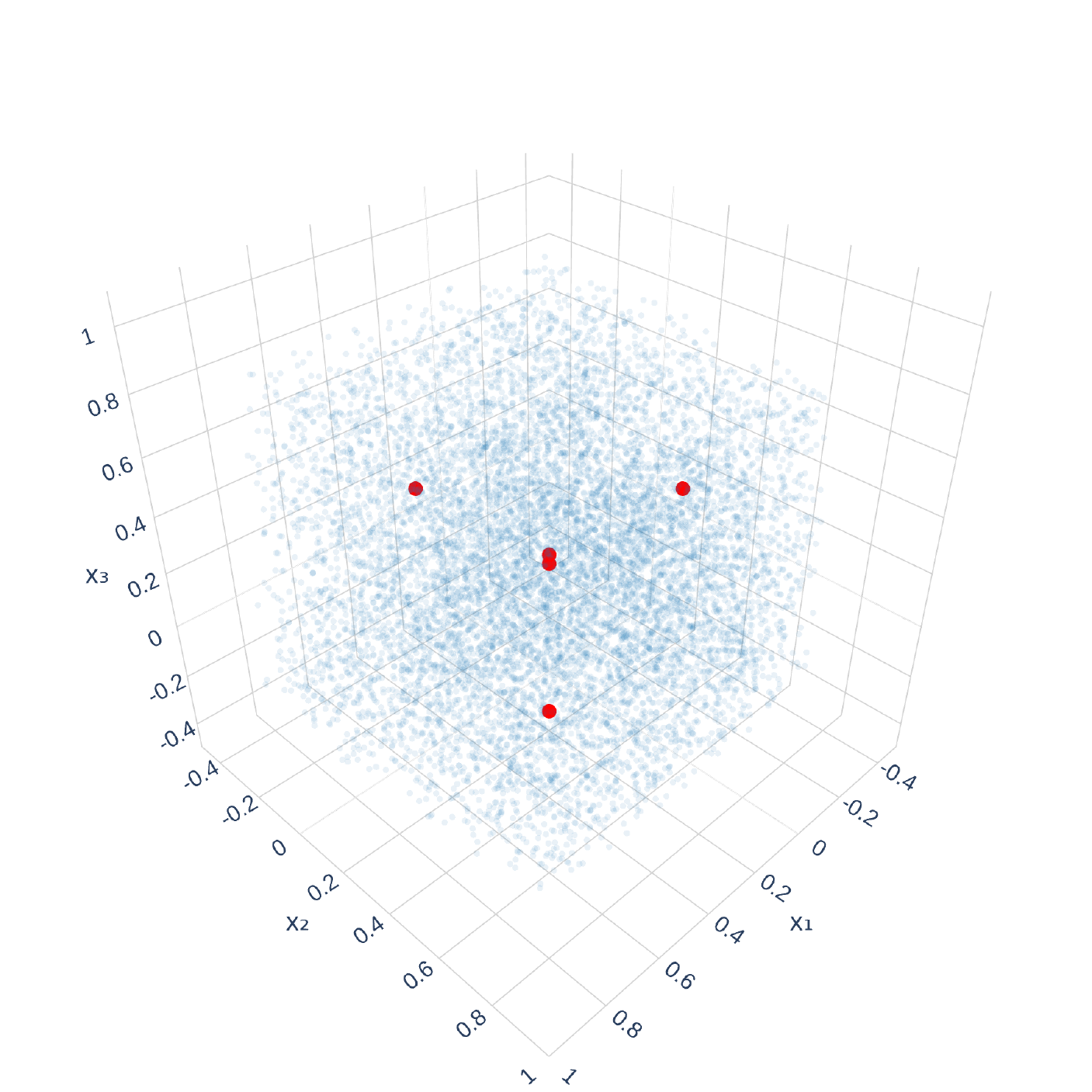}
        \caption{$t=0.2$}
    \end{subfigure}
    ~ 
    \begin{subfigure}[t]{0.4\textwidth}
        \centering
        \includegraphics[width=\textwidth]{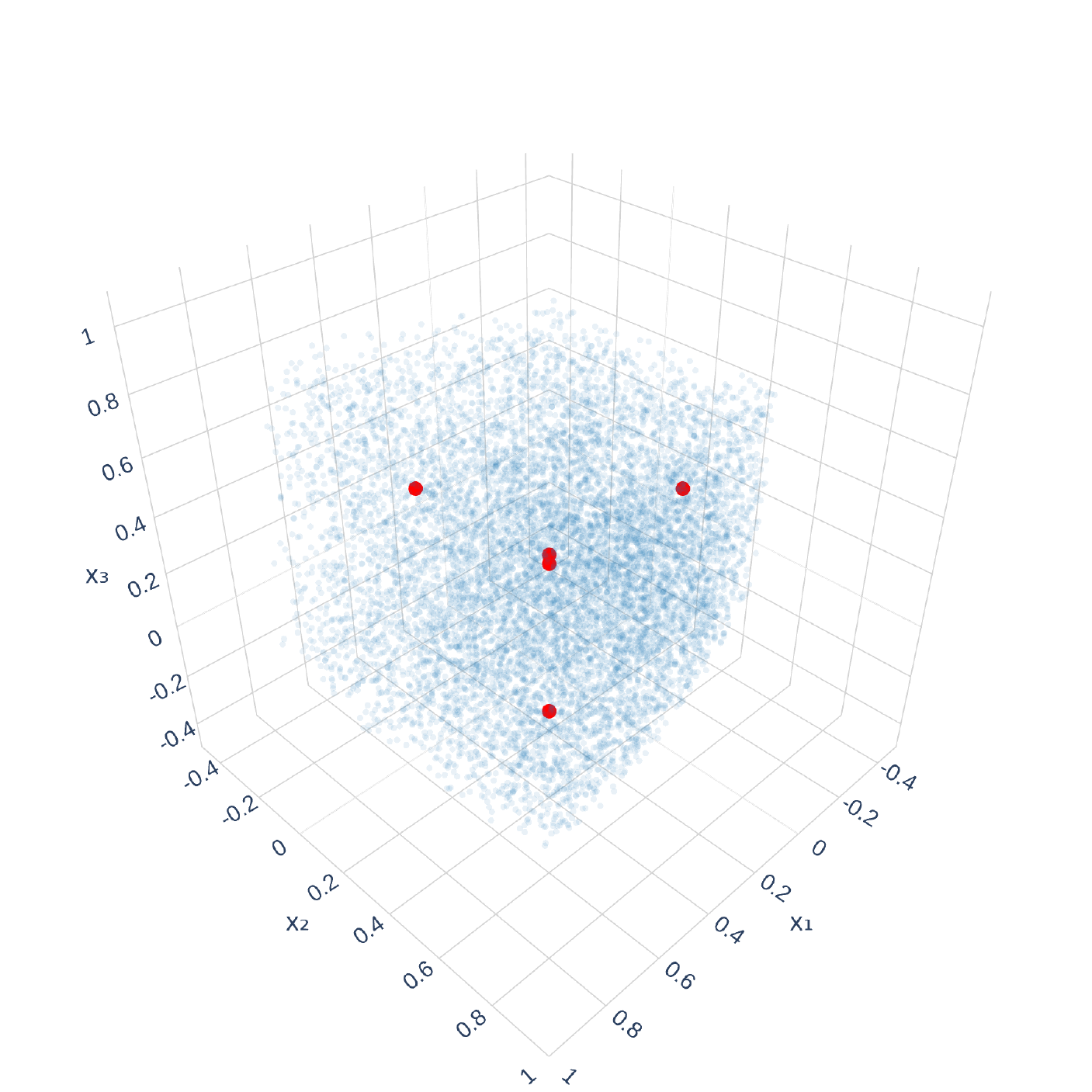}
        \caption{$t=0.4$}
    \end{subfigure}%
    \vspace{-0.3em}
    ~ 
    \begin{subfigure}[t]{0.4\textwidth}
        \centering
        \includegraphics[width=\textwidth]{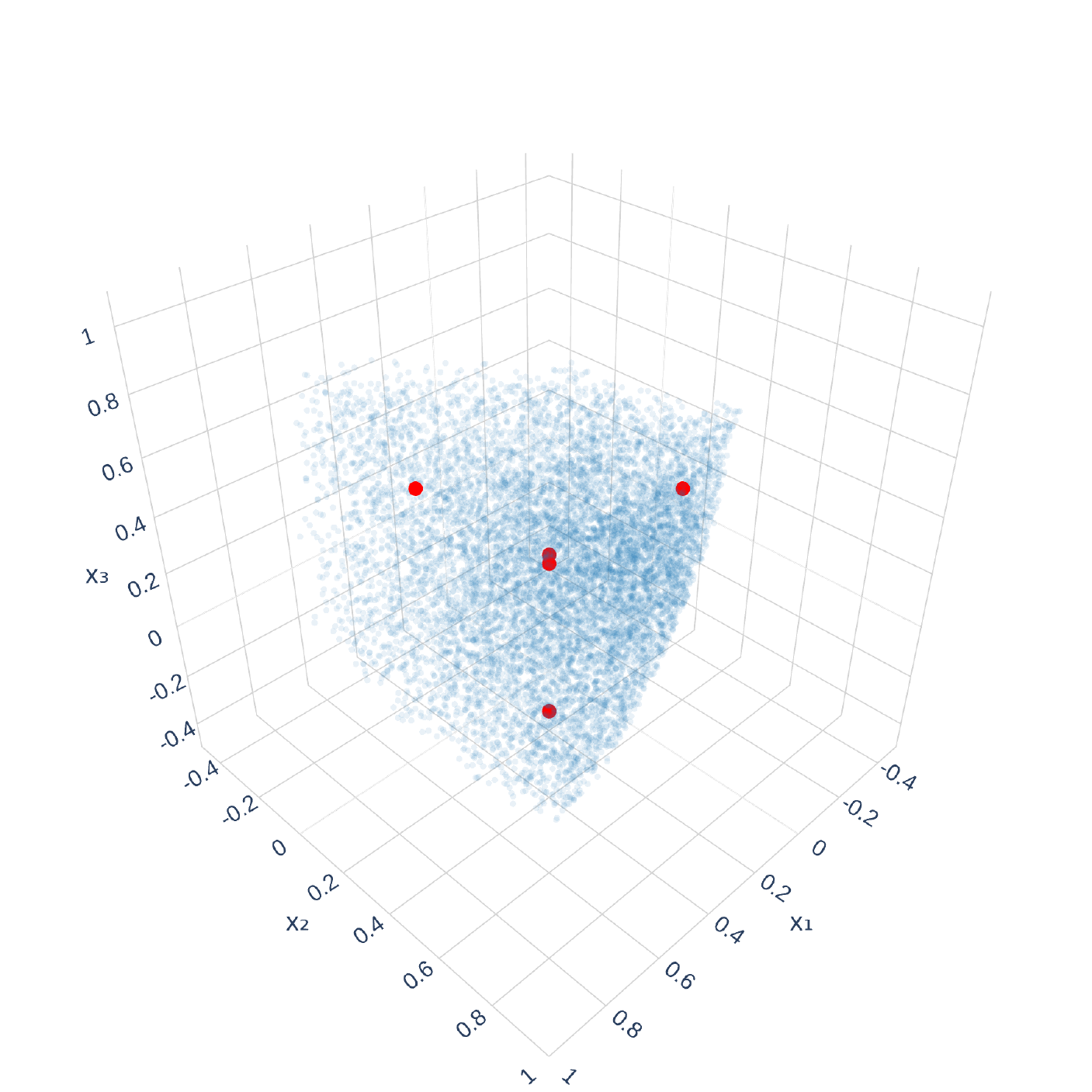}
        \caption{$t=0.6$}
    \end{subfigure}
    \vspace{-0.3em}
    ~ 
    \begin{subfigure}[t]{0.4\textwidth}
        \centering
        \includegraphics[width=\textwidth]{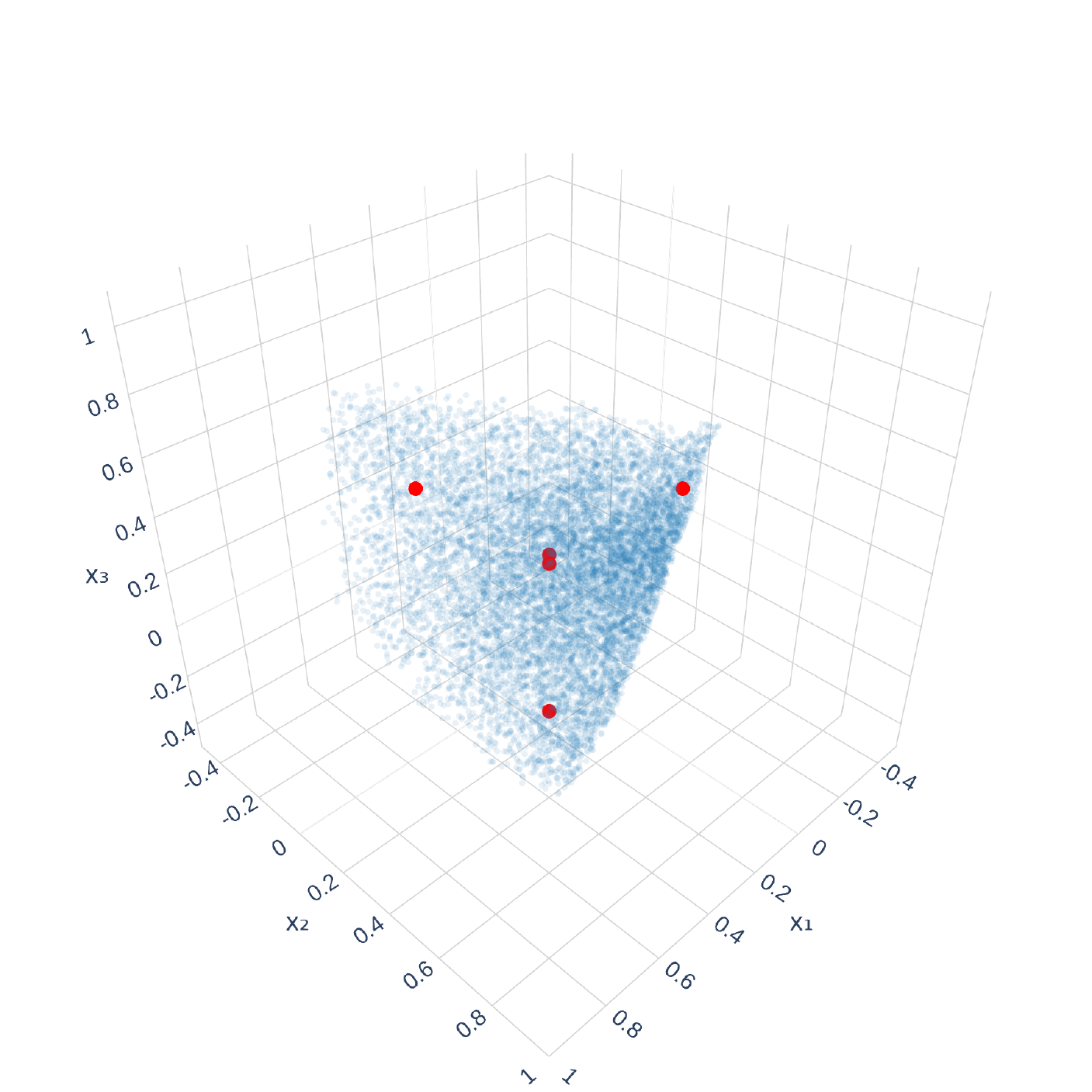}
        \caption{$t=0.8$}
    \end{subfigure}%
    \vspace{-0.3em}
    ~ 
    \begin{subfigure}[t]{0.4\textwidth}
        \centering
        \includegraphics[width=\textwidth]{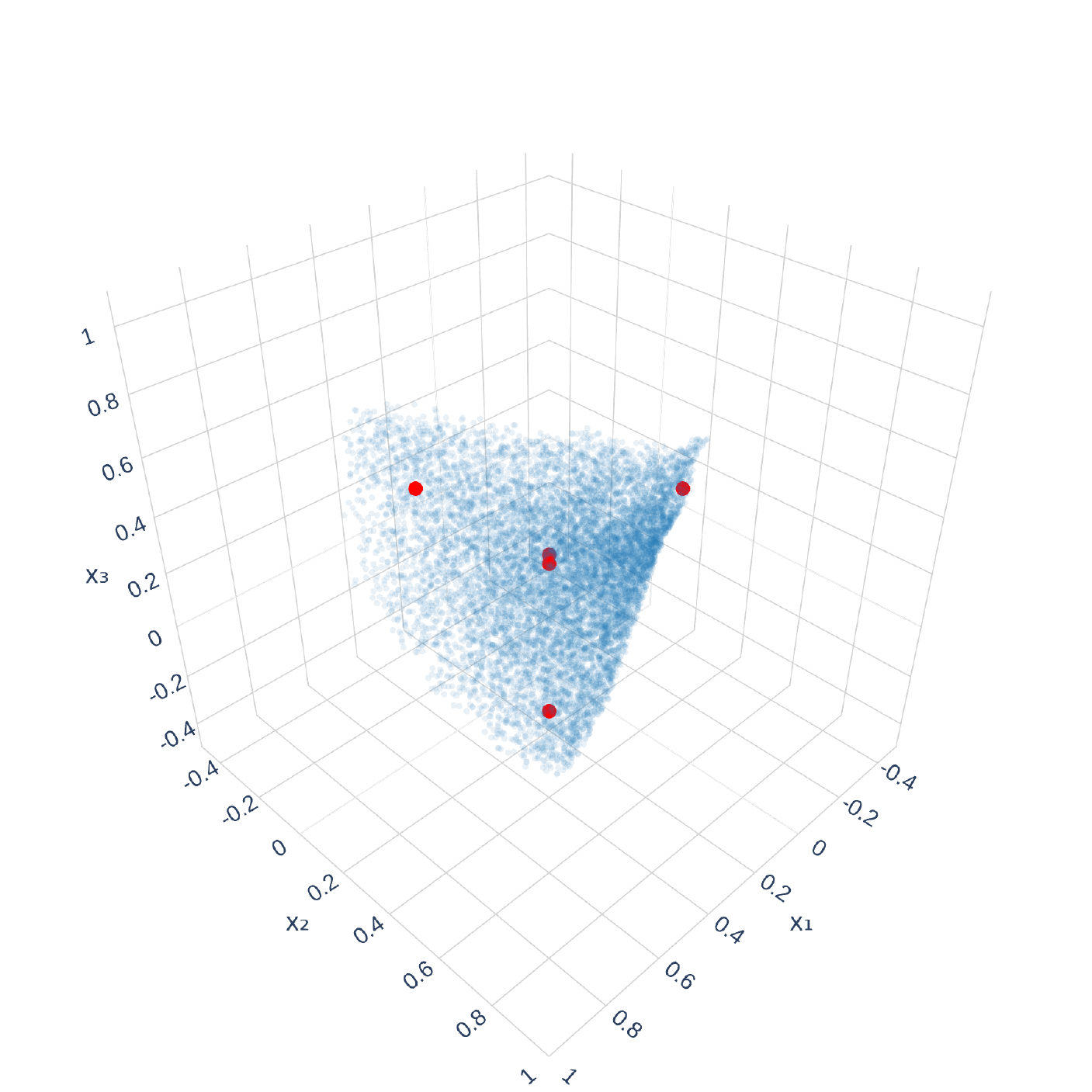}
        \caption{$t=1$}
    \end{subfigure}
    \caption{Plots of sample points of 3D non-potential game computed with pushforward map $T_{\theta(t)}$ from $t=0$ to $t=1$. The initial samples are drawn from a uniform distribution in $[0,1]^3$. The coefficient for diffusion term is $\sigma_1 = \sigma_2 = \sigma_3 = 0.1$. The red points mark the known Nash equilibria $(0,0,0)$, $\left( \frac{3}{8}, \frac{3}{8}, \frac{3}{8} \right)$, $(0.5, 0.5, 0)$, $(0.5, 0, 0.5)$ and $(0, 0.5, 0.5)$.}
    \label{fig: 3d_nonpotential_pushforward_uniform}
\end{figure}

\subsubsection{5D Non-potential Game}
In this subsection, we show an example of simulating the dynamics of a 5-dimensional Cournot game with the same parameters $a = 3$, $b = 1$, $d = 1$ and $\mu = 2$ to illustrate the application in solving high-dimensional games. As in the previous example, we exploit the symmetry of the payoff functions to identify seven Nash equilibria, $(0, 0, 0, 0, 0)$, $\left( \frac{7}{32}, \frac{7}{32}, \frac{7}{32}, \frac{7}{32}, \frac{7}{32} \right)$, $\left( 0, \frac{5}{18}, \frac{5}{18}, \frac{5}{18}, \frac{5}{18} \right)$, $\left( \frac{5}{18}, 0, \frac{5}{18}, \frac{5}{18}, \frac{5}{18} \right)$, $\left(\frac{5}{18}, \frac{5}{18}, 0, \frac{5}{18}, \frac{5}{18}  \right)$, $\left( \frac{5}{18}, \frac{5}{18}, \frac{5}{18}, 0, \frac{5}{18} \right)$, and $\left( \frac{5}{18}, \frac{5}{18}, \frac{5}{18}, \frac{5}{18}, 0 \right)$. By stability analysis, we know that none of the Nash equilibria is stable. While additional equilibria may exist, our stability analysis is limited to those currently identified. Below in \cref{fig: 5d_nonpotential_brd_uniform} and \cref{fig: 5d_nonpotential_pushforward_uniform}, we show the sample plots computed with best response dynamics and pushforward map, respectively. The results obtained from the pushforward map align with those from the best response dynamics. No evidence of stable Nash equilibria is observed in either computation.

\begin{figure}[!th]
    \centering
    \begin{subfigure}[t]{0.4\textwidth}
        \centering
        \includegraphics[width=\textwidth]{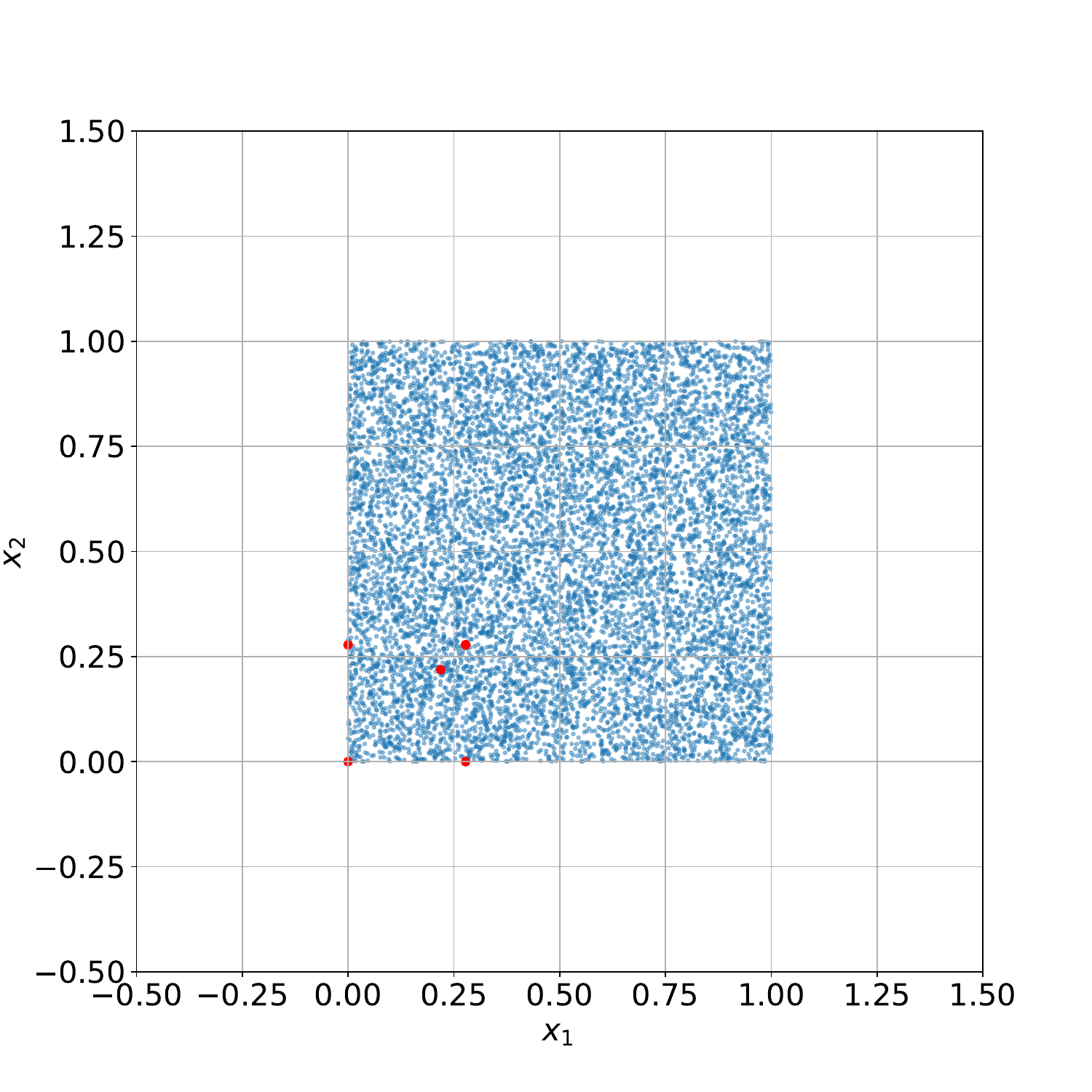}
        \caption{$t=0$}
    \end{subfigure}%
    ~ 
    \begin{subfigure}[t]{0.4\textwidth}
        \centering
        \includegraphics[width=\textwidth]{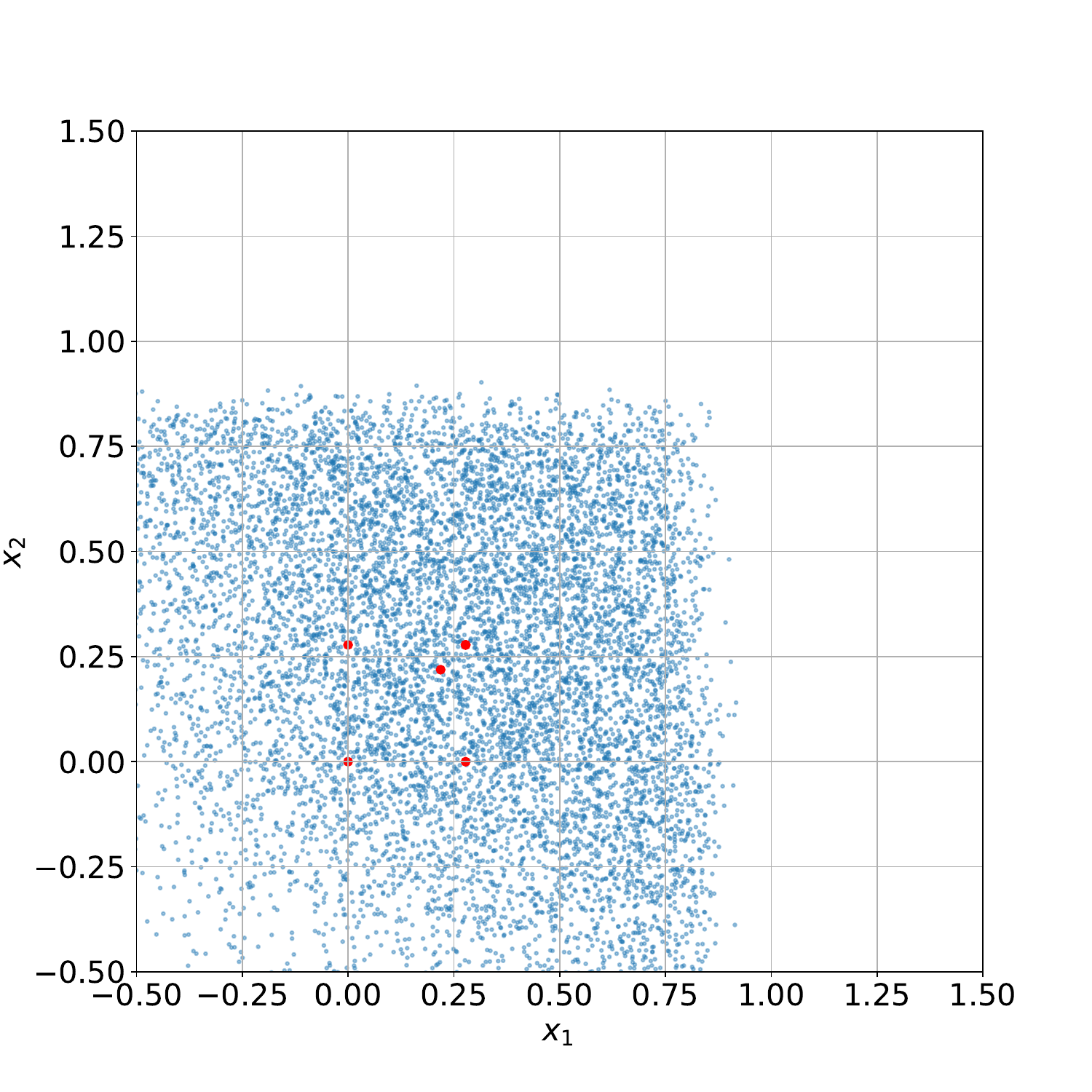}
        \caption{$t=0.12$}
    \end{subfigure}
    ~ 
    \begin{subfigure}[t]{0.4\textwidth}
        \centering
        \includegraphics[width=\textwidth]{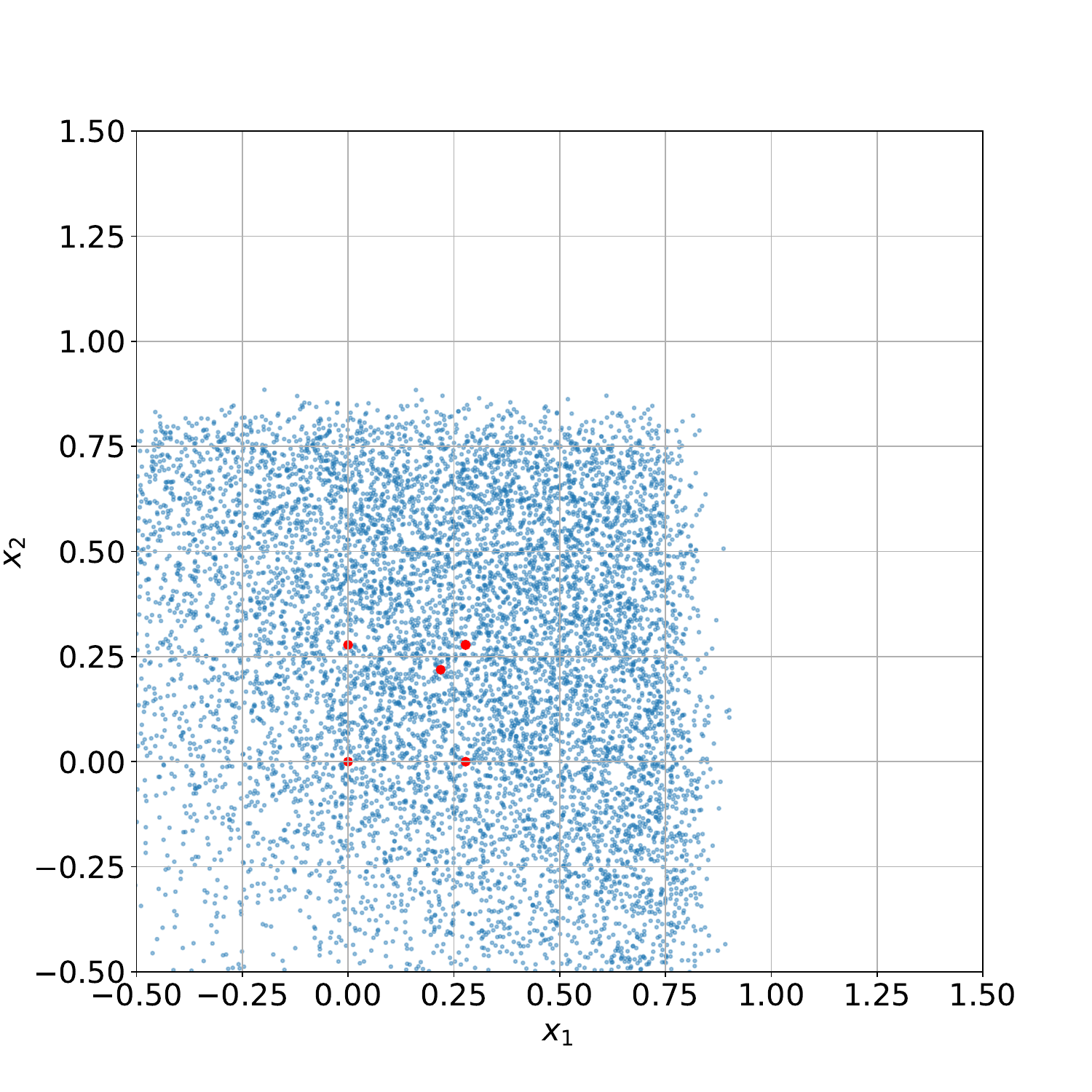}
        \caption{$t=0.14$}
    \end{subfigure}%
    ~ 
    \begin{subfigure}[t]{0.4\textwidth}
        \centering
        \includegraphics[width=\textwidth]{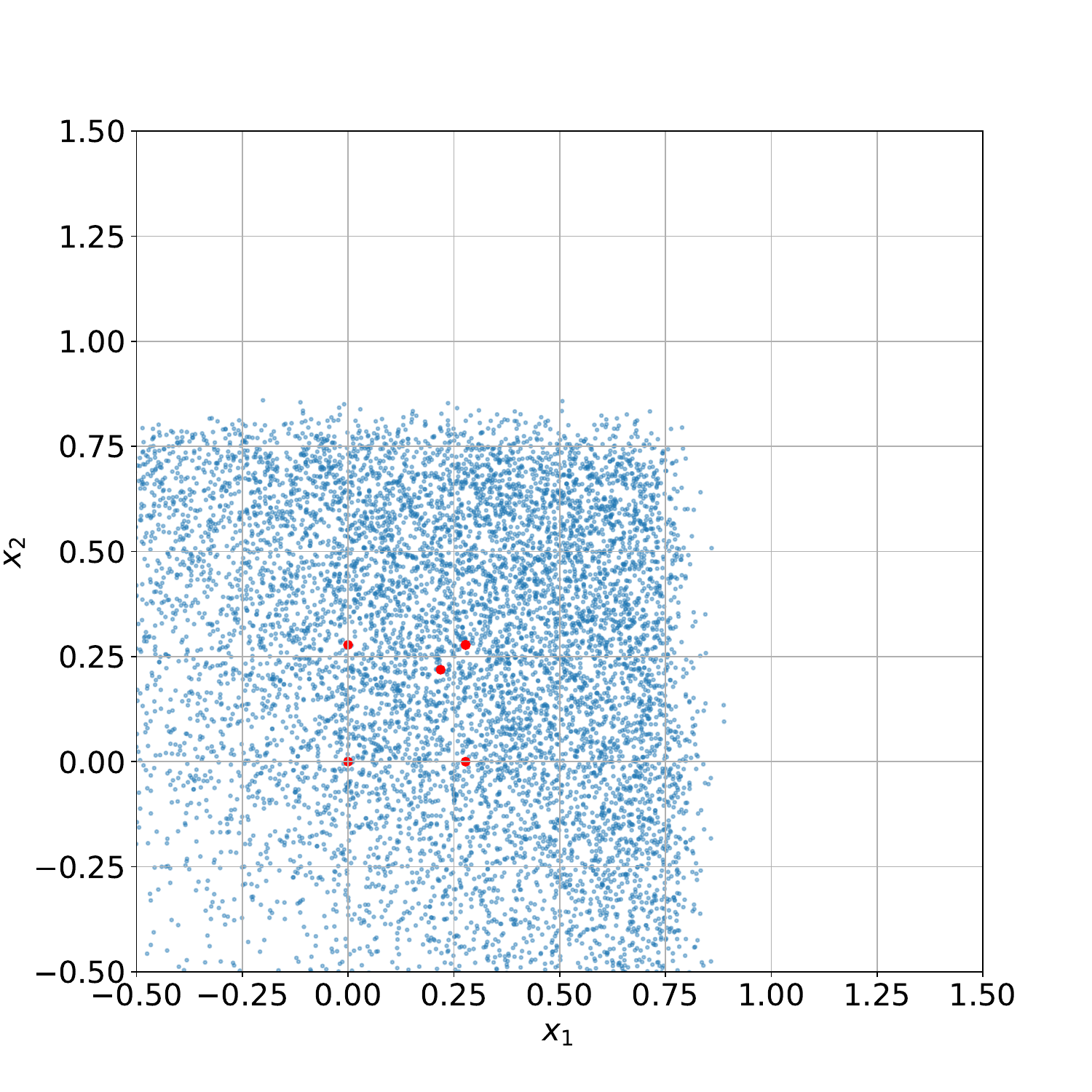}
        \caption{$t=0.16$}
    \end{subfigure}
    \caption{Plots of sample points of 5D non-potential game computed with best response dynamics from $t=0$ to $t=0.16$ projected to $x_1-x_2$ plane. The initial samples are drawn from a uniform distribution in $[0,1]^5$. The coefficient for diffusion term is $\epsilon_1 = \epsilon_2 = 0.1$. The red points mark the known Nash equilibria.}
    \label{fig: 5d_nonpotential_brd_uniform}
\end{figure}

\begin{figure}[!th]
    \centering
    \begin{subfigure}[t]{0.4\textwidth}
        \centering
        \includegraphics[width=\textwidth]{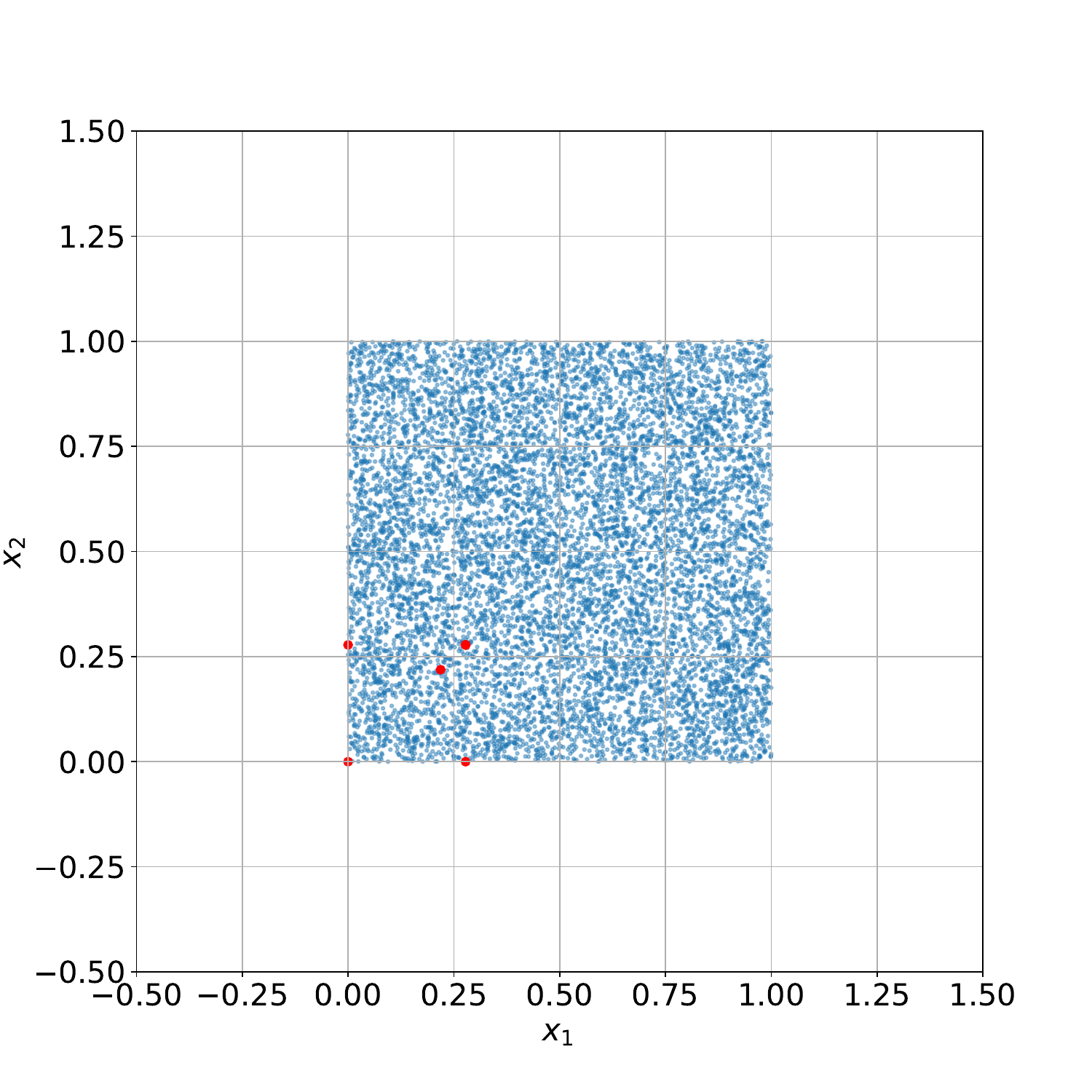}
        \caption{$t=0$}
    \end{subfigure}%
    ~ 
    \begin{subfigure}[t]{0.4\textwidth}
        \centering
        \includegraphics[width=\textwidth]{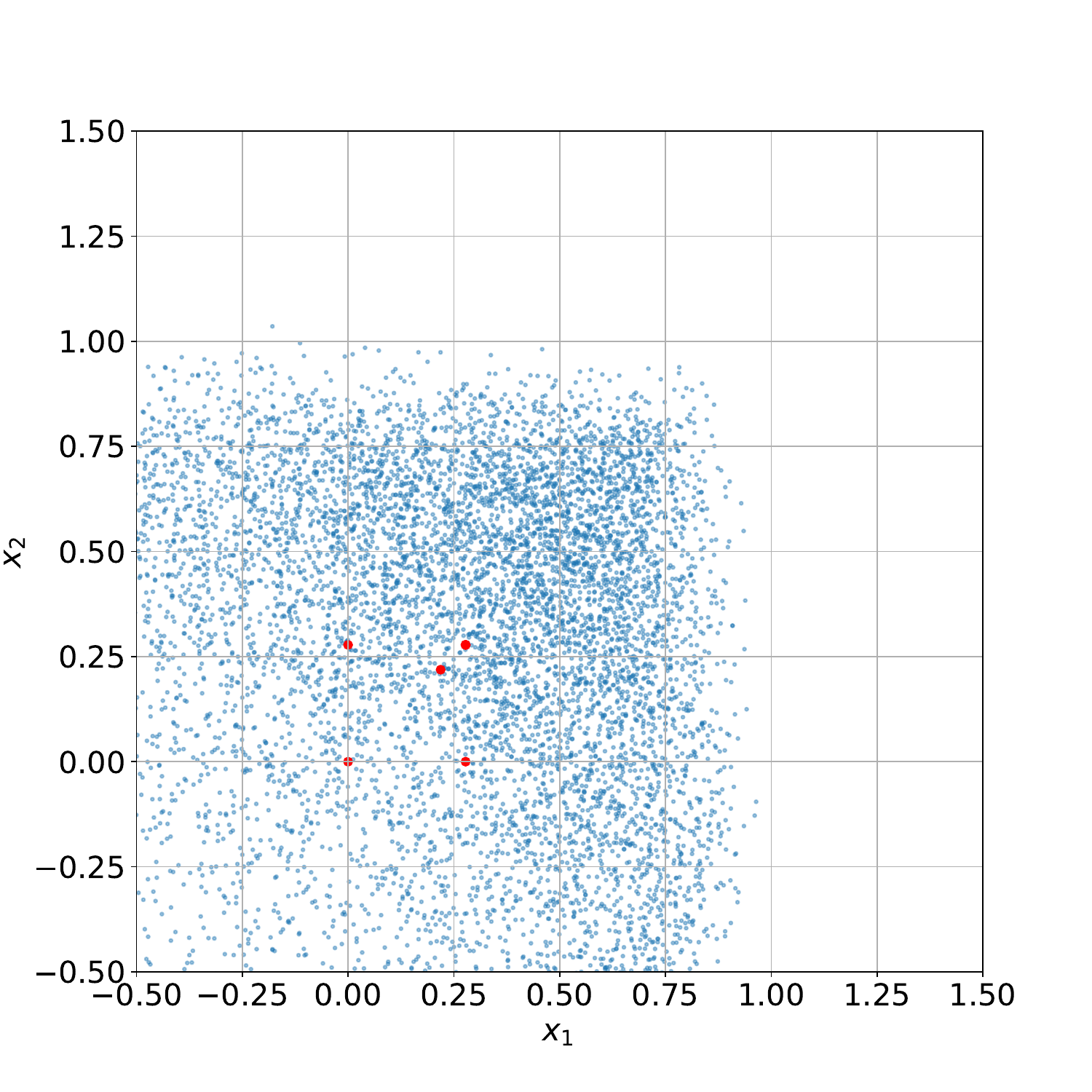}
        \caption{$t=0.12$}
    \end{subfigure}
    ~ 
    \begin{subfigure}[t]{0.4\textwidth}
        \centering
        \includegraphics[width=\textwidth]{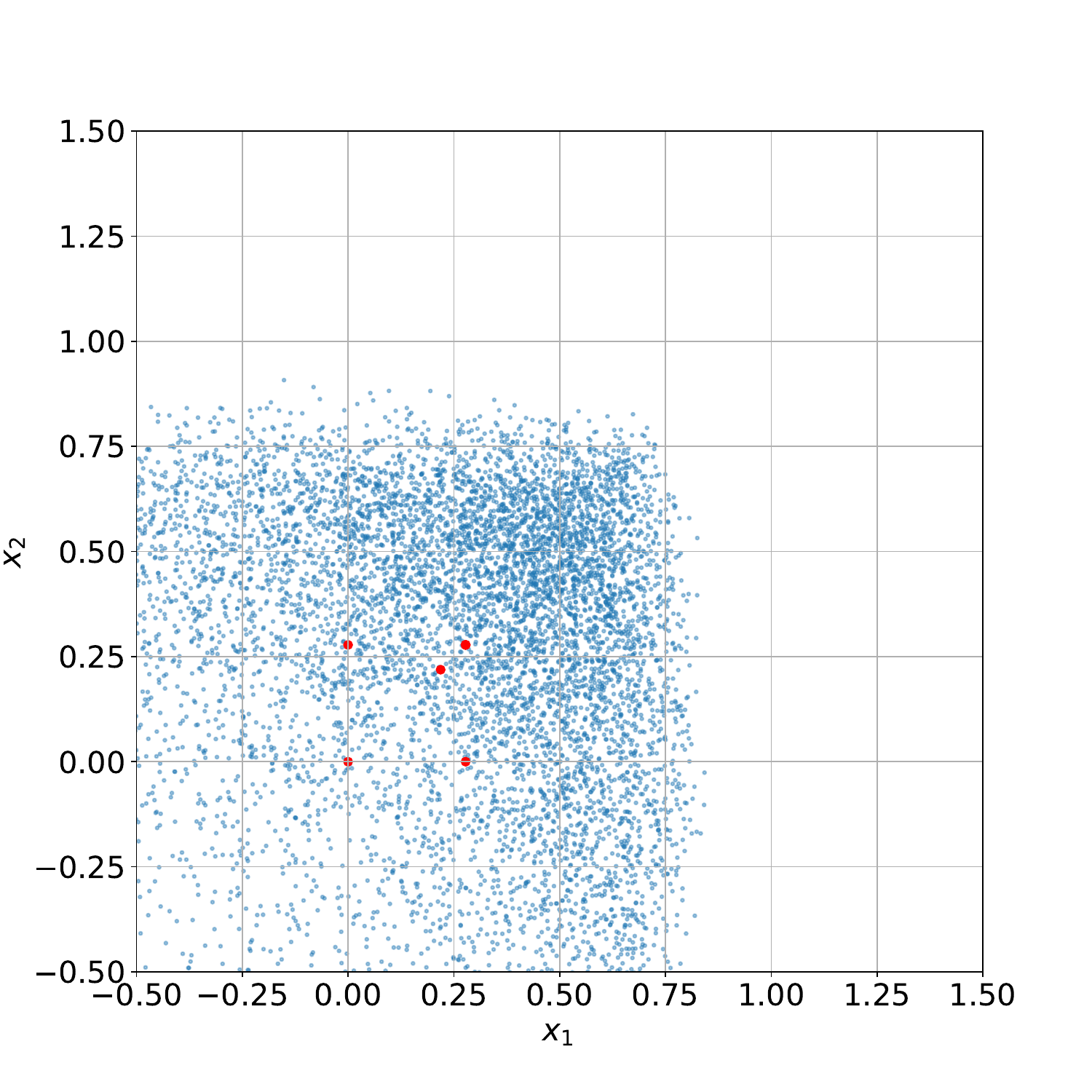}
        \caption{$t=0.14$}
    \end{subfigure}%
    ~ 
    \begin{subfigure}[t]{0.4\textwidth}
        \centering
        \includegraphics[width=\textwidth]{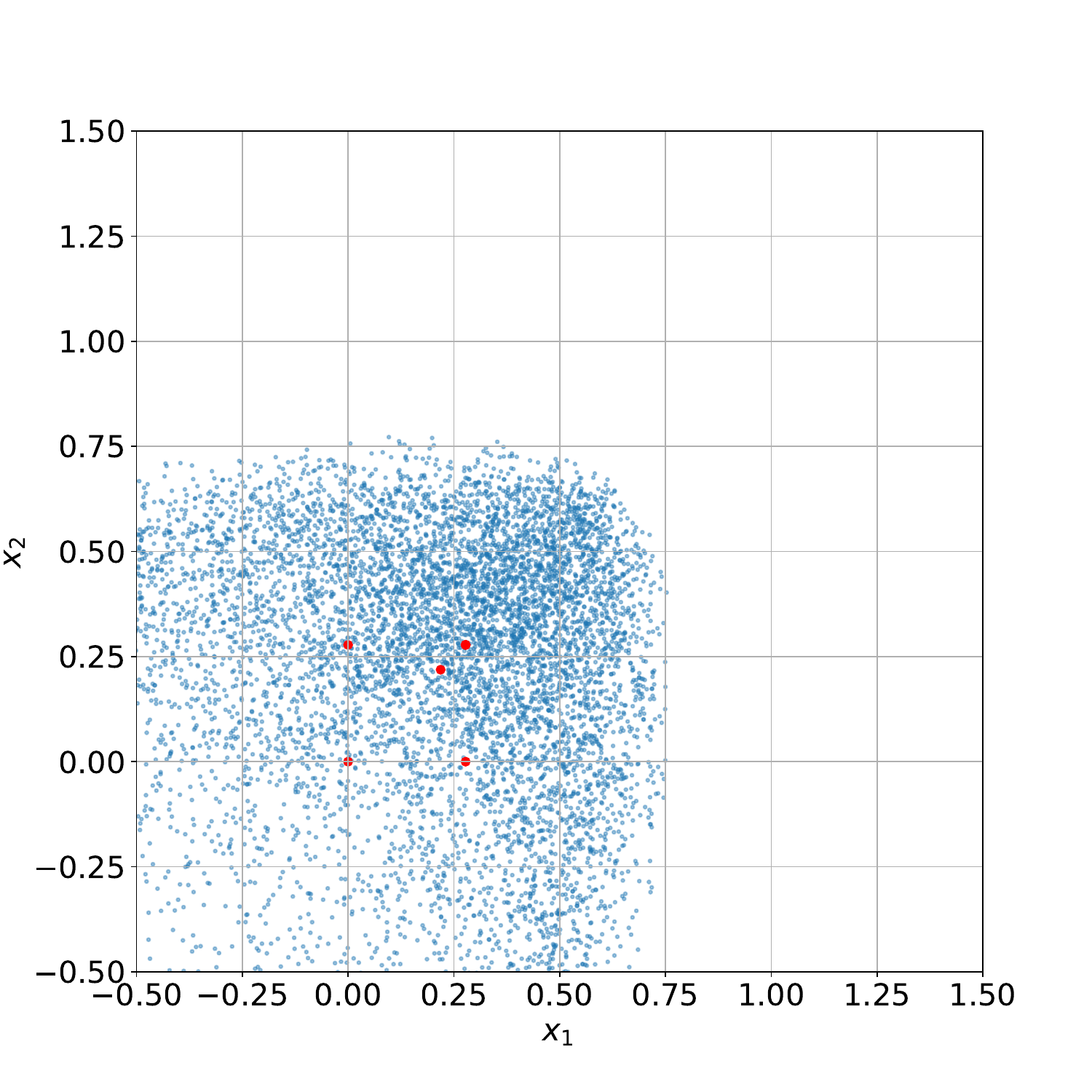}
        \caption{$t=0.16$}
    \end{subfigure}
    \caption{Plots of sample points of 5D non-potential game computed with pushforward map $T_{\theta(t)}$ from $t=0$ to $t=0.16$ projected to $x_1-x_2$ plane. The initial samples are drawn from a uniform distribution in $[0,1]^5$. The coefficient for diffusion term is $\epsilon_1 = \epsilon_2 = 0.1$. The red points mark the known Nash equilibria.}
    \label{fig: 5d_nonpotential_pushforward_uniform}
\end{figure}

\section{Discussion}
\label{sec: discussion}
We have proposed a new approach for computing game dynamics by focusing on how the probability of selecting different strategies evolves over time. The probability densities are represented by parameterized pushforward maps, which enables us to reduce the infinite-dimensional Fokker-Planck equation to a finite-dimensional ODE system on the parameter space. A key contribution of this work is a novel derivation of the parameterized equation that does not rely on a gradient-flow structure, making the computational framework applicable to general, non-gradient systems. Moreover, because the parameterized framework keeps the geometric structure of the parameter space, we're able to obtain an error estimate in Wasserstein-2 metric, in a manner analogous to PWGF. 

The proposed parameterized game dynamics provide a scalable and flexible computational tool, particularly well-suited for high-dimensional continuous strategy spaces where traditional grid-based methods are infeasible. The accuracy and computational speed of the algorithm depend on the performance of the linear system solver used to update the neural network parameters. In this work, Neural ODEs are employed to construct the parameterized pushforward maps, enabling simultaneous generation of samples and probability densities. Alternative neural network architectures and linear system solvers could be considered in future work to improve accuracy.

\backmatter





\bmhead{Acknowledgements}
This research is partially supported by NSF grants DMS-2152960, DMS-2307465, DMS-2307466, and ORN grant N00014-21-1-2891. All authors made equal contributions.



\section*{Declarations}


\begin{itemize}
\item Conflict of interest/Competing interests: No conflict of interest between the authors. 
\item Data availability: Data will be made available on reasonable request. 
\item Code availability: Code will be made available on reasonable request.
\item Author contribution: Equal contribution for both authors. 
\end{itemize}

\bibliography{references}

@article{Lafferty,
  title = {The {{Density Manifold}} and {{Configuration Space Quantization}}},
  volume = {305},
  number = {2},
  journal = {Transactions of the American Mathematical Society},
  author = {Lafferty, John D.},
  year = {1988},
  pages = {699-741}
}

@article{otto-PME,
  title = {The {{Geometry}} of {{Dissipative Evolution Equations}}: {{The Porous Medium Equation}}},
  volume = {26},
  number = {1-2},
  journal = {Communications in Partial Differential Equations},
  author = {Otto, Felix},
  year = {2001},
  pages = {101-174}
}

@article{kopel1996simple,
  title={Simple and complex adjustment dynamics in Cournot duopoly models},
  author={Kopel, Michael},
  journal={Chaos, Solitons \& Fractals},
  volume={7},
  number={12},
  pages={2031--2048},
  year={1996},
  publisher={Elsevier}
}

@article{jin2025parameterized,
  title={Parameterized Wasserstein gradient flow},
  author={Jin, Yijie and Liu, Shu and Wu, Hao and Ye, Xiaojing and Zhou, Haomin},
  journal={Journal of Computational Physics},
  volume={524},
  pages={113660},
  year={2025},
  publisher={Elsevier}
}

@book{hofbauer1998evolutionary,
  title={Evolutionary games and population dynamics},
  author={Hofbauer, Josef and Sigmund, Karl},
  year={1998},
  publisher={Cambridge university press},
  address   = {Cambridge, UK}
}

@incollection{smith1982evolution,
  title={Evolution and the Theory of Games},
  author={Smith, John Maynard},
  booktitle={Did Darwin get it right? Essays on games, sex and evolution},
  pages={202--215},
  year={1982},
  publisher={Springer},
  address   = {Boston, MA}
}

@article{nash1950equilibrium,
  title={Equilibrium points in n-person games},
  author={Nash Jr, John F},
  journal={Proceedings of the national academy of sciences},
  volume={36},
  number={1},
  pages={48--49},
  year={1950},
  publisher={national academy of sciences}
}

@book{cournot1838recherches,
  title={Recherches sur les principes math{\'e}matiques de la th{\'e}orie des richesses},
  author={Cournot, Antoine Augustin},
  year={1838},
  publisher={L. Hachette},
  address  = {Paris},
}

@article{monderer1996potential,
  title={Potential games},
  author={Monderer, Dov and Shapley, Lloyd S},
  journal={Games and economic behavior},
  volume={14},
  number={1},
  pages={124--143},
  year={1996},
  publisher={Elsevier}
}

@article{rosenthal1973class,
  title={A class of games possessing pure-strategy Nash equilibria},
  author={Rosenthal, Robert W},
  journal={International Journal of Game Theory},
  volume={2},
  pages={65--67},
  year={1973},
  publisher={Physica-Verlag}
}

@incollection{von2007theory,
  title={Theory of games and economic behavior: 60th anniversary commemorative edition},
  author={Von Neumann, John and Morgenstern, Oskar},
  booktitle={Theory of games and economic behavior},
  year={2007},
  publisher={Princeton university press},
  address   = {Princeton, NJ}
}

@article{smith1972game,
  title={Game theory and the evolution of fighting},
  author={Smith, J Maynard},
  journal={On evolution},
  pages={8--28},
  year={1972},
  publisher={Edinburgh University Press Edinburgh, UK}
}

@article{smith1973logic,
  title={The logic of animal conflict},
  author={Smith, J Maynard and Price, George R},
  journal={Nature},
  volume={246},
  number={5427},
  pages={15--18},
  year={1973},
  publisher={Nature Publishing Group UK London},
}

@article{brown1951iterative,
  title={Iterative solution of games by fictitious play},
  author={Brown, George W},
  journal={Act. Anal. Prod Allocation},
  volume={13},
  number={1},
  pages={374},
  year={1951}
}

@book{hargreaves2004game,
  title={Game theory: A critical introduction},
  author={Hargreaves-Heap, Shaun and Varoufakis, Yanis},
  year={2004},
  publisher={Routledge},
  address   = {London and New York}
}

@book{fudenberg1998theory,
  title={The theory of learning in games},
  author={Fudenberg, Drew and Levine, David K},
  volume={2},
  year={1998},
  publisher={MIT press},
  address   = {Cambridge, MA}
}

@book{sandholm2010population,
  title={Population games and evolutionary dynamics},
  author={Sandholm, William H},
  year={2010},
  publisher={MIT press},
  address = {Cambridge, MA}
}

@article{monderer1996fictitious,
  title={Fictitious play property for games with identical interests},
  author={Monderer, Dov and Shapley, Lloyd S},
  journal={Journal of economic theory},
  volume={68},
  number={1},
  pages={258--265},
  year={1996},
  publisher={Elsevier}
}

@article{wu2025parameterized,
  title={Parameterized wasserstein hamiltonian flow},
  author={Wu, Hao and Liu, Shu and Ye, Xiaojing and Zhou, Haomin},
  journal={SIAM Journal on Numerical Analysis},
  volume={63},
  number={1},
  pages={360--395},
  year={2025},
  publisher={SIAM}
}

@article{hernandez2017survey,
  title={A survey of learning in multiagent environments: Dealing with non-stationarity},
  author={Hernandez-Leal, Pablo and Kaisers, Michael and Baarslag, Tim and De Cote, Enrique Munoz},
  journal={arXiv preprint arXiv:1707.09183},
  year={2017}
}

@article{schuster1983replicator,
  title={Replicator dynamics},
  author={Schuster, Peter and Sigmund, Karl},
  journal={Journal of theoretical biology},
  volume={100},
  number={3},
  pages={533--538},
  year={1983},
  publisher={Elsevier}
}

@inproceedings{chen2024teng,
  title={TENG: Time-Evolving Natural Gradient for Solving PDEs With Deep Neural Nets Toward Machine Precision},
  author={Chen, Zhuo and Mccarran, Jacob and Vizcaino, Esteban and Soljacic, Marin and Luo, Di},
  booktitle={International Conference on Machine Learning},
  pages={7143--7162},
  year={2024},
  organization={PMLR}
}

@article{wu2025deep,
  title={Deep Tangent Bundle (DTB) method: a Deep Neural Network approach to compute solutions of PDES},
  author={Wu, Hao and Zhou, Haomin},
  journal={arXiv preprint arXiv:2509.00957},
  year={2025}
}

@article{blume1993statistical,
  title={The statistical mechanics of strategic interaction},
  author={Blume, Lawrence E},
  journal={Games and economic behavior},
  volume={5},
  number={3},
  pages={387--424},
  year={1993},
  publisher={Elsevier}
}

@article{hofbauer2002global,
  title={On the global convergence of stochastic fictitious play},
  author={Hofbauer, Josef and Sandholm, William H},
  journal={Econometrica},
  volume={70},
  number={6},
  pages={2265--2294},
  year={2002},
  publisher={Wiley Online Library}
}

@article{benaim1999mixed,
  title={Mixed equilibria and dynamical systems arising from fictitious play in perturbed games},
  author={Bena{\i}m, Michel and Hirsch, Morris W},
  journal={Games and Economic Behavior},
  volume={29},
  number={1-2},
  pages={36--72},
  year={1999},
  publisher={Elsevier}
}

@article{hofbauer2007evolution,
  title={Evolution in games with randomly disturbed payoffs},
  author={Hofbauer, Josef and Sandholm, William H},
  journal={Journal of economic theory},
  volume={132},
  number={1},
  pages={47--69},
  year={2007},
  publisher={Elsevier}
}

@incollection{benamou2017variational,
  title={Variational mean field games},
  author={Benamou, Jean-David and Carlier, Guillaume and Santambrogio, Filippo},
  booktitle={Active Particles, Volume 1: Advances in Theory, Models, and Applications},
  pages={141--171},
  year={2017},
  publisher={Springer}
}

@article{lasry2007mean,
  title={Mean field games},
  author={Lasry, Jean-Michel and Lions, Pierre-Louis},
  journal={Japanese journal of mathematics},
  volume={2},
  number={1},
  pages={229--260},
  year={2007},
  publisher={Springer}
}

@article{achdou2020mean,
  title={Mean field games and applications: Numerical aspects},
  author={Achdou, Yves and Lauri{\`e}re, Mathieu},
  journal={Mean Field Games: Cetraro, Italy 2019},
  pages={249--307},
  year={2020},
  publisher={Springer}
}

@article{matsui1992best,
  title={Best response dynamics and socially stable strategies},
  author={Matsui, Akihiko},
  journal={Journal of Economic Theory},
  volume={57},
  number={2},
  pages={343--362},
  year={1992},
  publisher={Elsevier}
}

@article{smith1984stability,
  title={The stability of a dynamic model of traffic assignment—an application of a method of Lyapunov},
  author={Smith, Michael J},
  journal={Transportation science},
  volume={18},
  number={3},
  pages={245--252},
  year={1984},
  publisher={INFORMS}
}

\end{document}